\documentclass[a4paper]{article}
\usepackage{amsthm,amssymb,amsmath,enumerate,graphicx,epsf,mathabx}
\pdfoutput=1
\newcommand{\COLORON}{1}
\newcommand{\NOTESON}{0}
\newcommand{\Debug}{0}
\usepackage{bbold}
\usepackage[usenames]{color} 
\usepackage{amsthm,amssymb,amsmath,bbm,enumerate,graphicx,epsf,stmaryrd,accents}
\usepackage[bookmarks, colorlinks=false, breaklinks=true]{hyperref} 

\usepackage{authblk}

\newcommand{\comment}[1]{}
\newcommand{\COMMENT}[1]{}

\definecolor{darkgray}{rgb}{0.3,0.3,0.3}
\newcommand{\defi}[1]{{\color{darkgray}\emph{#1}}}

\comment{
	\begin{lemma}\label{}	
\end{lemma}
\begin{proof}

\end{proof}

\begin{theorem}\label{}
\end{theorem} 
\begin{proof} 	

\end{proof}

}

\newtheorem{proposition}{Proposition}[section]
\newtheorem{definition}[proposition]{Definition}
\newtheorem{theorem}[proposition]{Theorem}
\newtheorem{corollary}[proposition]{Corollary}

\newtheorem{lemma}[proposition]{Lemma}
\newtheorem{observation}[proposition]{Observation}
\newtheorem{conjecture}{{Conjecture}}[section]

\newtheorem{problem}[conjecture]{{Problem}}

\newtheorem{examp}[proposition]{Example}

\theoremstyle{definition}
\newtheorem{remark}{Remark}[subsection]

\newcommand{\FIG}{0}

\ifnum \NOTESON = 1 \newcommand{\note}[1]{ 

\hspace*{-30pt}
	{\color{blue}  NOTE: \color{Turquoise}{\small  \tt \begin{minipage}[c]{1.1\textwidth}  #1 \end{minipage} \ignorespacesafterend }} 
	
	}
\else \newcommand{\note}[1]{} \fi

\newcommand{\afsubm}[1]{ \ifnum \Debug = 1 {\mymargin{#1}}
\fi} 

\ifnum \Debug = 1 
\else  \fi

\ifnum \FIG = 1 \newcommand{\fig}[1]{Figure ``{#1}''}
\else \newcommand{\fig}[1]{Figure~\ref{#1}} \fi

\ifnum \FIG = 1 
\else  \fi

\ifnum \Debug = 1 \usepackage[notref,notcite]{showkeys}
\fi

\ifnum \COLORON = 0 \renewcommand{\color}[1]{}
\fi

\newcommand{\N}{\ensuremath{\mathbb N}}

\newcommand{\ca}{\ensuremath{\mathcal A}}
\newcommand{\cb}{\ensuremath{\mathcal B}}
\newcommand{\cc}{\ensuremath{\mathcal C}}

\newcommand{\cf}{\ensuremath{\mathcal F}}
\newcommand{\cg}{\ensuremath{\mathcal G}}
\newcommand{\ch}{\ensuremath{\mathcal H}}
\newcommand{\ci}{\ensuremath{\mathcal I}}
\newcommand{\cj}{\ensuremath{\mathcal J}}

\newcommand{\cp}{\ensuremath{\mathcal P}}

\newcommand{\cs}{\ensuremath{\mathcal S}}
\newcommand{\cgr}{\ensuremath{\mathcal R}}
\newcommand{\ct}{\ensuremath{\mathcal T}}
\newcommand{\cu}{\ensuremath{\mathcal U}}

\newcommand{\cx}{\ensuremath{\mathcal X}}

\newcommand{\oo}{\ensuremath{\omega}}

\newcommand{\bet}{\ensuremath{\beta}}

\newcommand{\ccg}{\ensuremath{\mathcal C(G)}}

\newcommand{\sm}{\backslash}

\newcommand{\sydi}{\triangle}

\newcommand{\isom}{\cong}

\newcommand{\cls}[1]{\ensuremath{\overline{#1}}}

\makeatletter
\DeclareRobustCommand{\cev}[1]{%
  \mathpalette\do@cev{#1}%
}
\newcommand{\do@cev}[2]{%
  \fix@cev{#1}{+}%
  \reflectbox{$\m@th#1\vec{\reflectbox{$\fix@cev{#1}{-}\m@th#1#2\fix@cev{#1}{+}$}}$}%
  \fix@cev{#1}{-}%
}
\newcommand{\fix@cev}[2]{%
  \ifx#1\displaystyle
    \mkern#23mu
  \else
    \ifx#1\textstyle
      \mkern#23mu
    \else
      \ifx#1\scriptstyle
        \mkern#22mu
      \else
        \mkern#22mu
      \fi
    \fi
  \fi
}

\makeatother

\newcommand{\nin}{\ensuremath{{n\in\N}}}

\newcommand{\pth}[2]{\ensuremath{#1}\text{--}\ensuremath{#2}~path}

\newcommand{\seq}[1]{\ensuremath{(#1_n)_{n\in\N}}} 
\newcommand{\seqi}[1]{\ensuremath{(#1_i)_{i\in\N}}} 
 
\newcommand{\g}{\ensuremath{G\ }}
\newcommand{\G}{\ensuremath{G}}

\newcommand{\ceil}[1]{\ensuremath{\left\lceil #1 \right\rceil}}

\newcommand{\wqo}{well-quasi-ordered}
\newcommand{\wqoing}{well-quasi-ordering}
\newcommand{\bqo}{better-quasi-ordered}
\newcommand{\bqoing}{better-quasi-ordering}

\newcommand{\GMT}{Graph Minor Theorem}

\newcommand{\Lr}[1]{Lemma~\ref{#1}}
\newcommand{\Lrs}[1]{Lemmas~\ref{#1}}
\newcommand{\Tr}[1]{Theorem~\ref{#1}}
\newcommand{\Trs}[1]{Theorems~\ref{#1}}
\newcommand{\Sr}[1]{Section~\ref{#1}}
\newcommand{\Srs}[1]{Sections~\ref{#1}}
\newcommand{\Prr}[1]{Pro\-position~\ref{#1}}
\newcommand{\Prb}[1]{Problem~\ref{#1}}
\newcommand{\Cr}[1]{Corollary~\ref{#1}}
\newcommand{\Cnr}[1]{Con\-jecture~\ref{#1}}
\newcommand{\Or}[1]{Observation~\ref{#1}}

\newcommand{\Dr}[1]{De\-fi\-nition~\ref{#1}}

\newcommand{\Rr}[1]{Remark~\ref{#1}}
\newcommand{\impl}[2]{\ref{#1} $\to$ \ref{#2}}

\renewcommand{\iff}{if and only if}
\newcommand{\fe}{for every}
\newcommand{\Fe}{For every}

\newcommand{\st}{such that}

\newcommand{\ti}{there is}
\newcommand{\ta}{there are}

\newcommand{\obda}{without loss of generality}

\newcommand{\wrt}{with respect to}

\newcommand{\leth}{large enough that}

\newcommand{\labtequ}[2]{
 \begin{equation} \label{#1} 	\begin{minipage}[c]{0.9\textwidth}  #2 \end{minipage} \ignorespacesafterend \end{equation} }

\newcommand{\mymargin}[1]{
 \ifnum \Debug = 1
  \marginpar{%
    \begin{minipage}{\marginparwidth}\small%
      \begin{flushleft}%
        {\color{blue}#1}%
      \end{flushleft}%
   \end{minipage}%
  }%
 \fi
}%

\newcommand{\extras}[1]{
 \ifnum \Debug = 1
\section{Extras} #1
 \fi
}%

\newcommand{\mySection}[2]{}

\usepackage{mathtools,enumitem}
\usepackage[percent]{overpic}
\usepackage{csquotes}

\usepackage{authblk}
\usepackage[greek,english]{babel}  
\usepackage{hyperref}
\RequirePackage{latexsym} \RequirePackage{amsthm}
\RequirePackage{amsmath}

\usepackage{enumerate}
\RequirePackage{amssymb} \RequirePackage{makeidx}
\usepackage{dsfont}
\usepackage{mathrsfs}

\newcommand{\forb}[1]{\mathrm{Forb}(#1)}

\newcommand{\amal}[2]{\ensuremath{#1^{T_#2}}}
\newcommand{\pst}{\cp^*}
\newcommand{\hst}{\ensuremath{H^*_{\omega_1}}}
\newcommand{\Rank}{\ensuremath{\mathrm{Rank}}}
\newcommand{\QRank}{\ensuremath{Q\mathrm{Rank}}}
\newcommand{\ran}[1]{\ensuremath{\mathrm{Rank}_{#1}}}
\newcommand{\RM}{\ensuremath{\cgr^\bullet}}
\newcommand{\RMn}{\ensuremath{\cgr^{n\bullet}}}
\newcommand{\ranm}[1]{\ensuremath{\ran{#1}^\bullet}}
\newcommand{\ranmn}[1]{\ensuremath{\ran{#1}^{n \bullet}}}
\newcommand{\ranmN}[2]{\ensuremath{\ran{#1}^{#2 \bullet}}}
\newcommand{\ranmm}[1]{\ensuremath{\ran{#1}^{m \bullet}}}
\newcommand{\mbu}{\ensuremath{m^\bullet}}
\newcommand{\cgrc}{\ensuremath{\cgr_{\mathrm c}}}

\newcommand{\Sb}{\ensuremath{S^\bullet}}
\newcommand{\ccm}{\cc^\bullet}
\newcommand{\ccmf}{\cc^\bullet_\mathrm{fin}}

\newcommand{\cont}{\ensuremath{2^{\aleph_0}}}

\newcommand{\uf}{\ensuremath{UF}}

\newcommand{\minem}{minor embedding}
\newcommand{\mm}{\ensuremath{<_\bullet}}

\usepackage{authblk}

\begin{document}

\title{On better-quasi-ordering under graph minors}

\author{Agelos Georgakopoulos\thanks{Supported by  EPSRC grants EP/V048821/1 and EP/V009044/1.}}
\affil{  {Mathematics Institute}\\
 {University of Warwick}\\
  {CV4 7AL, UK}}

\date{\today}
\maketitle

\begin{abstract}
In the aftermath of the Robertson--Seymour \GMT, Thomas conjectured that the countable graphs are well-quasi-ordered under the minor relation. We prove that this conjecture, when restricted to  graphs with no infinite paths (rays), is equivalent to the statement that the finite graphs are \bqo, another well-known open problem. Even more, we prove that the latter implies that the countable rayless graphs are \bqo. 

We prove several other statements to be equivalent to  the above, one of which being that  the rayless countable graphs of rank $\alpha$ can be decomposed into exactly $\aleph_0$ minor-twin classes \fe\ ordinal $\alpha<\omega_1$. 

By restricting the latter statement to trees, and combining it with Nash-Williams' theorem that the infinite trees are \wqo, we deduce as a side result that  a minor-closed family of $\N$-labelled rayless forests is Borel ---in the  Tychonoff product topology--- \iff\  it does not contain all rayless forests.

As another side-result, we prove Seymour's self-minor conjecture for rayless graphs of any cardinality. 
\end{abstract}

{\bf{Keywords:}} graph minor, well-quasi-order, better-quasi-order, countable rayless graph, Borel set, self-minor conjecture. \smallskip

{\bf{MSC 2020 Classification:}} 05C83, 05C63, 06A07.
\maketitle

\section{Introduction}

The celebrated Robertson--Seymour \GMT\  \cite{GMI}--\cite{GMXX} 
states that the finite graphs are \wqo\ under the minor relation $<$. This paper focuses on two well-known problems that it left open: 

\begin{conjecture}[Folklore \cite{DiestelBook25,PeqTow}] \label{con bqo}
The finite graphs are better-quasi-ordered (BQO) under the minor relation.
\end{conjecture}

\begin{conjecture}[Thomas' conjecture \cite{ThoWel}] \label{con Tho}
The countable graphs are well-quasi-ordered (WQO) under the minor relation.
\end{conjecture}

While WQO is a simple notion, its strengthening BQO is much harder to define but has the benefit of being preserved by natural operations that occur in inductive arguments, making it an indispensable tool even when one is only interested to prove that an order is WQO; this idea has found applications in several areas including combinatorics  \cite{NWbqo,ThoBet}, order theory \cite{LavFra}, set theory \cite{MarWel} and topology  \cite{Carroy}. 
See \cite{PeqTow} for a survey of these notions.  According to Pequignot \cite{PeqTow}, the poset of finite graphs endowed with the minor relation is\\ 

\vspace*{-.2cm} 
{\it {``the only naturally occurring WQO which is not yet known to be BQO''.}}
\medskip

A graph is \defi{rayless}, if it does not contain an (1-way) infinite path. The first main result of this paper provides a strong connection between the above conjectures, and much more: 

\begin{theorem} \label{fin bqo}
The following statements are equivalent: 
\begin{enumerate}[itemsep=0.0cm]
\item \label{FB} The finite graphs are BQO;
\item \label{RW} The countable {rayless} graphs are WQO; 
\item \label{RB} The countable {rayless} graphs are BQO.
\end{enumerate}
\end{theorem}
Thus we have reduced the  better-quasi-ordering of the finite graphs to a comparatively simple statement about infinite graphs. 

\medskip
Commenting on Thomas' \Cnr{con Tho} and similar questions, Robertson, Seymour \& Thomas \cite{RoSeThoExc} wrote: 
\begin{displayquote}
{\it ``There is not much chance of proving these conjectures because they imply that the set of all finite graphs is `second-level better-quasi-ordered' by minor containment, which in itself seems to be a hopelessly difficult problem''.}
\end{displayquote}
The (easier) implication \impl{RW}{FB} of \Tr{fin bqo} provides a far reaching strengthening of this claim: it says that the words `second-level' can be deleted. Moreover, the implication \impl{RW}{RB} is a considerable strengthening. Interestingly, the proof of the latter implication passes via the implication \impl{RW}{FB}. Part of this proof applies to arbitrary quasi-orders (\Cr{cor gen}). Our most difficult result is the implication \impl{FB}{RW}.
\medskip

While the above quote seems as timely as ever, the results of this paper provide new tools for attacking Conjectures~\ref{con bqo} and~\ref{con Tho}, and point out interesting special cases. 
See \Sr{sec outl} for more. 

But let me first try to explain the above quote. What is meant by `second-level better-quasi-ordered' here is probably the following. Given two sets of graphs  $\cg,\cg'$, we write $\cg <_* \cg'$ if \fe\ $G\in \cg$ \ti\ $H\in \cg'$ \st\ $G<H$. Note that $<_*$ is a quasi-order on the powerset $\cp(\cf)$ of the class \cf\ of finite graphs. 

\begin{problem} \label{prob wqo}
Are the sets of finite graphs \wqo\ under $<_*$? 
\end{problem}
An equivalent formulation is whether the set of minor-closed families of finite graphs is \wqo\ by the inclusion relation. 

The connection between \Prb{prob wqo} and \Cnr{con bqo} is that one can take the definition of 
$<_*$ further, and extend it from an ordering of $\cp(\cf)$ into an ordering of $\cp(\cp(\cf))$, and iterate transfinitely in order to define better-quasi-ordering; see \Sr{sec bqo} for details. In particular, our \Tr{fin bqo} implies that if Thomas'  \Cnr{con Tho} has a positive answer, then so does \Prb{prob wqo}. 
\medskip

Motivated by \Tr{fin bqo}, we undertake a thorough study of Thomas' conjecture for rayless graphs, noting that some long-standing open problems for infinite graphs have been settled in the rayless case \cite{BBDSTwins,unfrayless,PitSteStr}. As a warm-up, we will prove Seymour's (unpublished) self-minor conjecture for rayless graphs of any cardinality, which to the best of my knowledge was open even for countable graphs: 

\begin{corollary} \label{cor Seym Intro}
\Fe\ infinite rayless graph $G$, there is a minor model of \g in itself which is not the identity. 
\end{corollary}
(Oporowski \cite{Opo90} has found uncountable counterexamples to Seymour's self-minor conjecture, which of course contain rays.) 

\medskip
A well-known equivalent way to define what it means for a quasi-order $(Q,\leq)$ to be  \wqo\ is to say that it has no infinite antichain and no infinite descending chain. None of these conditions implies the other in general, but we will show that in our setup they are in fact equivalent. Let \defi{$\cgr$} denote the class of countable rayless graphs.
Our second main result can be summarized as follows: 

\begin{theorem} \label{main Intro}
The following statements are equivalent: 
\begin{enumerate}[itemsep=0.0cm
,label=(\alph*)]
\item \label{I iir} $\cgr$ is \wqo;
 \item \label{I ivp} $\cgr$ has no infinite descending chain;    
\item \label{I i} $\cgr$ has no infinite antichain; 
\item \label{I twin} \fe\ ordinal $\alpha<\omega_1$, the number of minor-twin classes of countable rayless graphs of rank $\alpha$ is  $\aleph_0$.
\end{enumerate}
\end{theorem}

As a consequence, to prove that \cgr\ is \wqo, it would suffice to prove that the rayless graphs of rank $\alpha$ are \wqo\ \fe\ ordinal $\alpha<\omega_1$. But the most interesting aspect of \Tr{main Intro} is the equivalence of the \wqoing\ of \cgr\ to the cardinality condition \ref{I twin}, which I will now explain. We say that two graphs $G,H$ are \defi{minor-twins}, if both $G<H$ and $H<G$ hold. Each rayless graph can be assigned an ordinal number, called its \defi{rank}, by recursively decomposing \g into graphs of smaller ranks (\Sr{sec rank}) similarly to the definition of rank for Borel sets or Hausdorff rank for linear orders. It is known that every ordinal $\alpha$ smaller that the least uncountable ordinal $\omega_1$ is the rank of some countable graph \G. Assuming $\neg \mathrm{CH}$, i.e.\ that $\aleph_1< \cont$,  an alternative way to formulate \ref{I twin} is to say that $\cgr$ consists of exactly $\aleph_1$ minor-twin equivalence classes. 

We will prove a refinement of \Tr{main Intro} (\Tr{tfae}) whereby \cgr\ is replaced by the subclass of graphs of (up to) a given rank. The rank 0 case of this refinement provides a statement equivalent to the \GMT\ (\Cr{UF ctble}).

This refinement is also needed for the proof of the implication \impl{FB}{RW} of \Tr{fin bqo}. Moreover, it adds several further equivalent statements to those of \Trs{fin bqo} and~\ref{main Intro} (\Tr{main}). \Tr{tfae} is proved via an intricate transfinite induction in which we have to show all of these conditions to be equivalent before being able to proceed to the next rank. In fact, we will have to introduce additional equivalent conditions for the induction to work, which involve graphs that have a finite set of their vertices \defi{marked}, endowed with a \defi{marked minor} relation that maps each marked vertex to a branch set containing at least one marked vertex (\Sr{sec finite}).

\medskip
I expect that \Tr{main Intro} remains true, with minor modifications of the proof, when replacing \cgr\ and \cf\ by a variety of subclasses, e.g.\ planar graphs. For the case of rayless trees we will explicitly prove the implication \ref{I iir} $\to$  \ref{I twin}. Combining this with Nash-Williams' \cite{NWbqo} theorem that the trees are \wqo\footnote{In fact, we will need a strengthening of this, proved by Thomas, saying that the graphs of tree-width $k$ are \wqo\ \fe\ $k\in \N$.}, we deduce as above that there are only countably many minor-twin classes of countable rayless trees of rank $\alpha$ \fe\ $\alpha<\omega_1$.  Combining this further with ongoing work with J.~Greb\'ik \cite{GeoGreCom} connecting  minor-closed families with Borel subsets of the space \cg\ of $\N$-labelled graphs\footnote{\cg\ denotes the space of graphs $G$ with $V(G)=\N$, encoded as functions from $\N^2$ to $\{0,1\}$ representing the edges, endowed with the (Tychonoff) product topology. The vertex labelling is ignored when considering minors; it is only used to define the topology on \cg.}, we will deduce the following. 

\begin{theorem} \label{thm Borel Intro}
Let $\ct\subset \cg$ be a minor-closed family of $\N$-labelled rayless forests. Then \ct\ is Borel \iff\ it is proper, i.e.\ it does not contain all rayless forests.
\end{theorem}

\subsection{Structure of the paper}
After some preliminaries, we prove the implication \impl{RW}{FB} of \Tr{fin bqo} in \Sr{sec wbqo}. Much of the rest of the paper is devoted to \Tr{main Intro}, and we prepare it with a warm-up: \Sr{sec R1} focuses on graphs of rank 1, introducing some  fundamental ideas of the paper, and concluding with \Cr{cor Seym Intro}. \Sr{sec subcl} is devoted to a key tool (\Lr{subclasses}) implying that, under natural conditions, the minor-twin class of a rayless graph \g of rank $\alpha$ is determined by the class of marked-minors of \g of ranks $<\alpha$. \Lr{subclasses} is needed for the proofs of all three of \Trs{fin bqo}, \ref{main Intro} and \ref{thm Borel Intro}. Sections~\ref{sec forests} and~\ref{sec Borel} constitute an Intermezzo devoted to the proof of  \Tr{thm Borel Intro} (this constitutes an early example of a statement about unmarked graphs for the proof of which marked graphs are necessary). After this, we return to the proof of \Tr{main Intro}, introducing some un-marking techniques in \Sr{unmarking}, followed by the main technical part in \Sr{sec fix rk}, and concluding with  \Sr{sec main prf}. We then use some of the tools gathered thereby to prove the implication \impl{FB}{RB} of \Tr{fin bqo} in \Sr{sec conv}. We finish with some open problems in \Sr{sec outl}.

\medskip
For a reader wishing to gain an impression of the proofs without reading all details, I recommend reading the following selection. \Sr{sec rank} excluding the proof of \Or{min tree}; \Sr{sec bqo}. \Sr{sec wbqo} up to \eqref{qrank}. \Sr{sec R1}, skimming the proof of \Lr{Rank 1}. The statement of \Lr{subclasses}; its proof is the most challenging one in the paper, and \Lrs{UF} and \ref{Rank 1} illustrate the main ideas avoiding the more {\em annoying parts}. \Srs{sec forests} and \ref{sec Borel} are optional; they only contribute to \Tr{thm Borel Intro}. The statement of \Tr{tfae}, and implications \ref{R iip} $\to$ \ref{R iii} and \ref{R iiip} $\to$  \ref{R ip} of its proof. The statement of \Tr{main}. \Sr{sec conv} excluding the proof of \Lr{subclasses R}. 

\section{Preliminaries}

We will be following the terminology of Diestel \cite{DiestelBook25} for graph-theoretic concepts and Jech \cite{Jech} for set-theoretic ones.

\subsection{(Well)-Quasi-Orders} \label{sec WQOs}

A quasi-order $(Q,\leq)$ consists of a set $Q$ and a binary relation $\leq$ on $Q$ which  is  reflexive and transitive (but not necessarily antisymmetric). 
A quasi-order $(Q,\leq)$ is said to be \defi{\wqo}, if \fe\ sequence \seq{G}\ of its elements \ta\ $i<j$ \st\ $G_i \leq G_j$. If such $i,j$ exist then we say that \seq{G}\ is \defi{good}, otherwise it is \defi{bad}. 

A well-known consequence of Ramsey’s theorem is 
\begin{proposition}[{\cite[Proposition~12.1.1]{DiestelBook25}}] \label{wqo ch}
 $(Q,\leq)$ is WQO \iff\ it has no infinite antichain and no infinite sequence \seq{G}\ \st\ $G_{n+1} \leq G_n$ and $G_n \not\leq G_{n+1}$ \fe\ $n$.
\end{proposition} 
Such a sequence \seq{G}\ is called a  \defi{descending chain}.
The following is also well-known: 

\begin{observation}[{\cite[Corollary~12.1.2]{DiestelBook25}}] \label{good seqs}
Every sequence \seq{G}\ of elements of a WQO $(Q,\leq)$ has a subsequence $\{G_{a_n}\}_\nin$ \st\ $G_{a_n} \leq G_{a_k}$ \fe\ $1\leq n < k$.
\end{observation}

\subsection{(Finite) graph minors and marked graphs} \label{sec finite}

Let $G,H$ be graphs. A \defi{minor model} of $G$ in $H$ is a collection of disjoint connected subgraphs $B_v, v\in V(G)$ of $H$, called \defi{branch sets}, 
and edges $E_{uv}, uv\in E(G)$ of $H$, 
such that each $E_{uv}$ has one end-vertex in $B_u$ and one in $B_v$. We write $G<H$ to express that such a model exists, and say that $G$ is a \defi{minor} of $H$.

A \defi{minor embedding} of $G$ into $H$ is a map $h$ assigning to each $v\in V(G)$ a connected subgraph $B_v$ of $H$, and to each $uv\in E(G)$ an edge $E_{uv}$ of $H$ \st\ these sets form a minor model of $G$ in $H$. We write \defi{$h: G< H$} to denote that $h$ is such a minor embedding. 

A \defi{marked graph} is a pair consisting of a graph \g and a subset $A$ of $V(G)$, called the \defi{marked vertices}. Given two marked graphs $(G,A), (H,A')$, a   \defi{marked minor (model)} of $G$ in $H$ is defined as above, except that for each marked vertex $v$ of $G$, we require that the corresponding branch set $B_v$ contains at least one marked vertex of $H$. We write $(G,A) <(H,A')$, or \defi{$G \mm H$} when $A,A'$ are fixed,  if this is possible. We also extend the above definition of \minem\ canonically to marked minors. 

\medskip

Given a set $X$ of graphs, we write $\forb{X}$ for the class of graphs $H$ \st\ no element of $X$ is a minor of $H$. 

\medskip
We recall the Robertson--Seymour \GMT, which we will only use in the proof  \Cr{cor Seym Intro} (and the warm-up \Lrs{UF} and \ref{Rank 1}). 
\begin{theorem}[{\cite{GM23}}] \label{GMT}
The finite marked graphs are \wqo\ under $\mm$.\footnote{This version of the \GMT\ is a special case of statement 1.7 of {\cite{GM23}}, obtained by replacing $\Omega$ by $\{0,1\}$.}
\end{theorem} 

We conclude this subsection with some basic facts relating connectivity and minors. Firstly, note that if \cb\ a minor model of $G$ in $H$, then it maps each component of \g into a component of $H$. The following is a straightforward extension of this which is easy to prove: 

\begin{observation} \label{min comps}
Let $G,H$ be graphs, let \cb\ a minor model of $G$ in $H$, and let $A\subset V(G)$. Then \cb\ maps each component of $G - A$ into a component of $H - \cb(A)$. \qed
\end{observation}

Here, $\cb(A)$ stands for the image of $A$ under \cb, i.e.\ $\bigcup_{v\in A} B_v$.

A \defi{block} of a graph is a maximal 2-connected subgraph. 
\begin{lemma} \label{blocks}
Let $G,H$ be graphs, and $\cb=\{B_v \mid v\in V(G)\}$ a 
minor model of \g in $H$. Then for each block $D$ of \G, \ti\ a block $D'$ of $H$, \st\ each $B_v, v\in V(D)$ intersects $D'$. Moreover, there is a model $\cb'$ of $D$ in $D'$, obtained by intersecting each $B_v$ with $D'$.
\end{lemma}
\begin{proof}
\Fe\ $e=uv\in E(D)$, let $B_e$ be a $B_u$--$B_v$~edge in $H$, called a \defi{branch edge}. Since $D$ is 2-connected, any two edges $e,f$ of $D$ lie in a common cycle $C$. Moreover, there is a cycle $C'$ in $\bigcup_{v\in V(C)} B_v \subseteq H$ containing $B_e$ and $B_f$. Since there is a unique block containing any edge of a graph, it follows that there is a block $D'$ of $H$ containing $\{B_e \mid e\in E(D)\}$, and this block intersects each branch edge, and hence each branch set of $D$. 

For the second sentence, suppose $x\in V(D')$ is a cut-vertex of $H$. Then all branch sets of \cb\ are contained in the component $K$ of $H-x$ containing $D'-x$, except possibly for a unique branch set $B$ containing $x$. If such a $B$ exists, then after replacing it with $B \cap K$, \cb\ is still a model of $G$, because no branch edge of \cb\ can lie in $H-K$. Thus doing so for every cut-vertex $x\in V(D')$ we modify \cb\ into the desired model $\cb'$ of \g in $D'$.
\end{proof}

\subsubsection{Suspensions} \label{sec susp}

Given a (marked) graph \G, we define its \defi{suspension $S(G)$} by adding an unmarked vertex $s_G$ and joining $s_G$ to each $v\in V(G)$ with an edge.
Given a marked graph \G, we define its \defi{marked suspension $S^\bullet(G)$} by adding a marked vertex $s_G$ and joining $s_G$ to each $v\in V(G)$ with an edge.

\comment{
\begin{lemma} \label{lem cones un}
Let $G,H$ be unmarked graphs. Then $G<H$ \iff\ \\ $S(G)<S(H)$.
\end{lemma}
}

Note that $S(G)$ is always connected even if $G$ is not. Combining this with the following observation will be often convenient, as it will allow us to assume that any bad sequences we consider consist of connected graphs. 

\begin{lemma} \label{lem cones}
Let $G,H$ be marked graphs. Then $G<H$ \iff\ \\ $S^\bullet(G)<S^\bullet(H)$.
\end{lemma}
\begin{proof}
The forward implication is trivial. For the backward implication, let $M=\{B_v \mid v\in V(S^\bullet(G))\}$ be a model of  $S^\bullet(G)$ in $S^\bullet(H)$. If no branch set $B_v$ contains $s_H$, then $M$ witnesses that $S^\bullet(G)$, and hence $G$, is a minor of $H$ and we are done. If $s_H$ is in the branch set of $s_G$, then by deleting that branch set from $M$ we obtain a model of  $G$ in $H$.

Thus it only remains to consider the case where some $B_v$ with $v\in V(G)$ contains $s_H$. In this case $B_{s_G}$ cannot contain $s_H$ too, and so $B_{s_G} \subseteq H$, and $B_{s_G}$ contains a marked vertex since $s_G$ is marked. Then by removing $B_{s_G}$ from $M$, and re-defining $B_v$ to be $B_{s_G}$, we have modified $M$ into a model of \g in $H$.
\end{proof}

The unmarked version of \Lr{lem cones} is also true, and easier to prove along the same lines: 
\begin{lemma} \label{lem cones unm}
Let $G,H$ be (unmarked) graphs. Then $G<H$ \iff\ \\ $S(G)<S(H)$. \qed
\end{lemma}


\subsubsection{Minor-twins} \label{sec twins}

We say that $G$ and $H$ are \defi{minor-twins}, if both $G<H$ and $H<G$ hold. Any two finite minor-twins are isomorphic, but in the infinite case the relation is much more interesting. 

The class of countable graphs can be decomposed into its  \defi{minor-twin classes}, whereby two graphs belong to the same class whenever they are minor-twins. The minor-twin class of a graph $G$ will be denoted by $[G]_<$. These definitions have obvious analogues for marked minors.

Given a graph class \cc, we let \defi{$|\cc|_<$} denote the cardinality of the set of minor-twin classes of elements of \cc. We define \defi{$|\cc|_{\mm}$} analogously for marked minors.

\subsection{The Rank of a rayless graph} \label{sec rank}

A graph is \defi{rayless}, if it does not contain a 1-way infinite path. Schmidt~\cite{schmidt83} assigned to every rayless graph an ordinal number,  its \defi{rank}, reminiscent of the notion of rank for Borel sets. This notion often enables us to prove results about rayless graphs by transfinite induction on the rank.

The notion of rank comes from the observation that it is possible to construct all rayless graphs by a recursive, transfinite procedure, starting with the class of finite graphs and then, in each step, glueing graphs constructed in previous steps along a common finite vertex set, to obtain new rayless graphs as follows.

\begin{definition}\label{def:clu} For every ordinal $\alpha$, we recursively  define a class of graphs $\ran{\alpha}$ by transfinite induction on $\alpha$ as follows:
\begin{itemize}
\item $\ran{0}$ consists of the finite graphs; and
\item if $\alpha>0$, then a graph \g is in $\ran{\alpha}$ if \ti\ a finite $S\subset V(G)$ \st\ each component of $G - S$ lies in $\ran{\beta}$ for some $\beta<\alpha$.
\end{itemize}
\end{definition}

Schmidt~\cite{schmidt83,DiestelBook25,HalStr} proved that a graph is rayless if and only if it belongs to $\ran{\alpha}$ for some $\alpha$. Easily, if $G$ is countable then so is this $\alpha$, but it may be greater than \oo\ (\Or{min tree} below). Let $\ran{<\alpha}:= \bigcup_{\beta< \alpha} \ran{<\beta}$.

The \defi{rank $\Rank(G)$} of a rayless graph \g is the least ordinal $\alpha$ \st\ $G\in \ran{\alpha}$.
For a class \cc\ of graphs, we let \defi{$\Rank(\cc)$} be the least ordinal $\alpha$ \st\ $\Rank(G)<\alpha$ holds \fe\ $G\in \cc$.

Schmidt~\cite{schmidt83,HalStr} also proved that each rayless graph \g has a unique \defi{kernel $A(G)$}, i.e.\ a minimal set of vertices $S$ \st\ each component of $G - S$ lies in $\ran{\gamma}$ for some $\gamma<\Rank(G)$. We will use the following simple observation:

\begin{proposition} \label{inf bet}
\Fe\ ordinal $\alpha> 0$ every $G\in \Rank_\alpha$, and every $\beta<\alpha$, there are infinitely many components $C$ of $G-A(G)$ \st\ $\Rank(C)\geq \beta$.
\end{proposition}
\begin{proof}
If the set \cc\ of such components $C$ is finite, then the finite set $A(G) \cup \bigcup_{C\in \cc} A(C)$ separates $G$ into components that all have ranks less than $\beta$, yielding the contradiction $\Rank(G)\leq \beta$.
\end{proof}

\subsubsection{Rank and minors} 

It is well-known, and not hard to prove, that rank is monotone \wrt\ minors:
\begin{observation}[{\cite[Proposition~4.4.]{HalStr}}] \label{minor rank}
Let $G,H$ be graphs with $G<H$. Then  $\Rank(G) \leq \Rank(H)$.
\end{observation}

The following may be well-known but I could not find a reference:
\begin{observation} \label{apices}
Let $G,H$ be graphs with  $\Rank(G) = \Rank(H)=\alpha$, and $\cb=\{B_v \mid v\in V(G)\}$ a minor model of \g in $H$. Then $B_v$ intersects $A(H)$ \fe\ $v\in A(G)$.
\end{observation}
\begin{proof}
Let $v\in A:= A(G)$, and let $\cc=\cc_v$ denote the set of components of $G-A$ sending an edge to $v$. We claim that
\labtequ{KG}{$\Rank(\cc)=\alpha$.}
Indeed, if $\Rank(\cc)<\alpha$, then $G_v:=G[\{v\}\cup \bigcup \cc]$ has rank less than $\alpha$ too. Moreover, each component of $G- (A-v)$ is either $G_v$ or a component of $G-A$, and therefore it has rank less than $\alpha$ (\fig{figGv}). This contradicts the fact that $A$ is, by definition, a minimal set with this property.

\begin{figure} 
\begin{center}
\begin{overpic}[width=.45\linewidth]{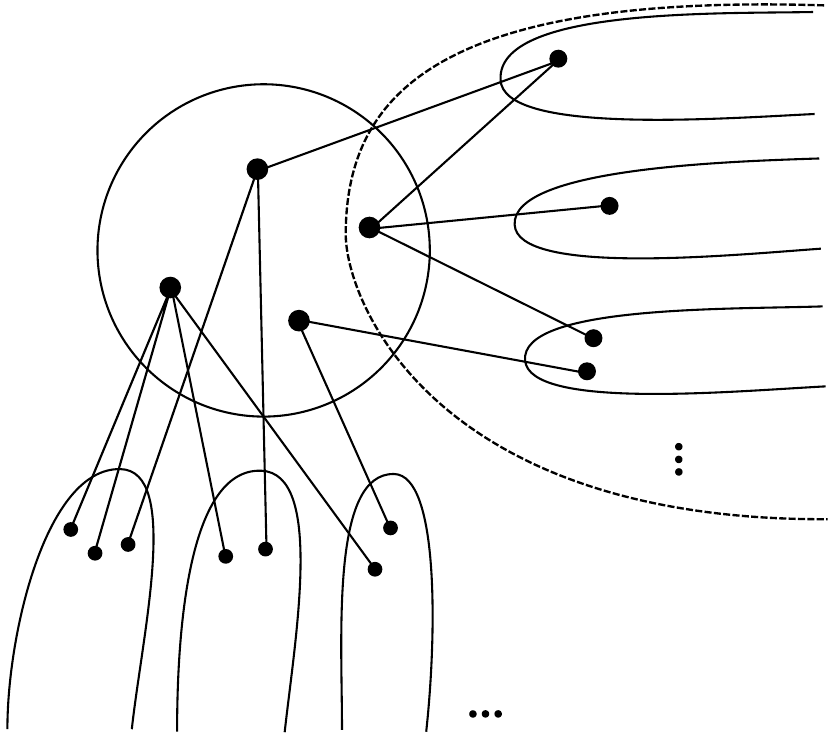} 
\put(43,58){$v$}
\put(47,85){$G_v$}
\put(15,63){$A$}
\end{overpic}
\end{center}
\caption{The components of $G-A$ and $G- (A-v)$ in the proof of \eqref{KG}.} \label{figGv}
\end{figure}

\medskip
Suppose $B_v$ contains no vertex of $A':=A(H)$ for some $v\in A$. Since $B_v$ is connected, it is contained in a component $C$ of $H-A'$. By \eqref{KG}, we have $\Rank(G_v)= \alpha$. This remains true if we delete from $G_v$ those elements of $\cc_v$ containing a branch set intersecting $A'$, since there are at most finitely many such branch sets. Let $G'_v$ be the subgraph of $G_v$ obtained after this deletion. Since $G'_v$ is connected, and all its branch sets avoid $A'$,  its image under \cb\ is contained in $C$. But $\Rank(C)<\Rank(H)=\alpha$, while $\Rank(G'_v)= \alpha$ as observed above. This contradicts \Or{minor rank}, hence $B_v$ must intersect $A'$.
\end{proof}

\begin{observation} \label{min tree}
\Fe\ countable ordinal $\alpha$, \ti\ a countable tree $T_\alpha$ with $\Rank(T_\alpha)=\alpha$, \st\ $T_\alpha< G$ \fe\ non-empty, connected, graph $G$ with $\Rank(G)\geq \alpha$. 
\end{observation}
\begin{proof}
We will prove the statement  by induction on $\alpha$. For this it will be convenient to think of each $T_\alpha$ as a rooted tree, and denote its root by $r_\alpha$. Instead of $T_\alpha< G$, we will prove the following strengthening, which will be important for our induction:
\labtequ{any root}{\fe\ $v\in V(G)$, there is a model of $T_\alpha$ in \g \st\ the branch set of $r_\alpha$ contains $v$.}

For $\alpha=0$, we just let $T_\alpha$ be the tree on one vertex, and note that \eqref{any root} is trivially satisfied by letting $B_{r_\alpha}=\{v\}$. 

For $\alpha>0$, we construct $T_\alpha$ as follows. We start with the disjoint union of countably infinitely many copies of $T_\beta$ for each $\beta<\alpha$, add a new vertex $r_\alpha$, and join $r_\alpha$ to the root of each such copy of $T_\beta$ with an edge. Note that $A(T_\alpha)=\{r_\alpha\}$ by construction. 

\smallskip
This completes the construction of $T_\alpha$ \fe\ ordinal $\alpha$, and it now remains to prove \eqref{any root}, which we do by transfinite induction of  $\alpha$. Having checked the start $\alpha=0$ of the induction above, we may assume that \eqref{any root} holds for all ordinals $\beta<\alpha$. Given $G$ as above and $v\in V(G)$, we construct the desired model of $T_\alpha$ in \g as follows. Pick a \pth{v}{A(G)}\ $P$ in \G. Moreover, \fe\ $x,y\in A(G)$, pick a \pth{x}{y}\ $P_{xy}$ in \G. Since $A(G)$ is finite, the union of all these paths meets only a finite set \cc\ of components of $G-A(G)$. We let  $B_{r_\alpha}=A(G) \cup \bigcup \cc$ be the branch set of $r_\alpha$ in our model. All other branch sets will be chosen within $G - \bigcup \cc$. 

Let \seq{C}\ be an enumeration of the components of $T_\alpha - r_\alpha$, and recall that each $C_n$ is isomorphic to $T_\beta$ for some $\beta<\alpha$. For $i=1,2,\ldots$, we recursively find a model of $C_i$ in  $G - \bigcup \cc$ as follows. We pick a component $C'_i$ of  $G - \bigcup \cc$ with $\Rank(C'_i)\geq \beta$ \st\ $C'_i \not\in \cc$, and $C'_i\neq C'_j$ for any $j<i$, which $C'_i$ exists by \Prr{inf bet}. Since $G$ is connected, there is a vertex $v_i\in C'_i$ sending an edge $e_i$ to $A(G)$. Applying the inductive hypothesis \eqref{any root} with $G$ replaced by  $C'_i$, and $v$ replaced by $v_i$, and $\alpha$ replaced by $\beta$, we obtain a minor model of $C_i \isom T_\beta$ in $C'_i$, in which the branch set corresponding to $r_\beta$ contains $v_i$. Adding the edges $e_i$, and the branch set $B_{r_\alpha}$ to these models for all $i\in \N$ we obtain the desired model of  $T_\alpha$ in \G.
\end{proof}

\subsection{Better-quasi-orders} \label{sec bqo}

Rather than repeating the original definition of a better-quasi-order, we will work with an equivalent one. Intuitively, a quasi-order $(Q,\leq)$ is \bqo\  if it is \wqo, and so are its subsets, sets of subsets, sets of sets of subsets, and so on transfinitely, whereby we recursively extend $\leq$ from elements of $Q$ to subsets of $Q$. \mymargin{Improve} To make this precise, we need the following terminology.

Let $\cp^*(A)$ denote the set of non-empty subsets of a set $A$, i.e.\ $\cp^*(A):= \cp(A)\sm \{\emptyset\}$. Let $Q$ be a quasi-order. For every ordinal $\alpha$ we define, by transfinite induction, the `iterated power set' $V^*_\alpha(Q)$ as follows. We start our induction by setting $V^*_0(Q):=Q$. Having defined $V^*_\alpha(Q)$, we let $V^*_{\alpha+1}(Q):= \cp^*(V^*_\alpha(Q))$. Finally, if $\alpha$ is a limit ordinal, we let $V^*_\alpha(Q):=\bigcup_{\beta< \alpha} V^*_\beta(Q)$. We say that $X\in V^*_\alpha(Q)$ is \defi{hereditarily countable} if it is a countable set of hereditarily countable sets (the latter having been defined recursively, starting by declaring each element of $Q$ to be  hereditarily countable).

Having defined $V^*_\alpha(Q)$ \fe\ $\alpha$, we let $V^*(Q):=\bigcup_{\alpha} V^*_\alpha(Q)$. The \defi{$Q$-rank $\QRank(X)$} of an element $X\in V^*_\alpha(Q)$ is defined as $\alpha+1$ where $\alpha$ is the least ordinal \st\ $X\in V^*_\alpha(Q)$. Thus $\QRank(X)=1$ \iff\ $X\in Q$ (the $+1$ may look strange for now, but it will be justified later).

\begin{definition} \label{def qo}
Given a quasi-order $(Q,\leq)$, recursively define a quasi-order $\leq_+$ on $V^*(Q)$ as follows
\begin{enumerate}
\item \label{qo i} if $X,Y \in Q$, then $X\leq_+ Y$ in $V^*(Q)$ \iff\ $X\leq Y$ in $Q$;
\item \label{qo ii} if $X\in Q$ and $Y \not \in Q$, then
$X\leq_+ Y$ \iff\ there exists $Y' \in Y$ with $X\leq_+ Y'$;
\item \label{qo iv} if $X\not \in Q$ and $Y \not \in Q$, then
$X\leq_+ Y$ \iff\ for every $X'\in X$ there exists $Y'\in Y$ with $X'\leq_+ Y'$.
\end{enumerate}
\end{definition}

Let $H^*_{\omega_1}(Q)$ denote the set of hereditarily countable elements of $V^*_{\omega_1}(Q)$ equipped with the above quasi-order  induced from $V^*(Q)$.

\begin{theorem}[{\cite[Theorem 3.45]{PeqTow}}\footnote{Private communication with Yann Pequignot suggests that this theorem was known to experts on BQO, but apparently it first appear in print in this formulation in \cite{PeqTow}.}] \label{lem H wqo}
A quasi-order $Q$ is BQO if and only if $H^*_{\omega_1}(Q)$ is WQO. \mymargin{{\tiny if and only if $H^*_{\omega_1}(Q)$ is well-founded {\cite[Theorem 3.46]{PeqTow}}.}}
\end{theorem}

We will effectively use \Tr{lem H wqo} as our definition of \defi{\bqoing}. (We will refrain from repeating Nash-Williams' original definition of \bqoing, as this will never be used in the paper.) 
\smallskip

Let $Seq(Q)$ be the set of all finite or countably infinite sequences with elements in $(Q,\leq)$. We endow $Seq(Q)$ by a quasi-ordering $\preceq$ by letting $S \preceq T$ if there exists an embedding from $F$ into $G$, i.e.\ a strictly increasing map $\phi$ from the index set of $S$ to that of $T$, \st\ $S(i) \leq T(\phi(i))$ \fe\ $i$. We will use the following well-known lemmas about \bqoing.

\begin{lemma}[{\cite{NWbqo}, \cite[Lemma 4]{KuhWel}}] \label{lem Seq}
If a quasi-order $Q$ is BQO, then so is $Seq(Q)$.
\end{lemma}

\begin{lemma}[{\cite[COROLLARY~22A]{NWbqo}, \cite[Lemma 3]{KuhWel}}] \label{lem prod}
If two quasi-orders $Q_1,Q_2$ are BQO, then so is $Q_1 \times Q_2$.
\end{lemma}

For the proof of the implication \impl{FB}{RB} of \Tr{fin bqo} we will use the following closedness property of $H^*_{\omega_1}$. 

\begin{observation} \label{obs H}
For every quasi-order $Q$, we have\\ $H^*_{\omega_1}(H^*_{\omega_1}(Q))= H^*_{\omega_1}(Q)$.
\end{observation}
\begin{proof}
Notice that any countable non-empty subset $X$ of $H^*_{\omega_1}(Q)$ belongs to $H^*_{\omega_1}(Q)$\footnote{This is the crucial observation, and it is used by Pequignot in {\cite[Theorem 3.46]{PeqTow}}. It is important here that we are working with $H^*_{\omega_1}$ instead of $V^*$.}; this is because each $Y\in X$ belongs to some $V^*_{\alpha_Y}(Q)$, and letting $\beta:= \sup_{Y\in X} \alpha_Y$ we have $X\in V^*_{\beta}(Q) \cup V^*_{\beta+1}(Q)$ by the definitions because $\beta<\oo_1$ as $\oo_1$ is a regular ordinal. This means that the hereditarily countable sets in $V^*_{1}(H^*_{\omega_1}(Q))$ are already in $H^*_{\omega_1}(Q)$. By induction, the same applies to $V^*_{\alpha}(H^*_{\omega_1}(Q))$ \fe\ $\alpha<\oo_1$, and so $H^*_{\omega_1}(H^*_{\omega_1}(Q))= H^*_{\omega_1}(Q)$.
\end{proof}


\section{\cgr\ is WQO impies \cf\ is BQO} \label{sec wbqo}

The aim of this section is to prove the implication \impl{RW}{FB} 
of \Tr{fin bqo}:
\begin{theorem} \label{fin bqo back}
If the countable { rayless} graphs are WQO, then the finite graphs are BQO.
\end{theorem}
\begin{proof}
Let $Q$ be the set of finite \defi{1-connected graphs}, i.e.\ connected graphs with at least two vertices. 
By \Lrs{lem Seq} and~\ref{lem prod}, it suffices to prove that $Q$ is \bqo. Indeed, any finite graph $G$ is the disjoint union of a number $n_G$ of isolated vertices and a sequence $s_G$ of 1-connected subgraphs. Thus we can represent $G$ as a pair $(n_G,s_G)$, and apply \Lr{lem prod} to such pairs after applying \Lr{lem Seq} to these sequences, recalling that $\N$ is well-ordered and hence \bqo.

\medskip
It thus remains to prove
\labtequ{Q bqo}{If the class \cgr\ of countable rayless graphs is WQO, then $Q$ is BQO.}
To prove this we will define a map $T$ that assigns to each set $X$ in $H^*_{\omega_1}(Q)$ a rayless graph $T(X)$ in such a way that 
\labtequ{T}{$T(X)< T(Y)$ implies $X\leq Y$ \fe\ $X,Y\in H^*_{\omega_1}(Q)$,}
where $\leq$ stands for the relation $\leq_+$ of \Dr{def qo}.
To see how this implies \ref{Q bqo}, recall that, by \Lr{lem H wqo}, if  $H^*:=H^*_{\omega_1}(Q)$ is WQO then $Q$ is BQO. Let \seq{X}\ be a sequence of elements of $H^*$. If the countable rayless graphs are WQO, then $\{T(X_n)\}_\nin$ is good, hence so is \seq{X} by \ref{T}, and so $H^*$ is WQO as desired.

It remains to define $T$ and prove that it satisfies \ref{T}. We define $T(X), X\in H^*$ by transfinite induction on the $Q$-rank of $X$ as follows. If $\QRank(X)=1$, in which case $X\in Q$ is a finite connected graph (on at least two vertices), then $T(X)$ comprises a countably infinite collection of pairwise disjoint copies of $X$, and an additional vertex $r$, called the \defi{root}, joined by an edge to all other vertices. 

If $\QRank(X)>1$, then $T(X)$ comprises a countably infinite collection of pairwise disjoint copies of $T(X')$ for each $X'\in X$, and an additional root vertex $r$ joined by an edge to the root of each such $T(X')$. 

It is easy to show, by transfinite induction on $\QRank(X)$, that $T(X)$ is a countable graph, and that it is rayless; to see the latter, notice that any ray in $T(X)$ can visit $r$ at most once, hence it would need to have a sub-ray in a copy of $T(X')$ for some $X'\in X$ if $\QRank(X)>1$, or in a copy of $X$ if $\QRank(X)=1$. It is not hard to show, again by transfinite induction on $\QRank(X)$, that 
\labtequ{qrank}{$\Rank(T(X))=\QRank(X)$ \fe\ $X\in H^*$.}
(This justifies the $+1$ in the definition of $\QRank(X)$; without it this equation would hold only when $\QRank(X)\geq \omega$.)

\medskip
We call $r=:r(T(X))$ the \defi{root} of $T(X)$. The \defi{children} of $r$ are the copies of the roots $\{r(T(X')) \mid X' \in X\}$. The \defi{descendant} relation is the transitive closure of the child relation just defined. Given a descendant $r'$ of $r$, we write $\ceil{r'}$ for the component of $T(X)$ containing $r'$ formed when deleting the edge from $r'$ to its parent. We set $\ceil{r}=T(X)$. Notice that $\ceil{r'}$ is isomorphic to $T(Y)$ for some element $Y$ of the transitive closure of $X$.

By construction, we can assign to each vertex $v$ of $T(X), X\in H^*$ an ordinal number called the \defi{level} of $v$, similar to the notion of \QRank, as follows. If $X\in Q$, each vertex of $X$ and its copies has level 0. For $X\in H^*$, we proceed inductively to define the \defi{level} of $r(T(X))$ to be the smallest ordinal that is larger than the level of each vertex $v\neq r(T(X))$ of $T(X)$, which is defined in a previous step of the induction since $v\in V(T(X'))$ for some $X'\in X$. 

Notice that for each triangle $\Delta$ of $T(X)$, at least two of the vertices of $\Delta$  have level 0.
Recall moreover that each $X\in Q$ contains at least one edge. It thus follows from our construction that, \fe\ $X\in H^*$,
\labtequ{triangles}{a vertex $v$ of $T(X)$ lies in infinitely many triangles of $T(X)$ \iff\ $v$ has level 1, and no vertex of level $>1$ lies in a triangle.}

\medskip

We now prove that this definition of $T$ satisfies \ref{T}, by a nested transfinite induction on $\QRank(X)$ and $\QRank(Y)$.

\begin{remark} \label{rem Q}
More generally, we can let $Q'$ be any non-empty family of 1-connected finite graphs. Then \ref{triangles} still holds. By restricting \eqref{T} to $Q'$, we deduce the following variant of \eqref{Q bqo}: if $T[Q']$, i.e.\ the image of $H^*_{\oo_1}(Q')$ under $T$, is \wqo, then $Q'$ is \bqo. 
The remainder of this proof works verbatim when replacing $Q$ by $Q'$, and $\cgr$ by $T[Q']$ in \eqref{Q bqo}. 
\end{remark}

\medskip
Our inductive proof of \eqref{T} starts with $\QRank(X)$ being 1: assume that $T(X)<  T(Y)$ holds for some $X\in Q$ and $Y\in H^*$. Let $B_r$ denote the branch set corresponding to the root $r=r(T(X))$ in some minor model $\cb$ of $T(X)$ in $T(Y)$. We claim that $B_r$ contains at least one vertex of level 1 in $Y$. Indeed, by \eqref{triangles} $r$ lies in infinitely many triangles of $T(X)$, and if  $B_r$ avoids level 1 vertices of $T(Y)$, then it lies in at most finitely many triangles of any minor of $T(Y)$.

Let $G$ be one of the copies of $X$ in $T(X)$, let $xy$ be an edge of $G$, and notice that $xyr$ is a triangle of $T(X)$. Then the branch sets $B_x,B_y$ of $\cb$ are contained in a component $C$ of level 0 vertices of $T(Y)$. Let $r'$ be the unique level 1 vertex of $T(Y)$ sending edges to $C$. Notice that $r'\in B_r$. Since $G$ is connected, and $r'$ separates $C$ from the rest of $T(Y)$, we deduce that $B_v\subset C$ \fe\ $v\in V(G)$. By our construction of $T(Y)$, there are infinitely many level 0 components $C'$ of $T(Y)$ isomorphic with $C$ and incident with $r'$. It follows that 
\labtequ{ceil}{$T(X)\leq \ceil{r'}$.}
If $\QRank(Y)=1$, then $C$ is just a copy of $Y$ and we obtain the desired $X \leq Y$ since $G$ was a copy of $X$. If $\QRank(Y)>1$, then let $Y'$ be the element of $Y$ \st\ some copy of $T(Y')$ in $T(Y)$ contains $r'$, and hence $\ceil{r'}$. In this case \eqref{ceil} implies $T(X)\leq T(Y')$, and by transfinite induction on $\QRank(Y)$, of which the previous case is the initial step, we deduce $X \leq Y'$, and hence $X \leq Y$ by \ref{qo ii} of \Dr{def qo}. 

Thus we have completed the initial step $\QRank(X)=1$ of our inductive proof of \eqref{T}. Assume now that $\QRank(X)>1$, and $T(X)\leq T(Y)$ holds for some $Y\in H^*$. We cannot have $\QRank(Y)=1$ by \Or{minor rank} and \eqref{qrank}. Thus we are in case \ref{qo iv} of \Dr{def qo}, and so our task is to find, for each $X'\in X$, some $Y'\in Y$ \st\ $X'\leq Y'$.

To this aim, let again $B_r$ denote the branch set of some  minor model $\cb$ of $T(X)$ in $T(Y)$  corresponding to the root $r=r(T(X))$. Let $r'$ be the vertex of $B_r$ of maximal level among all vertices of $B_r$; this exists because  $B_r$ is connected, and hence if it contains vertices of all levels $\beta< \alpha$ for some ordinal $ \alpha$, then it must also contain a vertex of level $\alpha$. By \eqref{triangles}, the level of $r'$ is at least 1. Notice that all but at most one of the edges of $r$ are mapped by $\cb$ to descendants of $r'$, because $r'$ sends at most one edge to a non-descendant. Call this edge $e$ if it exists. Since $T(X)$ contains infinitely many pairwise disjoint copies of $T(X')$ incident with $r$, at least one (in fact almost all) of these copies $G$ is mapped by $\cb$ to a subgraph of $T(Y)$ avoiding $e$. Therefore, $G$ is mapped by $\cb$ into $\ceil{r''}$ for some child $r''$ of $r'$, because $G$ is connected, $r'$ separates its children, and $r'\in B_r$ cannot lie in $B_v$ for any $v\in V(G)$. Let $Y'$ be the element of $Y$ for which some copy of $T(Y')$ contains $r''$, which exists since $r''\neq r(T(Y))$ as $r''$ is a child of another vertex. Then $T(Y')$ contains $\ceil{r''}$, and hence $T(X')< T(Y')$ because $T(X')< \ceil{r''}$. Since $\QRank(X')<\QRank(X)$, our inductive hypothesis yields $X'\leq Y'$ as desired, completing the inductive step.
\end{proof}

By \Rr{rem Q}, we immediately deduce
\begin{corollary} \label{cor Q}
Let $F$ be a set of finite graphs, and let $Q$ be the set of 1-connected elements of $F$. If $T[Q']$ is \wqo, then $F$ is \bqo. 
\end{corollary} 

It is not hard to prove the converse of \eqref{T},  by induction on the $Q$-rank 
by recursively preserving the property that the branch set of the root contains the root of the target graph:
\labtequ{Tc}{$X\leq Y$ implies $T(X)< T(Y)$ \fe\ $X,Y\in H^*_{\omega_1}(Q)$.}
Thus combining \eqref{T},  \eqref{Tc} and \Tr{lem H wqo}, and recalling that  \defi{$T[Q]$} denotes the image of $H^*_{\omega_1}(Q)$ under $T$, we deduce
\labtequ{cor iff}{The finite graphs are \bqo\ \iff\ $T[Q]$ is \wqo.}
This can be thought as an additional equivalent condition that could be added to \Tr{fin bqo}.


\section{Graphs of rank 1, and the self-minor conjecture} \label{sec R1}

This section introduces some of the fundamental techniques used throughout the paper, and serves as a preparation towards the more difficult \Sr{sec subcl}. It handles graphs of rank 1. The section culminates with the proof of \Tr{cor Seym Intro}, which is based on the same techniques.
\medskip

Let \defi{\uf} (to be read ``union-finite'') denote the class of countable graphs \g \st\ each component of \g is  finite. Note that $\uf\subset \Rank_1$. 
For $G\in \uf$, we let $\cc(G)$ denote the class of finite graphs $H$ \st\ $H< G$. 
The following is a toy version of \Lr{subclasses}, a central tool for the proof of \Tr{main Intro}.

\begin{lemma} \label{UF}
\Fe\ $G,G'\in \uf$, we have $G<G'$ \iff\ $\cc(G)\subseteq \cc(G')$.
\end{lemma}

To see the relevance of this to cardinality of minor-twin classes as in \Tr{main Intro} \ref{I twin}, let us prove that \Lr{UF} implies that $|\uf|_<= \aleph_0$. In fact, this statement is equivalent to the \GMT:

\begin{corollary} \label{UF ctble}
$|\uf|_<= \aleph_0$ \iff\ the finite graphs are \wqo.
\end{corollary}
\begin{proof}
For the backward direction, note that
\Lr{UF} says that the minor-twin class $[G]_<$ of $G\in \uf$ is determined by $\cc(G)$. Using the fact that the finite graphs are \wqo, we can express $\cc(G)$ as $\cc(G)=\forb{X}$ for a finite set $X$ of finite graphs. Thus there are countably many choices for $\cc(G)$, hence for $[G]_<$, since there are countably many finite graphs to choose $X$ from. 

\smallskip
For the forward direction, suppose for a contradiction that \seq{G}\ is an anti-chain of finite graphs under the minor relation. By \Lr{lem cones unm}, we may assume that each $G_n$ is connected. For each $X\subset \N$, let $\cc_X$ be the (minor-closed) class of finite graphs $\forb{X}$. Let $G_X:= \bigcup \cc_X \in \uf$ be the (countably infinite) graph obtained as the disjoint union of all the graphs in $\cc_X$. Note that \fe\ finite graph $H$, we have $H<G_X$ \iff\ $H\in \cc_X$ because each $G_i\in X$ is connected. This implies that $G_X$ is not a minor-twin of $G_Y$ whenever $X\neq Y$, because any graph in the symmetric difference $X\sydi Y$ is a minor of exactly one of $G_X,G_Y$. Since there are continuum many $X\subset \N$, we have obtained continuum many minor-twin classes $[G_X]_<$ (thus we could add $|\uf|_< <  \cont$ as a further equivalent statement).
\end{proof}

Generalising the idea of the proof of  \Cr{UF ctble} to higher ranks will be the key to proving the equivalence \ref{I iir} $\leftrightarrow$ \ref{I twin} of \Tr{main Intro},  the main difficulty being that we do not have an analogue of the \GMT\ for ranks higher than 0. 

\medskip
We prepare the proof of \Lr{UF} by recalling a well-known idea: \\ 

\noindent {\bf Hilbert's Hotel Principle:} Suppose a hotel has infinitely many single rooms, numbered $R_1,R_2, \ldots$, and each $R_i$ is occupied by a guest $G_i$. If a new guest $G$ arrives, they can be accommodated in $R_1$, by moving each $G_i$ to $G_{i+1}$. 

\begin{proof}[Proof of \Lr{UF}]
The forward implication follows immediately by restricting a minor model of $G$ in $G'$ to any $H\in \cc(G)$.

For the backward implication, suppose $\cc(G)\subseteq \cc(G')$, and let $G_n$ be the union of the first $n$ components of $G$ in a fixed but arbitrary enumeration of its components. Let $G'_n$ be a subgraph of $G'$ \st\ $G_n< G'_n$, which exists since $\cc(G)\subseteq \cc(G')$. Easily, we may assume $G'_n$ is finite. Let $h_n: G_n < G'_n$ be a minor embedding (as defined in \Sr{sec finite}).

Call a component $C_i$ of $G$ \defi{$h$-stable}, if $\bigcup_n h_n(C_i)$ is finite; in other words, if $C_i$ is mapped to a finite set of components of $G'$ by the $h_n, n\in \N$. Let \seq{S}\ be an enumeration of the $h$-stable components of \G, and \seq{U} an enumeration of the other components of \G. One of these enumerations may be finite, or even empty.

Let $G_S$ denote the (possibly empty) subgraph of \g consisting of its $h$-stable components. By a standard compactness argument, there is a minor embedding $h_S: G_S < G'$ \st\ $h_S(S_i)$ coincides with $h_n(S_i)$ for infinitely many values of $n$. If $G_S = G$ we are done, so suppose from now on $U_1$ exists. 

We will now modify $h_S$, recursively in at most \oo\ steps $i=1,2,\ldots$, into a minor embedding of $G$ into $G'$, whereby in step $i$ we handle $U_i$. Importantly, the image of $h_S$ might be all of $G'$, and so we may have to reshuffle $G_S$ inside $G'$ to make space for the $U_i$'s.

We set $h^0_S:= h_S$, and assume recursively that $h^j_S: G_S \cup \{U_1, \ldots,U_{j-1}\}$ has been defined for every $j<i$. Moreover, we assume that a finite number of components of $G$ have been \defi{nailed} into components of $G'$, which means that we promise that $h^i_S$ will coincide with $h^{i-1}_S$ 
on all components that have been nailed before. We will ensure that every component of \G---stable or not--- will be nailed in some step. No components have been nailed at the beginning of step $1$.
If $U_i$ does not exist, then we just let $h^{i}_S= h^{i-1}_S$, and nail $S_i$ ---this is the easy case, and we can just terminate the process as $h^{i-1}_S$ is a minor embedding of $G$ into $G'$ in this case. 

Otherwise, let $\seq{C^i}$ be an infinite sequence of distinct components of $G'$ into which $U_i$ is embedded by some $h_n$, which exists since $U_i$ is not $h$-stable. 
As the finite graphs are \wqo\ by \Tr{GMT}, $\seq{C^i}$ has an infinite subsequence \seq{Y}\ \st\ $Y_r<Y_m$ \fe\ $r<m \in \N$ by \Or{good seqs}. Pick $k=k(i)$ such that no $Y_m, m\geq k$ has been nailed yet.

If $Y_k$ does not intersect the image of $h^{i-1}_S$, we let $h^{i}_S$ extend $h^{i-1}_S$ by embedding $U_i$ into $Y_k$; this is possible since some $h_n$ embeds $U_i$ into $Y_k$ by the definition of the latter. We nail $U_i$ to $Y_k$ (thereby promising that $h^{j}_S$ will embed $U_i$ into $Y_k$ \fe\ $j\geq i$). Finally, if $S_i$ exists, we nail it to the component containing $h^{i-1}_S(S_i)$, again promising that $h^{j}_S(S_i)$ is fixed from now on.

It remains to consider the ---more difficult--- case where $Y_k$ intersects the image of $h^{i-1}_S$. In this case, imitating Hilbert's Hotel principle, we modify  $h^{i-1}_S$ into $h^{i}_S$ by shifting the `contents' of each $Y_m, m\geq k$ to $Y_{m+1}$, and mapping $U_i$ into $Y_k$. To make this precise, fix a minor embedding $g_m: Y_m < Y_{m+1}$ \fe\ $m\geq k$, which exists by the definition of \seq{Y}. Then, \fe\ $m\geq k$, and every component $C$ of $G$ \st\ $h^{i-1}_S(C)$ intersects $Y_m$ ---and therefore $h^{i-1}_S(C)$ is contained in $Y_m$--- we let $h^{i}_S(C):= g_m \circ h^{i-1}_S(C)$, so that $h^{i}_S$ embedds $C$ into $Y_{m+1}$. Thus $h^{i}_S$ now maps the domain $G_S \cup \{U_1, \ldots,U_{i-1}\}$ of $h^{i-1}_S$ to $G' \sm Y_k$. We extend $h^{i}_S$ to $U_i$, embedding $U_i$ into $Y_k$ (by imitating some $h_n$), and we nail $U_i$ to $Y_k$.

Again, if $S_i$ exists, we nail it to the component containing $h^{i}_S(S_i)$ ---which may coincide with the component containing $h^{i-1}_S(S_i)$, or have been shifted from some $Y_m$ to $Y_{m+1}$.

This completes the definition of $h^{i}_S, i\in \N$. Note that each of $U_i,S_i$ that exists has been nailed by step $i$, and its  $h^{\ell}_S$-image is fixed for $\ell\geq i$. Thus $h^{i}_S$ converges, as $i\to \infty$, to a minor embedding $h: G \to G'$, proving our claim $G<G'$.
\end{proof}

Our next result extends \Lr{UF} from \uf\ to its superclass $\Rank_1$. For this we will need to adapt the above definition of $\cc(G)$ (\Dr{def ccm} below), whereby we will have to consider marked graphs. 
 
\begin{definition} \label{def part}
For $G\in \ran{\alpha}, \alpha \geq 1$, a \defi{co-part} of \g is a component of $G- A(G)$. Given a co-part $C$ of \G, we call the subgraph $G[C \cup A(G)]$ induced by $C \cup A(G)$ a \defi{part} of $G$. 
\end{definition} 
Note that each part of \g has lower rank than that of \G.

\begin{definition} \label{def ranm}
Let \defi{$\ranm{\alpha}$} denote the class of marked  graphs $(G,M)$ with $G\in \ran{\alpha}$  and $M$ finite. Define \defi{$\ranm{<\alpha}$} analogously.
\end{definition}

\begin{definition} \label{def ccm}
Given a rayless graph $G\in \ran{\alpha}$, we let \defi{$\ccm(G)$} denote the class of marked graphs in $\ranm{<\alpha}$ that are marked-minors of $(G,A(G))$. 
\end{definition}

The following lemma, which extends \Lr{UF}, is again not formally needed for our later proofs; we include it as a warm-up towards the more difficult \Lr{subclasses}, but the reader will need to be familiar with its proof.  

\begin{lemma} \label{Rank 1}
Let $G,H\in \ran{1}$, and suppose $|A(G)|= |A(H)|$. We have\\ $G<H$ \iff\ $\ccm(G)\subseteq \ccm(H)$.
\end{lemma}
\begin{proof}
As before, the forward implication is straightforward.

For the backward implication, let \seq{P}\ be an enumeration of the parts of \G, and  \seq{P'}\  an enumeration of the parts of $H$. Let $G_n$ be the graph  $\bigcup_{i\leq n} P_i$ with $A:=A(G)$ marked. Choose a (marked) minor embedding $h_n: G_n < H_n$, where $H_n\subset H$ is finite and has $A':=A(H)$ as its marked vertex set. 


We will use \seq{h}\ to construct a model of \g in $H$ using the ideas of the proof of \Lr{UF}, whereby we need to pay special attention to the vertices in $A$, in particular to how their branch sets intersect $H\sm  A'$. 

Since each co-part $P_i \sm A$ is connected, it is mapped to a co-part of $H$ by each $h_n, n\geq i$. Note that $h_n(x) \cap A'$ is a singleton $\{x'\}$ \fe\ $x\in A$ and $\nin$ by \Or{apices}. Thus this map $x \mapsto x'$ is a bijection from $A$ to $A'$. By passing to a subsequence of $\{h_n\}$ if necessary, we may assume that this bijection is fixed \fe\ \nin. Moreover, we can choose subsequences $\{h_n \} \supseteq \{h^1_n\} \supseteq \{h^2_n\} \ldots$ of $\{h_n \}$ \st\ \fe\ $x\in A$ and every $i\in \N$, the intersection of the branch set $h^i_n(x)$ with the parts $\{P'_1, \ldots, P'_i\}$ is independent of $n$\mymargin{Uses rank=1}; indeed, having chosen $\{h^{i-1}_n \}$, we observe that since $P'_i$ and $A$ are finite, there is an infinite subsequence $\{h^{i}_n \}$ along which $h^{i-1}_n(x) \cap P'_i$ is constant \fe\ $x\in A$.  

Let $h'_n:=h^n_n$. Note that $\{h'_n\}$ is a subsequence of $\{h_n\}$, and that $h'_n(x)$ converges \fe\ $x\in A$, to a connected subgraph $h_A(x)$ of $H$ containing $x'$ and no other vertex of $A'$. This is the beginning of our construction of a minor model of \g in $H$. Let us now embed the vertices in $G \sm A$.

Similarly to the proof of \Lr{UF}, we call a part $P_i$ of $G$ \defi{$h'$-stable}, if $\bigcup_n h'_n(P_i)$ is finite.  Let \seq{S}\ be an enumeration of the $h'$-stable parts of \G, and \seq{U} an enumeration of its other parts. Let $G_S:= \bigcup_\nin S_n $. Again, a standard compactness argument yields a minor embedding $h_S: G_S < H$ \st\ $h_S(S_i \sm A)$ coincides with $h'_n(S_i\sm A)$ for infinitely many values of $n$ whenever $S_i$ exists\mymargin{Uses rank=1}. Note that $h_S$ extends $h_A$ since the latter is the limit of the restriction of $\{h'_n\}$ to $A$. By construction, $h_S=:h^0_S$ is a minor embedding of $G_S$ into $H$.


We continue by following the lines of the proof of \Lr{UF}: for $i=1,2,\ldots$, if $U_i$ exists, we let $\{C^i_n,\nin\}$ be an infinite sequence of distinct parts of $H$ \st\ each $C^i_n\sm A'$ contains $h'_m(U_i\sm A)$ for some $m\in \N$, whereby we used the fact that each $h_m$ maps each co-part of \g to one of $H$. 

Combining the marked-graph version of the \GMT\  \ref{GMT} with \Or{good seqs}, we deduce that \seq{C^i}\ has an infinite subsequence \seq{Y}\ \st\ $Y_r\mm Y_m$ \fe\ $r<m \in \N$, whereby  $A'$ is  the set of marked vertices of each $Y_n$.

From now on we will not need to use the assumption that $\Rank(G)=\Rank(H)=1$; this will be  important later, as the rest of this proof is also used for \Lr{subclasses}. 

As before, we pick $k=k(i)$ such that no $Y_m, m\geq k$ has been nailed yet, and the interesting case is where  $Y_k$ intersects the image of $h^{i-1}_S$.  
In this case, we want to apply Hilbert's Hotel principle again to `shift' the $h^{i-1}_S$-image within each $Y_m, m\geq k$ to $Y_{m+1}$, but we need to be careful with the image of  $A$. 
For this, we start by picking a marked-minor embedding $g_m: Y_m < Y_{m+1}$ \fe\ $m\geq k$. Note that as $A'$ is the set of marked vertices of each $Y_n$, each $g_m$ induces a permutation $\pi_m$ of $A'$. Moreover, by composing consecutive $g_m$'s we obtain marked-minor embeddings $g_{mt}: Y_m < Y_{t}$ \fe\ $t \geq m\geq k$, which again induce permutations $\pi_{mt}$ on $A'$. Applying \Lr{perms} we find a subsequence $(Y'_n)$ of $(Y_n)$ \st\ each of the corresponding permutations $\pi_{mt}$ is the identity. We may assume \obda\ that $Y'=Y$. Thus we can now repeat the idea of \Lr{UF} to shift the $h^{i-1}_S$-image within each $Y_m, m\geq k$ to $Y_{m+1}$ and embed, and nail, $U_i$ to $Y_k$ to obtain $h^{i}_S$. We also nail $S_i$, if it exists, to the part of $H$ containing $h^{i}_S(S_i\sm A)$. As before, the $h^{i}_S$ converge as $i\to \infty$ to a minor embedding of $G$ in $H$. 
\end{proof}

\subsection{Proof of \Cr{cor Seym Intro}}

We now use the above techniques to prove Seymour's self-minor conjecture for rayless graphs (\Cr{cor Seym Intro}), which we restate here for convenience:

\begin{corollary} \label{cor Seym}
\Fe\ infinite rayless graph $G$, there is a proper \minem\ $g: G < G$. 
\end{corollary}

We start with a simple lemma about permutations, which we will apply to permutations of $A(G)$ arising from self-minor models of $G$.

\begin{lemma} \label{perms}
For every sequence \seq{\pi}\ of permutations of a finite set $A$ there is an infinite index set $Y\subseteq \N$ \st\ $\pi_{jk}= Id$ \fe\ $j<k\in Y$, where $\pi_{jk}:= \pi_{k-1} \circ \pi_{k-2} \circ \ldots \circ \pi_{j+1} \circ \pi_j$.
\end{lemma}
\begin{proof}
\Fe\ $k\in \N$ let $S_k:= \pi_{k-1} \circ \pi_{k-2} \circ \ldots \circ \pi_{0}$, let $\pi$ be a permutation of  $A$ that coincides with $S_k$ for infinitely many $k$, and let $Y$ be the set of those $k$ except the least one. 
\end{proof}

\begin{proof}[Proof of \Cr{cor Seym}]
We can find a subgraph $H\subset G$ of rank 1
as follows. If $\Rank(G)=1$ we just let $H:= G=:G_0$. If $\Rank(G)>1$, we pick a co-part $G_1$ of \g with $\Rank(G_1)\geq 1$; such a $G_1$ exists, because if all co-parts of \g have rank 0 then \g has rank 1 by the definitions. We then iterate, with $G$ replaced by $G_1$, to obtain a sequence $G_1,G_2,\ldots$ of subgraphs of \G, each of rank at least 1. Note that $\Rank(G_i)> \Rank(G_{i+1})$ since $G_{i+1}$ is a co-part of $G_{i}$. Thus the sequence terminates because the ordinal numbers are well-ordered, and we let $H$ be the final member $G_k$ of this sequence. Clearly, $\Rank(H)=1$, for the sequence would have continued otherwise.

Let  $A'(H):= A(G) \cup A(G_1) \ldots \cup A(G_{k})$. (We can think of $A'(H)$ as the union of $A(H)$ with all its `parent' vertices.)
Note that $A'(H)$ is finite. 
Let $P_i, {i\in \ci}$ be a (possibly transfinite) enumeration of the parts of $H$, and let $P'_i:= G[A'(H) \cup P_i]$. Apply \Tr{GMT} to \seq{P'} as above (with $A'(H)$ always marked) to find an infinite $<$-chain that fixes $A'(H)$, for which we also use \Lr{perms} below, and then apply the HH principle to form a proper self-minor model of $G[A'(H) \cup \bigcup_{i\in\oo} P_i]$ (note that we are ignoring $P_i$ for $i\geq \oo$ so far). Extend this model to $G-H$, and to $P_i, i\geq \oo$, by the identity map, to obtain the desired $g: G < G$. 
\end{proof}

\comment{
	\subsection{A toy extension to higher rank}

Generalising \Lrs{UF} and \Lr{Rank 1} to higher ranks is much harder in general (and the topic of \Sr{sec subcl}), but there is a class of graphs for which it becomes substantially easier. This subsection, which can be skipped, is about this class. \medskip

We let $\cc^*(G)$ denote the class of marked graphs in $\ran{<\alpha}$ consisting of finite unions of parts $H$ of $G-A(G)$, with $A(G)$ marked.
We say that \g is \defi{\oo-repetitive}, if for each of its parts $H$, there are infinitely many parts of \g isomorphic to $H$.

\begin{lemma} \label{CG}
\Fe\ $\alpha< \oo_1$ and \oo-repetitive graphs $G,H\in \ran{\alpha}$, we have $G<H$ \iff\ $\cc^*(G)\subseteq \cc^*(H)$.
\end{lemma}
\begin{proof}
The forward implication follows again immediately by restricting a minor model of $G$ in $H$ to any $H\in \cc^*(G)$.

For the backward implication, let \seq{P}\ be an enumeration of the parts of \G, and let $G_n:= \bigcup_{i\leq n} P_i$.

Since $H$ is \oo-repetitive, we can \defi{dedicate} each part of $H$ to a unique part of \G\ so that for each $P_i$ there are infinitely many parts of $H$ of each isomorphism class that are dedicated to $P_i$. Then, for each $G_n$, we can choose a minor embedding $f_n: G_n < H$ that embeds each $P_i, i\leq n$ to the union of $A(H)$ with parts of $H$ dedicated to $P_i$. 

Since $A(H)$ is finite, there is an infinite subsequence 
\seq{f^1}\ of \seq{f}\ \st\ $f^1_j(P_1) \cap A(H)$ is the same \fe\ $j\in\N$. We may assume that the $f^1_j(P_1)$ coincide outside $A(H)$ too, because $P_1$ is always mapped to parts of $H$ dedicated to it. Repeating this argument with $P_2, P_3, \ldots$, we obtain subsequences $f^2 \supset f^3, \ldots$, \st\ $\bigcup_n f^n_1(P_n)$ is a minor embedding of $G$ into $H$. 
	\end{proof}
}

\section{Extending \Lr{UF} to Rank $>1$} \label{sec subcl}

Recall that $\ccm(G)$ denotes the class of marked minors of $(G,A(G))$ of lower rank (\Dr{def ccm}). The following is a key lemma for the proof of \Tr{main Intro}, and it generalises \Lr{Rank 1}. 
\begin{lemma} \label{subclasses}
Let $G,H$ be countable graphs with $\Rank(H)= \Rank(G)$ for some ordinal  $\alpha< \oo_1$. 
Assume \mymargin{first assumption used for \Lr{lem redundant}, second for  \Lr{lem Gn}.} $\ccm(H)$ is \wqo, 
and $|\ranm{\beta}\cap \ccm(G)|_{\mm}$ is countable \fe\ $\beta<\alpha$. Then we have $G<H$ \iff\ $\ccm(G)\subseteq \ccm(H)$.
\end{lemma}

Compared to \Lrs{UF} and~\ref{Rank 1}, this statement removes the condition $|A(G)|= |A(H)|$, and imposes two additional conditions in order to be able to handle ranks $\alpha>1$. To appreciate the role of these conditions, recall that when we used \Lr{UF} to prove \Cr{UF ctble}, we used the \GMT\ and the fact that there are countably many isomorphism types of finite graphs. We do not have analogous statements for higher ranks, and so our condition that $\ccm(H)$ is \wqo\ replaces the former, and the condition $|\ranm{\beta}\cap \ccm(G)|_{\mm} = \aleph_0$ replaces the latter. Note that both conditions are about graphs of lower rank than that of $G,H$, which will allow us to apply \Lr{subclasses} within inductive arguments. 

\smallskip
\Lr{subclasses} is an important reason why we are forced to consider marked graphs even though we are mainly interested in unmarked ones: 
\begin{remark} \label{rem ccm}
We cannot replace $\ccm(G)$ in \Lr{subclasses} by its unmarked version $\cc(G)$,  
as shown by the following example. Let $G_0$ be the disjoint union of the finite cliques $K_n, n\in \N$. Let $H=S(G_0)$ be its suspension (as defined in \Sr{sec susp}, and $G:= S(H)$. Easily, every finite minor of $G$ is a minor of $H$, and nevertheless $G\not< H$ by \Or{apices} since $|A(G)|=2$ and $|A(H)|=1$. This example also shows that it is important to mark the apex vertices and only those in the definition of $\ccm(G)$. 
\end{remark}

\begin{remark} \label{rem ccm R}
We can also not replace $\ccm(G)$ by its finitary version \defi{$\ccmf(G)$}, i.e.\ the class of finite marked graphs  that are marked-minors of $(G,A(G))$, as shown by the following example. Let $H=S(G_0)$ be as above, and let $H'=S(\omega \cdot H)$ be the suspension over the disjoint union of $\omega$ copies of $H$ (thus $\Rank(H')=2$). \Fe\ \nin, let $S_n:= S(\omega \cdot K_n)$. Let $G':=S(\bigcup_{\nin} S_n)$ (again $\Rank(G')=2$). Note that $H\not<S_n$ for any $n$, and using this it is not hard to see that $G'\not<H'$. On the other hand, both $\ccmf(G'),\ccmf(H')$ consist of all finite graphs with at most one marked vertex.
\end{remark}


We prepare the proof of \Lr{subclasses} with two lemmas:

\begin{lemma} \label{lem Gn}
Suppose $\Rank(G)=\alpha< \oo_1$, and $|\ranm{\beta}\cap \ccm(G)|_{\mm}$ is countable \fe\ $\beta<\alpha$. Then there is a sequence  $G_1 \subset G_{2} \subset \ldots$ of subgraphs of \G, containing $A:=A(G)$, \st\ 
\begin{enumerate}
\item \label{Gn ii} $\Rank(G_i)< \Rank(G)$ \fe\ $i \in \N$; and
\item \label{Gn iii} \fe\ $G' \subset G$ containing $A$ with $\Rank(G')< \Rank(G)$ \ti\ $n$ \st\ $(G',A) \mm (G_n,A)$. 
\end{enumerate}
\end{lemma}
\begin{proof}
By our assumption, the family of marked graphs $$\{(G',A) \mid A \subset G' \subset G, \Rank(G')< \Rank(G)\} \subseteq \ccm(G)$$ decomposes into countably many marked-minor-twin classes (whereby we use the fact that there are countably many ordinals $\beta< \alpha$). Enumerate these classes as \seq{\cc}, and pick a representative $G'_n \subset G$ from each $\cc_n$. Then $G_n:= \bigcup_{j\leq n} G'_j$ has the desired properties. 
\end{proof}

(We can achieve $G = \bigcup_\nin G_n$ if desired, by adding $P_n$ to $G'_n$, but we will not need this.) \smallskip

Given two graphs $G,H$ of the same rank satisfying $|A(G)|= |A(H)|$ (for example 
we could have $G=H$), and a \minem\ $h: G < H$, note that $h$ induces a bijection from $A(G)$ to $A(H)$ by \Or{minor rank}. 
\begin{definition} \label{def hA} 
We denote this bijection by \defi{$h^A$}.
\end{definition}

Call a part $P$ (as in \Dr{def part}) of $H\in \ran{\alpha}$  \defi{$H$-unstable}, if there is a sequence \seq{h}\ of minor embeddings $h_n: H < H$ \st\ each $h_n$ maps $P$ into a different part of $H$ and $h_n^A$ is the identity. Otherwise, we say that $P$ is \defi{$H$-stable}.

\begin{lemma} \label{lem redundant}
Let $H$ be a rayless graph \st\ $\ccm(H)$ is \wqo. Then at most finitely many parts of $H$ are $H$-stable.
\end{lemma}
\begin{proof}
Suppose not, and let \seq{P}\ be an enumeration of the infinitely many $H$-stable parts of $H$. Since $\ccm(H)$ is \wqo, \seq{P}\ is good, and so by \Or{good seqs}, there is an infinite chain $P_{a_1} \mm P_{a_2} \mm \ldots$. Let $h_i:  P_{a_i} \mm P_{a_{i+1}}$ be corresponding minor embeddings, and note that $h_i$ induces a permutation  $\pi_i$ on $A(H)$ by \Or{apices}. By \Lr{perms} we may assume, by passing to a subsequence if necessary, that each $\pi_i$ is the identity on $A(H)$. Combining this with the Hilbert Hotel Principle as in the proof of \Lr{Rank 1} we can define $h: H\mm H$ 
\st\ $h(P_{a_i}) \subseteq P_{a_{i+1}}$ \fe\ $i$, and $h^A$ is the identity. Let  $h_n, \nin$ be the composition of $h$ with itself $n$ times. Then \seq{h}\ witnesses that $P_{a_1}$ is $H$-unstable, contradicting our assumption.
\end{proof}

\begin{remark} \label{rem redu}
Call a part $P$ of $H\in \ran{\alpha}$ \defi{redundant}, if $H< H \sm P^c$, where $P^c:= P \sm A(H)$ denotes the co-part of $P$. Then \Lr{lem redundant} remains true if we replace `$H$-stable' by `irredundant'.
\end{remark}

We can now prove the main result of this section.

\begin{proof}[Proof of \Lr{subclasses}]
Again, the forward implication is trivial. \mymargin{Assuming $\Rank(H)= \Rank(G)$}
For the backward implication we will follow the approach of \Lr{Rank 1}, and it is assumed that the reader is familiar with its proof. The main technical difficulty in comparison to that lemma will be handling the stable parts, and those co-parts using vertices of $A'$ in their branch sets. 
\medskip

Let \seq{P}\ be an enumeration of the parts of \G.
Let \seq{G}\ be a sequence of subgraphs of \g as provided by \Lr{lem Gn}. We may assume that $P_n \subseteq G_n$, because we can add $\bigcup_{i\leq n} P_i$ to $G_n$ without increasing its rank. Let $P_i^c:=P_i\sm A$ be the co-part of $P_i$. 

Let $A:=A(G)$ and $A':=A(H)$. 
%
%
Since $\ccm(G)\subseteq \ccm(H)$, \ti\ a sequence of marked-minor embeddings $h_n: (G_n,A) \mm (H,A')$. 

We say that $P_i$ \defi{uses} $x\in A'$ in $h_n$, if $x\in h_n(P_i^c)$, and we say that $P_i$ \defi{uses $A'$} in $h_n$ if it uses some $x\in A'$. We are going to use a subsequence of $(h_n)$ whose members treat $A'$ in a similar way; to make this precise, given $f: (G_m,A) \mm (H,A')$, we define its \defi{$A'$-signature} $\sigma_f: A' \to X$, where $X:= A \cup \{\emptyset\}$, as follows. For each $x\in A'$ used by some $a\in A$ in $f$, we let  $\sigma_f(x)=a$. For each other $x\in A'$ we let $\sigma_f(x)=\emptyset$.

\comment{
	\begin{align*}
&\sigma_f(x)=v, \text{ if $v\in A$;}\\
&\sigma_f(x)=P, \text{ if $v$ lies in $P^c$ for some $P\in \ca_g$, and }\\
&\sigma_f(x)=\emptyset \text{ otherwise.}
\end{align*}
If $x$ is not used by any vertex in $f$, then we also 	let $\sigma_f(x)=\emptyset$.
}

Since the target $X$ of $\sigma_{h_n}$ is finite, there is an infinite subsequence of $(h_n)$ along which $\sigma_{h_n}$ is a constant map $z$, and we may assume \obda\ that this subsequence coincides with $(h_n)$. We call any such sequence \seq{h}\ a \defi{$z$-sequence}.

\medskip
In \Lr{Rank 1} we had $|A|= |A'|$, which ensured that $z$ was a bijection between $A,A'$, and every co-part $P_i^c$ of \g was mapped to a co-part of $H$ by each $h_n$, but neither of this needs to be the case now. Therefore, we will need a more delicate definition of stable and unstable parts, which we prepare with the following notion.

We say that $P_i$ is \defi{$h$-annoying}, if $P_i^c$ uses $A'$ in almost all $h_n$. Let $\ca=\ca_h$ denote the set of  $h$-annoying parts of \G. Note that 
\labtequ{ca fin}{$|\ca_h|\leq |A'|$,}
because each $x\in A'$ is used by at most one co-part in each \minem.  Easily, by passing to a subsequence, we may assume that \fe\ $P_i\not\in \ca_h$,  
\labtequ{Pic}{$P_i^c$ uses $A'$ in at most finitely many $h_n$.}
Indeed, we can construct a subsequence $(h')$ of $h$ with this property recursively as follows. Let $h^0_n:= h_n$, and for $i= 1,2,\ldots$, let $h^i_n$ be a subsequence of $h^{i-1}_n$ in which $P_i$ never uses $A'$ if $P_i\not\in \ca_{h^{i-1}}$. If  $P_i\in \ca_{h^{i-1}}$, just let $h^i=h^{i-1}$. Note that $h^i$ may have annoying parts that were not annoying in $h^{i-1}$, but they are boundedly many by \eqref{ca fin}. 

We may assume, by passing to a co-final subsequence of $h^i$, that each $h^i$-annoying part $P$ uses $A'$ in $h^i_n$ for every $n$. Note that this means that $P$ uses $A'$ for every $n$ in any subsequence of $h^i_n$. Since each $h^i$ is a subsequence of $h^{i-1}$, this implies that $\ca_{h^i}$ is increasing with $i$, and therefore its limit $\ca$ is a finite set of parts of $G$.

Let $h'_n:= h_n^n$. Note that $\ca_{h'}= \ca$, and therefore \eqref{Pic} is satisfied if we replace $h$ by $h'$ (and $\ca_h$ by $\ca$). Since $h'$ is a subsequence of $h$, we will from now on just assume $h=h'$ to keep the notation simpler.
\medskip

We now adapt the definition $h$-stable from \Lr{UF} to the non-annoying parts of any $z$-sequence  \seq{g}. We call a part $P\not\in\ca_g$ of \g \defi{$g$-stable}, if there is a subgraph $H'$ of $H$ consisting of finitely many parts of $H$ with the following property: for  all $n$ \st\ $g_n(P)$ does not use $A'$, we have $g_n(P) \subset H'$. Call $P\not\in\ca_g$ \defi{$g$-unstable} otherwise. Note that if $P$ is $g$-unstable, then 
\labtequ{unstab C}{there is an infinite sequence \seq{C}\ of distinct parts of $H$ and a strictly increasing sequence \seq{\tau}\ \st\ $g_{\tau(n)}(P^c) \subseteq C_n^c$ \fe\ $n$.}
This is because $g_n(P^c)$ is contained in a co-part of $H$ by \Lr{min comps} whenever $P$ does not use $A'$. 

We claim that there is $z$-sequence \seq{g}\ maximizing the set of unstable parts in the following sense: 
 \labtequ{unstab}{There is a $z$-sequence \seq{g}, $g_n: G_n \to H$, such that every part $P$ of \g that is $f^P$-unstable \wrt\ some $z$-sequence \seq{f^P} is also $g$-unstable.
}
To see this, enumerate those parts $P$ as \seqi{U}, and form $g_n$ by picking infinitely many members from each sequence $(f^{U_i})$, assuming \obda\ that $f^{U_i}_n(U_i^c)$ lie in distinct co-parts of $H$ for different values of $n$, which we can by \eqref{unstab C}, and in particular $f^{U_i}_n(U_i^c)$ never uses $A'$. Note that no $U_i$ can be $g$-annoying, since infinitely members of ($f^{U_i})$ are also members of $g$. This proves \eqref{unstab}.

 \medskip
Fix $(g_n)$ as in \eqref{unstab}, and let $\cs$ be the set of $g$-stable parts of $G$, and $\cu$ the set of all other parts of \G. Thus $\{P_n \mid \nin\} = \ca_g \ \dot{\cup}\ \cs \ \dot{\cup}\ \cu$.

Let $\cs'$ be the set of $H$-stable parts of $H$ ---as defined before \Lr{lem redundant}---  and $\cu'$ the set of all other parts of $H$. Let $H_S:= \bigcup \cs'\subset H$. By \Lr{lem redundant}, 
 \labtequ{csfin}{$\cs'$ is finite, and therefore $\Rank(H_S)<\Rank(H)$.}
 %

Next, we claim that
\labtequ{gn cs}{\fe\ $P\in \cs$, and almost every $n$ \st\  $g_n(P)$ does not use $A'$, we have $g_n(P)\subset H_S$.}
%
Suppose this fails, which means that there is some $P\in \cs$, and an infinite strictly increasing sequence \seq{\tau}, \st\ $g_{\tau(n)}(P)$ does not use $A'$ and $g_{\tau(n)}(P^c)$ is contained in an element $U_n$ of $\cu'$ \fe\ $n$. Since $P$ is $g$-stable, $\bigcup_n g_{\tau(n)}(P)$ meets only finitely many parts of $H$, and so we may assume that these $U_n$ coincide with a fixed $U\in U'$. Let $\seq{f}, f_n: H \mm H$ be a sequence of \minem s witnessing that $U$ is $H$-unstable, i.e.\ the $f_n$ embed $U$ into infinitely many distinct parts of $H$ and $f_n^{A}$ is the identity on $A'$ \fe\ $n$. Then the sequence of compositions $f_n \circ  g_{\tau(n)}$ embed $P$ into infinitely many distinct parts of $H$, and so $P$ is $(f \circ  g \circ \tau)$-unstable. 
But this contradicts our choice of $g$ because of \eqref{unstab}, since $P$ is $g$-stable and $(h \circ  g \circ \tau)$ is a $z$-sequence. This contradiction proves \eqref{gn cs}.

\medskip
Let $G_S:= \bigcup (\cs \cup \ca_g)$. We now use \eqref{gn cs} to prove that
\labtequ{RGS}{$\Rank(G_S)< \Rank(H) (= \Rank(G))$.}
Suppose to the contrary $\Rank(G_S)=\Rank(G)$, which is larger than $\beta:= \Rank(H_S)$ by \eqref{csfin}. 

Let $A'':= A(H) \cup \bigcup_{Q\in \cs'} A(Q)$  (\fig{figQs}), and note that $A''$ is finite by \eqref{csfin}, and that 
\labtequ{lt beta}{each component of $H_S - A''$ has rank smaller than $\beta$}
by the definition of $A(Q)$. 

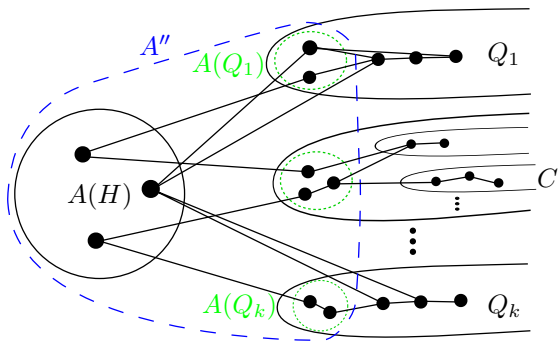
\begin{figure} 
\begin{center}
\begin{overpic}[width=.6\linewidth]{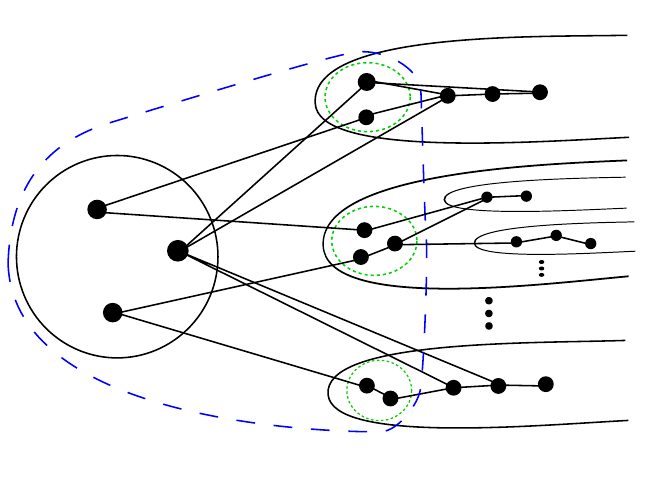} 
\put(88,59){$Q_1$}
\put(88,13){$Q_k$}
\put(34.5,57){\textcolor{green}{$A(Q_1)$}}
\put(36.4,13){\textcolor{green}{$A(Q_k)$}}
\put(12,33){$A(H)$}
\put(25,60){\textcolor{blue}{$A''$}}
\put(97,36){$C$}
\end{overpic}
\end{center}
\caption{The vertex set $A'':= A(H) \cup \bigcup_{Q\in \cs'} A(Q)$ in the proof of \eqref{lt beta}, enclosed by the dashed curve (blue).} \label{figQs}
\end{figure}

Since we are assuming $\Rank(G_S)>\beta$, there are infinitely many $P\in (\cs \cup \ca_g)$ with $\Rank(P^c)\geq \beta$ by \Prr{inf bet}. As $\ca_g$ is finite by \eqref{ca fin}, there are infinitely many such $P$ in $\cs$. Let us choose $|A''|+1$ of them, and denote them $S_0,  \ldots, S_{|A''|}$. By \eqref{gn cs} we can pick $m$ \leth\ \fe\ $i\in [0,\ldots, |A''|]$, either $g_m(S_i^c)$ uses $A'$ or $g_m(S_i) \subseteq H_S$. By the pigeonhole principle, \ti\ $k\in [0,\ldots, |A''|]$ \st\ $g_m(S_k^c)$ does not use $A''$ ---and hence $A'$--- because the $S_i^c$ are pairwise disjoint, and so at most one of them can use any fixed vertex of $A''$. But then $g_m(S_k^c)$ is contained in $H_S$ by the choice of $m$, and in fact in a component $C$ of $H_S - A''$ since $S_k^c$ is connected and it avoids $A''$. Thus $S_k^c<C$, which contradicts \Or{minor rank} since $\Rank(S_k^c)\geq \beta > \Rank(C)$ by \eqref{lt beta}. This contradiction proves \eqref{RGS}.

\medskip


From \eqref{RGS} we deduce $(G_S,A) \in \ccm(G)$, and so $(G_S,A) \mm (G_\ell,A)$ for some $\ell$,  and therefore $g_n$ induces a marked-minor embedding $(G_S,A) \mm (H,A')$ \fe\ large enough $n$. Pick such an $n$, and let $g_S:= g_n$. 
(We do not claim that $g_S(G_S)$ is contained in $H_S$, or that $G_S$ consists of finitely many parts of \G.)

\medskip
From now on we can proceed as in the proof of \Lr{Rank 1} to extend $g_S$ to $h: G<H$ (and in fact $h: (G,A) \mm (H,A')$, though we will not use this) by recursively applying the Hilbert Hotel Principle to the $i$th $g$-unstable part $U_i$ of \G. The main difference is that instead of appealing to the \GMT, we now use our assumption that $\ccm(H)$ is \wqo---combined with \Or{good seqs}--- in order to find a chain \seq{Y}\ within any sequence \seq{C^i}\ of parts of $H$ coming from \eqref{unstab C} applied to $U_i$. The only other difference is that rather than enumerating the parts in \cs, and nailing the $i$th one in step $i$, we now enumerate the co-parts of $H$ intersecting $G_S(\cs \cup \ca_g)$ as \seq{P'}, and ensure that the `contents' of $P'_i$ are never `shifted' after step $i$.  
\end{proof}

\section{Countability of minor-twin-types of rayless forests} \label{sec forests}

In this section we prove that \fe\ ordinal $\alpha< \oo_1$ there are countably many minor-twin classes of countable forests of rank $\alpha$ (\Tr{thm trees}). This is an important step towards the proof of the backward direction of \Tr{thm Borel Intro}, which we will conclude in the next section. Our proof relies on Thomas' theorem that the class \defi{$TW(t)$} of countable graphs of tree-width at most $t$ are \wqo\ \fe\ $t\in \N$ \cite[(1.7)]{ThoWel}. Although we are mainly interested in forests, our proof employs an intricate inductive argument for which it is essential to consider $TW(t)$ \fe\ $t$. For a similar reason, we have to consider marked graphs even though our focus is on unmarked ones. Our final result will apply to $TW(t)$, not just the forests. 

We prepare our proof with two lemmas. The first extends Thomas' aforementioned result to marked graphs. 

\begin{lemma} \label{lem TW m}
For every $t\in \N$, the class of marked graphs $(G,M)$ with $G\in TW(t)$ is \wqo\ under $\mm$.
\end{lemma}
This is easily proved by replacing each marked vertex by a complete graph of the right size: 
\begin{proof}
Let $((G_i,M_i))_{i\in\N}$ be a sequence of marked graphs in $TW(t)$. We need to show that it is good. Easily, we can assume that each $G_i$ has more than $t+2$ vertices.
Let $(G'_i,M'_i):= S^\bullet(S^\bullet((G_i,M_i)))$ ---the marked suspension as defined in \Sr{sec susp}---  and note that $G'_i$ is 2-connected, and $TW(G'_i)\leq t+2$. Applying \Lr{lem cones} twice, we deduce that $((G_i,M_i))_{i\in\N}$ is good if $((G'_i,M'_i))_{i\in\N}$ is, and so it remains to prove the latter. 

Next, for each $i$, we modify $(G'_i,M'_i)$ into an unmarked graph $G''_i$ of tree-width $t+3$ by attaching a copy of $K_{t+4}$ (which has tree-width $t+3$) to each $v\in M'_i$ by identifying $v$ with an arbitrary vertex of $K_{t+4}$. We claim that \fe\ $i,j\in \N$, 
\labtequ{cla Gi}{$G''_i< G''_j$ \iff\ $(G'_i,M'_i) \mm (G'_j,M'_j)$.}
This claim implies our statement, because $(G''_i)_{i\in\N}$ is good by Thomas' aforementioned theorem, implying that $((G'_i,M'_i))_{i\in\N}$ is good too.

The backward implication of \eqref{cla Gi} is trivial (and not needed for our proof). For the forward implication, suppose \cb\ is a minor model of $G''_i$ in $G''_j$. Since $G'_i$ is a block of $G''_i$, \Lr{blocks} says that \cb\ can be restricted into a minor model $\cb'$ of $G'_i$ within a block of $G''_j$. The latter block can only be $G'_j$, because $|G'_i|> t+4$ by our assumption, and all other blocks of $G''_j$ are copies of $K_{t+4}$. Moreover, applying \Lr{blocks} to each other block $B$ of $G''_i$, which is a $K_{t+4}$, we deduce that for each $v\in V(B)$, the branch set $B_v$ contains a vertex of a single block $B'$ of $G''_j$. This $B'$ must be a $K_{t+4}$ because $TW(B)=t+3> TW(G'_j)$. Thus \cb\ induces a 1-1 correspondence between the $t+4$ vertices in $B$ and the $t+4$ vertices in $B'$. It follows that the branch set of the unique vertex in $V(B) \cap M'_i$ contains the unique vertex in $V(B') \cap M''_j$. This means that $\cb'$ maps each marked vertex of $G'_i$ to a branch set containing a marked vertex of $G'_j$, which proves \eqref{cla Gi}.
\end{proof}

For our proof of \Tr{thm Borel Intro} in the next section we will need to prove that there are only countably many minor-twin types of countable forests of any given rank $\alpha< \oo_1$. The ideas involved are similar to the proof of \Cr{UF ctble}, except that instead of $\cc(G)$ we will be working with $\ccm(G)$, and therefore with marked graph. Moreover, instead of the \GMT, we now need to use Nash-Williams' theorem \cite{NWbqo} that the countable trees are \wqo. We have made one step towards adapting our tools to marked graphs with \Lr{lem TW m}, but we will need more: using the same unmarking idea as in the previous proof, we will next upper-bound the number of minor-twin classes of {\em marked} forests, and more generally graphs in $TW(t)$, by the number of minor-twin classes of {\em unmarked} graphs in $TW(t+1)$: 

\begin{lemma} \label{lem TW car}
For every ordinal $\alpha< \oo_1$, and every $t\in \N_{>0}$, we have\\ $|\ranm{\alpha} \cap TW(t)|_{\mm} \leq |\ran{\alpha} \cap TW(t+1)|_<$.
\end{lemma}
This lemma is the reason why it does not suffice to use Nash-Williams' theorem about forests, and instead we need Thomas' extension to $TW(t)$.
\begin{proof}
Pick a representative $(G_X,M_X)$ from each marked-minor-twin class $X$ of $\ranm{\alpha}  \cap TW(t)$, and use it to define the unmarked graph\\ $G'_X\in \ran{\alpha}  \cap TW(t+1)$ similarly to the proof of \Lr{lem TW m}, by attaching a copy of $K_{t+2}$ to each marked vertex and unmarking it.

Suppose $G_X,G_Y$ have the same number of marked vertices, say $n$. Similarly to  \eqref{cla Gi}, we claim that 
\labtequ{cla GX}{$G'_X< G'_Y$ \iff\ $(G_X,M_X) \mm (G_Y,M_Y)$.}
From this we immediately deduce that $G'_X,G'_Y$ are not minor-twins unless $X=Y$. This implies that each $n\in\N$ contributes at most $|\ran{\alpha}  \cap TW(t+1)|_<$ to the count of classes in $\ranm{\alpha} \cap TW(t)$, and since the former is infinite, we obtain the desired inequality.

Thus it only remains to check \eqref{cla GX}. The proof is similar to that of \eqref{cla Gi}, except that we now only apply \Lr{blocks} to the $K_{t+2}$'s of $G'_X$: if \cb\ is a minor model of $G'_X$ in $G'_Y$, then it maps each copy of $K_{t+2}$ in $G'_X$ to one in $G'_Y$, whereby we use the fact that $G_X$ has no $K_{t+2}$ minor as its tree-width is less than that of $K_{t+2}$. It follows easily from this that for each $v\in M_X$ the branch set $\cb(v)$ contains a vertex of $M_Y$. Moreover, if $w\in V(G_X) \sm M_X$,  then   $\cb(w)$ cannot intersect $G'_Y \sm G_Y$, because all vertices in the latter subgraph are needed to accommodate the copies of $K_{t+2}$ in $G'_X$. Thus by restricting \cb\ to $G_X$ we obtain a marked-minor model of $(G_X,M_X)$ in $(G_Y,M_Y)$, proving \eqref{cla GX}.
\end{proof}



We are now ready to prove the main result of this section:

\begin{theorem} \label{thm trees}
For every ordinal $\alpha< \oo_1$, and every $t\in \N_{>0}$, we have\\ $|\ranm{\alpha}  \cap TW(t)|_{\mm} = \aleph_0$.
\end{theorem}
\begin{proof}
We will prove the unmarked version 
\labtequ{Ra TW}{$|\ran{\alpha} \cap TW(t)|_< = \aleph_0$}
by a simultaneous (transfinite) induction on $\alpha$ and $t$. From this the statement  follows immediately from \Lr{lem TW car}.

For $\alpha=0$, our claim \eqref{Ra TW} is trivially true \fe\ $t$ as there are countably many (marked) finite graphs; here we do not need the restriction on the tree-width.

For the inductive step $\alpha>0$, we proceed by induction on $t$. Let \defi{$\ran{\alpha}^n$} be the set of those $G\in \ran{\alpha}$ with $|A(G)|=n$. To prove that $|\ran{\alpha} \cap TW(t)|_<$ is countable, it suffices to prove that $|\ran{\alpha}^n  \cap TW(t)|_<$ is countable \fe\ \nin. Let $G\in \ran{\alpha}^n(TW(t))$, and apply \Lr{subclasses} to deduce that the minor-twin class $[G]_{\mm}$ is determined by $n=|A(G)|$ and the class $\ccm(G)$; here, to be able to apply \Lr{subclasses}, we use the fact that $\ccm(G) \subseteq \ranmn{<\alpha}  \cap TW(t)$ is \wqo\ by \Lr{lem TW m}, and our inductive hypothesis that $|\ranm{\beta}  \cap TW(t)|_{\mm}$ is countable \fe\ $\beta<\alpha$. 

Since $\ccm(G)$ is \wqo, we can write $\ccm(G)$ as $\forb{X} \cap \ranmn{<\alpha}\cap TW(t)$ for some finite $X \subset \ranmn{<\alpha} \cap TW(t)$, whereby we used the fact that $TW(t)$ is minor-closed. 
We claim that $|\ranmn{<\alpha} \cap TW(t)|_{\mm}$ is countable. Indeed, \fe\ $\beta< \alpha$, \Lr{lem TW car} yields \\ $|\ranm{\beta}  \cap TW(t)|_{\mm} \leq |\ran{\beta}  \cap TW(t+1)|_<$, which is countable by our inductive hypothesis on $\alpha$. Taking the union over all $\beta<\alpha$, which are countably many as $\alpha<\omega_1$, establishes our claim that $|\ranmn{<\alpha} \cap TW(t)|_{\mm}$ is countable. Thus there are countably many ways to choose its finite subset $X$ from above, hence countably many ways to choose $\ccm(G)$, and therefore $[G]_<$. This proves \eqref{Ra TW}.
\end{proof}

\section{Proper minor-closed classes of rayless forests are Borel}  \label{sec Borel}

In this section we use \Tr{thm trees} to prove \Tr{thm Borel Intro}, which we restate for convenience: 

\begin{theorem} \label{thm Borel}
Let $\ct\subset \cg$ be a minor-closed family of $\N$-labelled rayless forests. Then \ct\ is Borel \iff\ it does not contain all rayless forests.
\end{theorem}

Recall that \cg\ denotes the space of graphs $G$ with $V(G)=\N$, which we call \defi{$\N$-labelled} graphs, encoded as functions from $\N^2$ to $\{0,1\}$ representing the edges, endowed with the  product topology.  Our condition that $\ct\subset \cg$ be  \defi{minor-closed} here means that if $G\in \ct$ and $H\in \cg$ is a minor of $G$ (according to the standard definition for unlabelled graphs as in \Sr{sec finite}), then $H\in \ct$. 

The reader will not need to know much about the topology of \cg\ in order to understand this section; the connection between minor-closed families and Borel sets is established by the following result that we will use to prove \Tr{thm Borel}:

\begin{theorem}[{\cite{GeoGreCom}}]\label{thm min Bor}
Let \g be a countable graph, and let $\cc\subset \cg$ denote the set of countable $\N$-labelled graphs that are isomorphic to a minor of \G. Suppose $\cc(G)=\forb{\cx} \cap \cg$ for a countable set  \cx\ of (unlabelled) graphs. Then \cc\ is a Borel subspace of \cg. 
\end{theorem}

Apart from this, the only topological statement that we will need in our proof is the fact that any countable union of Borel sets is itself Borel. 

We will also need the following lemma, which is perhaps well-known when restricted to trees, but we include a proof for completeness.

\begin{lemma} \label{lem exc T}
For every countable rayless tree $T$ \ti\ an ordinal $\alpha(T)< \oo_1$ \st\ $T<G$ holds \fe\ graph $G$ with $\Rank(G)> \alpha(T)$.
\end{lemma}
\begin{proof}
We will state a modified statement that will help us apply transfinite induction on $\Rank(T)$.
A \defi{rooted tree} $(T,r)$ is a tree $T$ with one of its vertices $r$ designated as the root. The \defi{tree-order $\leq_r$} on $V(T)$ is the partial order defined by setting $x \leq_r y$ for any two vertices $x,y$ \st\ $x$ lies on the unique path in $T$ from $r$ to $y$. Given rooted trees $(T,r), (T',r')$, we write $(T,r) \preceq (T',r')$ if there is a subgraph embedding $h$ of $T$ into $T'$ that respects the tree-order, i.e.\ $x \leq_r y$ implies $h(x) \leq_{r'} h(y)$ \fe\ $x,y \in V(T)$. We will prove: 
\labtequ{stronger}{For every countable rayless rooted tree $(T,r)$ \ti\ an ordinal $\alpha=\alpha(T)< \oo_1$ \st\ $(T,r) \preceq (T_\alpha, r_\alpha)$ holds,}
where $T_\alpha$ denotes the minimal tree of Rank $\alpha$, as provided by \Or{min tree}, with rooted $r_\alpha$ provided in its construction. 

For finite $T$ it is not hard to see, by induction on the size of $T$, that  $\alpha(T)=\oo$ suffices. 

For $\Rank(T)\geq 1$, the inductive hypothesis is easier to apply when $A(T)=\{r\}$, but this need not be the case. Therefore, we introduce the \defi{spread $S(T,r)$} of a rooted tree $(T,r)$, defined as $S(T,r):= \max_{x\in A(T)} d(x,r)$, where $d$ denotes the graph distance. Fixing $\Rank(T)$, we prove \eqref{stronger} by induction on $S(T,r)$ as follows. 

Let $C_1, C_2,\ldots$ be a (possibly finite) enumeration of the components of $T-r$, and root each $C_i$ at the unique neighbour $r_i$ of $r$ in $C_i$. We claim that 
\labtequ{spread}{\fe\ $i$, either $\Rank(C_i)<\Rank(T)$, or $S(C_i,r_i)<S(T,r)$ (or both).}
To see this, suppose first that $C_i \cap A(T)=\emptyset$. Then $\Rank(C_i)<\Rank(T)$ because $C_i$ is contained in a component of $T- A(T)$ in this case. Otherwise, let $A':= C_i \cap A(T)\neq \emptyset$. Then $\Rank(C_i)=\Rank(T)$, and it is not hard to check that $A(C_i)= A'$. Let $x\in A'$ be a vertex realising $S(C_i,r_i)$. Then $d(x,r)=1+d(x,r_i)$, implying the desired $S(T,r)> S(C_i,r_i)$.

Using \eqref{spread}, we can now define $\alpha:= 1+ \sup_{i\in\N} \alpha(C_i)$,  noting that $\alpha(C_i)$ is well-defined by induction on $S(T,r)$, nested inside our induction on  $\Rank(T)$. To start the induction on $S(T,r)$, we note that $S(T,r)=0$ \iff\ $A(T)=\{r\}$, in which case the first possibility always applies in \eqref{spread}, and therefore $\alpha(C_i)$ is well-defined by induction on $\Rank(T)$. 

We claim that $(T,r) \preceq (T_\alpha,r_\alpha)$. To prove this, we map $r$ to $r_\alpha$, and use our inductive hypothesis to embed each $C_i$ into an appropriate component of $T_\alpha- r_\alpha$, rooted at the neighbour of $r_\alpha$, preserving the tree order. The latter is possible because, by the construction of $T_\alpha$, there are infinitely many components of $T_\alpha- r_\alpha$ isomorphic to $T_\alpha(C_i)$ for each $i$, and so we can pick a distinct such component to embed each $C_i$ using our inductive hypotheses.

This proves \eqref{stronger}. Our statement now follows by forgetting the root, and noting that any graph $G$ with $\Rank(G)> \alpha(T)$ has a component $G'$ with $\Rank(G')\geq \alpha(T)$ by the definition of rank, and $G'$ contains $T_\alpha$ as a minor by \Or{min tree}.
\end{proof}

We can now prove the main result of this section:
\begin{proof}[Proof of \Tr{thm Borel}]
If \ct\ is the family $\cf\cgr$ of all $\N$-labelled rayless forests, then it is well-known that it is not Borel (in fact it is co-analytic complete) \cite{Kechris,GeoGreCom}. 

So suppose \ct\ excludes some rayless forest $T$ as a minor. Then \ct\ excludes a rayless tree, obtained from $T$ by adding a vertex and joining it to each component, and so  \Lr{lem exc T} implies that $\ct \subseteq \Rank_\alpha \cap \cf\cgr$ for some ordinal $\alpha< \oo_1$. 

We would like to apply \Tr{thm trees}, but there is a subtlety we need to address: the former result is about unlabelled graphs, while $\ct\subset \cg$.   Therefore, we let $\ct'$ denote the class of (unlabelled) graphs $G$ \st\ \g is isomorphic to a minor of some graph in $\ct$. Note that $\ct'$ is minor-closed in the standard sense, unlike \ct\ which does not contain any graphs with finite vertex set. We still have $\ct' \subseteq \Rank_\alpha$, and combining this with \Tr{thm trees}, we deduce that $\ct'$ consists of countably many minor-twin classes because there are countably many ordinals $\beta \leq \alpha$. Let \seq{\ct}\ be an enumeration of these classes, and pick a representative $F_n$ from each $\ct_n$. For each infinite $F_n$, let $F'_n$ be an element of \ct\ isomorphic to $F_n$. For each finite $F_n$, let $F'_n$ be an element of \ct\ consisting of infinitely many isolated vertices and a finite graph isomorphic to $F_n$. Let $\cc_n\subset \cg$ denote the set of countable $\N$-labelled graphs that are isomorphic to a minor of $F'_n$. Easily,  
\labtequ{ctU}{$\ct= \bigcup \cc_n$,}
because \ct\ is minor-closed, and $\bigcup \cc_n$ contains a minor-twin of each element of $\ct$. Note that, since the countable forests are \wqo\ (\Lr{lem TW m}), $\cc_n=\forb{X_n}$ for a finite set $X_n$ (consisting of $K_3$ and a finite set of forests). Thus each $\cc_n$ is a Borel subset of \cg\ by \Tr{thm min Bor}, and therefore \ct\ is Borel by \eqref{ctU}.
\end{proof}

\begin{remark} \label{rem Bor}
If Thomas' conjecture, or its restriction to rayless graphs, is true, then following the lines of the proof of \Tr{thm Borel}, but using \Tr{main} \ref{M iir} $\leftrightarrow$ \ref{M iii} instead of \Tr{thm trees}, we would deduce that 
if $\ct\subset \cg$ is a minor-closed family of $\N$-labelled rayless graphs that does not contain all rayless forests then \ct\ is Borel.
\end{remark}

\section{From marked to unmarked graphs} \label{unmarking}

Marked graphs play an important role in the proof of \Tr{main Intro}, mainly via \Lr{subclasses}. This section provides two important tools for the former, which essentially allow us to ignore the marking in certain cases: 

\begin{lemma} \label{lem unmark}
For every ordinal $\alpha$, if $\ranm{\alpha}$ has a (infinite) descending chain, then so does $\ran{\alpha}$. 
\end{lemma}

\begin{corollary} \label{cor unmark}
\Fe\ $0\leq \alpha<\omega_1$, we have $|\ran{\alpha}|_<= |\ranm{\alpha}|_{\mm}$.
\end{corollary}

\begin{proof}[Proof of \Lr{lem unmark}]
Suppose \ti\ a \mm-descending chain $G_1 \gneq G_2 \gneq G_3 \gneq  \ldots$ in  $\ranm{\alpha}$, and let $M_i$ denote the set of marked vertices of $G_i$. By \Lr{lem cones} we may assume \obda\ that each $G_i$ is 2-connected, because we may add a couple of (marked) suspension vertices to each $G_i$ without violating any of the relations $G_i \gneq G_j$. Moreover, we may assume that $A(G_i)$ is 2-connected \fe\ $i$, since every suspension vertex lies in $A(G_i)$. Here we use the obvious fact that $\Rank(S(G))=\Rank(G)$ \fe\ graph \G.

Since $(G',M')\mm (G,M)$ implies $|M|\geq |M'|$, we deduce that $|M_i|$ is monotone decreasing. Thus we may assume that it is constant (and at least 1, or there is nothing to prove). Similarly, since $(G',M')\mm (G,M)$ implies 
 $\Rank(G) \geq \Rank(G')$ (\Or{minor rank}), we may assume, by induction on $\alpha$, that $\Rank(G_i)= \alpha$ \fe\ $i$. Letting $A_i:=A(G_i)$, it follows from \Or{apices} that $|A_i|$ is monotone decreasing too, and again we may assume that it is constant. Similarly, we have $|A_i \cap M_i| \geq |A_{i+j} \cap M_{i+j}|$ \fe\ $i,j\in \N$, and so we can also assume that $|A_i \cap M_i|$ is constant. Note that this implies that $|M_i \sm A_i|$ is  constant too, and that
\labtequ{Mtil}{any minor model $\cb$ of $G_{i}$ in $G_{i-j}$ maps each vertex in $M_i \sm A_i$ to a branch set intersecting $M_{i-j} \sm A_{i-j}$.}

We claim that
\labtequ{MinA}{we may assume that $M_i\subseteq A_i$ \fe\ $i$.}
Indeed, if this is not the case, then let $\tilde{M} _i:= M_i \sm A_i$, and for each $x\in \tilde{M} _i$, attach a copy of the minimal tree $T_\alpha$ of Rank $\alpha$, provided by \Or{min tree}, to $x$, by identifying the root of $T_\alpha$ with $x$, to obtain the marked graph $(\tilde{G}_i,M_i\sm \tilde{M} _i)$ \fe\ $i\in \N$. We will show that \seq{\tilde{G}}\ is still a \mm-descending chain: firstly, using \eqref{Mtil} we can extend $\cb$ into a minor model of $\tilde{G}_{i}$ in $\tilde{G}_{i-j}$. This shows that $\tilde{G}_{i-j} \mm \tilde{G}_{i}$. To show that $\tilde{G}_{i} \not\mm \tilde{G}_{i+j}$, suppose $\tilde{\cb}$ is such a minor model. Then by \Lr{blocks}, since $G_{i}\subset \tilde{G}_{i}$ is 2-connected, \ti\ a model $\cb'$ of 
$G_{i}$ in a block of $\tilde{G}_{i+j}$. Since the only block of $\tilde{G}_{i+j}$ is $G_{i+j}$ by construction, we deduce that $G_{i} \mm G_{i+j}$, a contradiction that proves \eqref{MinA}.

Next, we claim that
\labtequ{MeqA}{we may assume that $M_i= A_i$ \fe\ $i$.}
To prove this, we extend each $G_i$ into a supergraph $G'_i$ as follows. For each unmarked $x\in A_i$, add two disjoint copies of $G_i$ to $G_i$, and join $x$ to each of its two copies by an edge (\fig{figCopies}, middle). Having done so for each $x\in A_i \sm M_i$ we have obtained $G'_i$. Note that each block of $G'_i$ is isomorphic to $G_i$ via an isomorphism that preserves the marking. To prove \eqref{MeqA} it suffices to check that $G'_{j-1} \gneq G'_{j}$ still holds \fe\ $j>1$. 

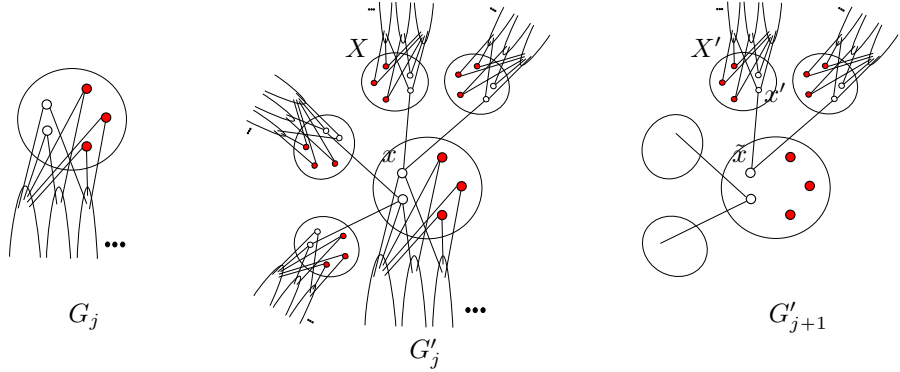
\begin{figure} 
\begin{center}
\begin{overpic}[width=1.0\linewidth]{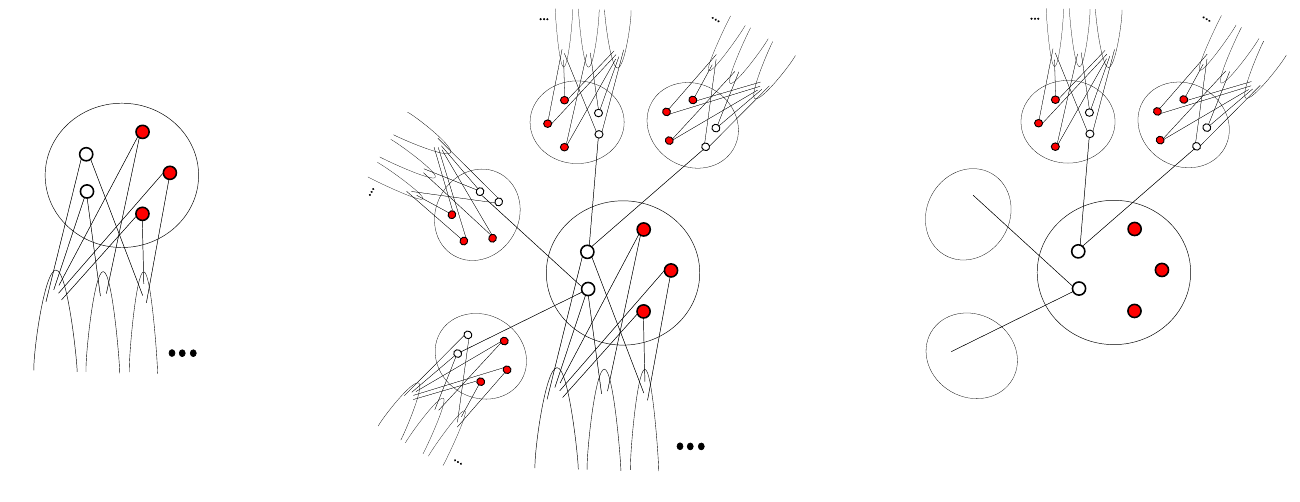} 
\put(9,3){$G_j$}
\put(46,-1){$G'_j$}
\put(85,3){$G'_{j+1}$}
\put(43,20.5){$x$}
\put(39,32){$X$}
\put(76.5,32){$X'$}
\put(81,20.5){$\tilde{x}$}
\put(84.5,27){$x'$}
\end{overpic}
\end{center}
\caption{The graph $G'_j$ in the proof of \eqref{MeqA} (middle), produced by joining copies of $G_j$ (left), and an attempt to embed it into $G'_{j+1}$ (right).} \label{figCopies}
\end{figure}

Let us first check $G'_j \mm G'_{j-1}$. Pick a minor embedding $f: G_{j} < G_{j-1}$. 
Using \eqref{Mtil} we can extend $f$ into a minor embedding of $G'_j$ in $G'_{j-1}$ by embedding each copy of $G_j$ attached to $x$ to the copies of $G'_{j-1}$ attached to the (unique) vertex $x' \in A_{j-1} \sm M_{j-1}$ contained in $f(x)$, by imitating $f$ inside these copies. This proves $G'_j \mm G'_{j-1}$.

To check $G'_{j} \not\mm G'_{j+1} $, suppose to the contrary there is a minor embedding 
$g: G'_j \mm G'_{j+1}$. Recall that, by \Or{apices}, each $g(x), x\in  A(G'_j)$ contains a distinct $x'\in A(G'_{j+1})$. Since $|A_i|$ and $|A_i \cap M_i|$ are constant, it follows that $|A(G'_i)|$ is constant too, hence $|A(G'_j)|=|A(G'_{j+1})|$. Thus there is a bijection  $x\mapsto x'$ from $A(G'_j)$ to $A(G'_{j+1})$.

Recall we are assuming that $A_j$ is 2-connected, and so $g$ maps $A_j\subset A(G'_j)$ so that each branch set intersects a fixed block $B$ of $G'_{j+1}$ by \Lr{blocks}. We claim that this block $B$ must be $G_{j+1}$ (rather than one of its copies in the construction of $G'_{j+1}$). Suppose to the contrary, there is $x\in A_j$ \st\ $x'\not\in A_{j+1}$. Thus $x'$ is a copy of an unmarked vertex of $A_{j+1}$ (\fig{figCopies}, right). Note that the branch set $g(x)$ of $x$ cannot contain the unique neighbour $\tilde{x}$ of $x'$ in $A_{j+1}$, because $x\mapsto x'$ and so $\tilde{x}$ is in the branch set of some other vertex of $A(G'_j)$. Thus $g(x)$ avoids $A_{j+1}$. Since at most one of the two copies of $G_j$ adjacent to $x$ can contain the edge $x'\tilde{x}$, it follows that $g$ maps at least one of these copies $X$ inside the copy $X'$ of $G_{j+1}$ containing $x'$. But this copy avoids $x'$ which is already used by $g(x)$, and we therefore reach a contradiction as we do not have enough vertices in $X'\cap A(G_{j+1})$ to accommodate $X\cap A(G_{j})$.


Thus $g$ maps $A_j$ so that each branch set intersects $G_{j+1}$. Since $G_j \supset A_j$ is 2-connected, each branch set of a vertex of $G_j$ intersects $G_{j+1}$ by the first sentence of \Lr{blocks}. By the second sentence, there is a model of $G_j$ in $G_{j+1}$ respecting the marking, a contradiction that proves \eqref{MeqA}.

\medskip
Recall that, by \Or{apices}, any unmarked minor model $f: G_i < G_j$ is a marked minor model of $(G_i, A(G_i))$ in $(G_j, A(G_j))$. Thus \eqref{MeqA} implies that $G_1 \gneq G_2 \gneq G_3 \gneq  \ldots$ is also an unmarked descending chain in $\ran{\alpha}$. 
\end{proof}

\begin{remark} \label{rem corunm}
In this proof we only used the assumption that the family \seq{G}\ is a descending chain in order to make each of $|A(G_n)|, |M_n|$ and $\Rank(G_n)$ independent of $n$. Using this, we then produced a modified family \seq{G'}\ of unmarked graphs \st\ $G_i \mm G_j$ \iff\ $G'_i < G'_j$ \fe\ $i,j$. This has \Cr{cor unmark} as an important consequence:
\end{remark}

\begin{proof}[Proof of \Cr{cor unmark}]
Easily,  $|\ran{0}|_<= |\ranm{0}|_{\mm}= \aleph_0$ since there are countably many (marked) finite graphs, so assume $\alpha\geq 1$ from now on.

The inequality $|\ran{\alpha}|_< \leq |\ranm{\alpha}|_{\mm}$ is trivial since each $<$-equivalence class of $\ran{\alpha}$ is contained in a distinct $\mm$-equivalence class of $\ranm{\alpha}$.

For the converse inequality, let $((G_i,M_i))_{i\in \ci}$ be a family of marked graphs, one from each $\mm$-equivalence class of $\ranm{\alpha}$. If \ci\ is countable then we are done since $|\ran{\alpha}|_<$ is infinite. If it is uncountable, then there is an uncountable subfamily $((G_i,M_i))_{i\in \cj\subseteq \ci}$ within which each of $|A(G_i)|, |M_i|$ and $\Rank(G_i)$ is constant, because there are countably many choices for each of these numbers. Using  \Rr{rem corunm} we can transform  this subfamily into a family $(G'_i)_{i\in \cj}$ of unmarked graphs of the same rank no two of which are minor twins. This completes the proof since $|\cj|=|\ci|=|\ranm{\alpha}|_{\mm}$.
\end{proof} 

\begin{remark} \label{rem anti}
We cannot adapt \Lr{lem unmark} to antichains, because we cannot keep $|A_i(G)|$ constant. 
\end{remark}

\section{Equivalences for a fixed rank} \label{sec fix rk}

The following is the main technical ingredient of the proof of \Tr{main Intro}; it strengthens it by providing additional equivalent conditions, and refines it by layering graphs \wrt\ their rank. Let 
$\ranmn{<\alpha}:= \{ (G,M) \in \ranm{<\alpha} \mid |M|\leq n\}$. (We do not require $M\subseteq A(G)$ here.)
\begin{theorem}\label{tfae}
The following are equivalent \fe\ ordinal $1\leq \alpha<\omega_1$: 
	\mymargin{
 $\bullet$ $\ranm{<\alpha}$ has no antichain??; 
$\bullet$ $\ranm{<\alpha}$ is \wqo??;
$\bullet$ $\ran{<\alpha}$ is \wqo?;
	}
\begin{enumerate}[itemsep=0.0cm, label=(\Alph*)]
 \item \label{R iip} $\ranmn{<\alpha}$ is \wqo\ \fe\ \nin;
 \item \label{R iii} $|\ran{\alpha}|_<$ is countable; 
 \item \label{R iiim} $|\ranm{\alpha}|_{\mm}$ is countable; 
 \item \label{R iiip} $|\ran{\alpha}|_< < 2^{\aleph_0}$;
 \item \label{R iv} $\ranm{\alpha}$ has no descending chain; 
 \item \label{R ivp} $\ran{\alpha}$ has no descending chain; 
\item \label{R ip} $\ranmn{<\alpha}$ has no antichain \fe\ \nin; 
\end{enumerate}
\end{theorem} 

The proof of \Tr{tfae} is involved, and it is not easy to break its statement up into individual equivalences: not only we prove some of the equivalences by cycling through several items, but also to prove some of the implications we perform an induction on $\alpha$ that requires other implications. 

Before proving \Tr{tfae}, we introduce a way to represent a subclass of $\ranmn{<\alpha}$ by a single, unmarked, graph in $\ran{\alpha}$:

\begin{definition} \label{def GC} 
Given a  marked-minor-closed class $\cc \subseteq \ranmn{<\alpha}$, we define a (unmarked) graph $G_\cc=G_\cc^n$ as follows. 
\begin{enumerate}
\item \label{GC i} \fe\   marked-minor-twin class \ch\ of elements of \cc, we pick a representative $(H,M) \in \ch$, and put \oo\ pairwise disjoint copies of $H$ into $G_\cc$;
\item \label{GC ii} we add a set $A_\cc$ of $n$ isolated vertices to $G_\cc$; and
\item \label{GC iii} for each copy $H_i$ of $(H,M)$ as in \ref{GC i}, and each of the (at most $n$) marked vertices $v$ of $H_i$, we add an edge from $v$ to a distinct element of $A_\cc$. 

\noindent \textnormal{Note that $E(H_i, A_\cc)$ is a complete matching of the set $M(H_i)$ of marked vertices of $H_i$ into $A_\cc$. But $M(H_i)$ is empty for some $H$, and so $G_\cc$ is disconnected. We perform step \ref{GC iii} in such a way that }
\item \label{GC iv} for each possible matching $m$ of $M(H)$ into $A_\cc$, there are infinitely many indices $i$ \st\ $E(H_i, A_\cc)$ coincides with $m$.
\end{enumerate}
\end{definition}

Importantly, we have $\Rank(G_\cc)\geq \Rank(\cc)$ and %
\labtequ{AGC}{$A(G_\cc)= A_\cc$}
 by construction. Moreover, 
\labtequ{RGC}{$\Rank(G_\cc)\leq \alpha$,}
because each component of $G_\cc\sm A_\cc$ belongs to \cc\ and hence to $\ranm{<\alpha}$.

The following lemma shows that, under natural conditions,  $G_\cc$ is `monotone' in \cc. 
\begin{definition} \label{def upto}
We say that \cc\ is \defi{addable up to rank $\alpha$}, if whenever \seq{G}\ is a sequence of graphs in \cc, their disjoint union $G:= \dot{\bigcup}_n G_n$ is also in \cc\ unless $\Rank(G)\geq  \alpha$. 
\end{definition}
For example, suppose $\cc= \ran{<\alpha} \cap \forb{H}$ for some connected graph $H\in \ran{<\alpha}$. Then it is easy to see that \cc\ is addable up to rank $\alpha$, because $H< \dot{\bigcup}_n G_n$ only if $H<G_i$ for some $i$.

\smallskip
\begin{lemma} \label{lem GC}
Let $\cc \subseteq \ranmn{<\alpha}$ and $\cc' \subseteq \ranmm{<\alpha}$ be marked-minor-closed classes for some $n\leq m\in \N$, and suppose both $\cc,\cc'$ are addable up to rank $\alpha$. 
Suppose $\Rank(\cc)=\Rank(\cc')=:\beta\leq \alpha$. Then $G_\cc< G_{\cc'}$ \iff\ $\cc \subseteq \cc'$.
\end{lemma}
We emphasize that the graphs $G_\cc, G_{\cc'}$ are unmarked, even though the classes $\cc,\cc'$ consist of  marked graphs.
\begin{proof}
For the forward implication, suppose there is a minor model $\cb=\{B_v\mid v\in V(G)\}$ of $\g:=G_\cc$ in $G':=G_\cc$. By \eqref{AGC} we have $A:= A(G)=A_\cc$ and $A':= A(G')=A_{\cc'}$. 

Pick $H\in \cc$. We will prove $H\in \cc'$, thus establishing the forward implication. We may assume that $H$ is connected, because if each component of $H$ lies in $ \cc'$ then so does $H$ by the addability of  $\cc'$; indeed, if $H$ has infinitely many components, then the supremum of their ranks is less than $\alpha$ since $H\in \cc$. Let $C$ be a component of $G \sm A$ which is a marked-minor-twin of $H$, which exists by the construction of \G\ and the connectedness of $H$. Recall that \g contains infinitely many pairwise disjoint copies of $C$. Only finitely many of those can have a vertex the branch set of which intersects the finite set $A'$, and so we may assume that $\cb(C)$ avoids $A'$. Thus by \Or{min comps}, $\cb(C)$ ---i.e.\ the submodel of \cb\ induced by $C$--- 
is contained in a component $C'$ of $G'\sm A'$. 
To prove $H\in \cc'$ it suffices to prove
\labtequ{CinCp}{$C\mm C'$,}
since $H< C$ and $\cc'$ is $\mm$-closed. We will prove \eqref{CinCp} by slightly modifying \cb\ and restricting it to $C$. This modification is needed to ensure that each marked vertex of $C$ is mapped to a branch set containing a marked vertex of $C'$.

Let $\{u_1, \ldots, u_n\}$ denote the marked vertices of $C$, and recall that each $u_i$ sends a unique edge $e_i$ to $A$; let $a_i\in A$ denote the other end-vertex of $e_i$  (\fig{figBu}). 


We consider the following two cases \fe\ $u_i$:

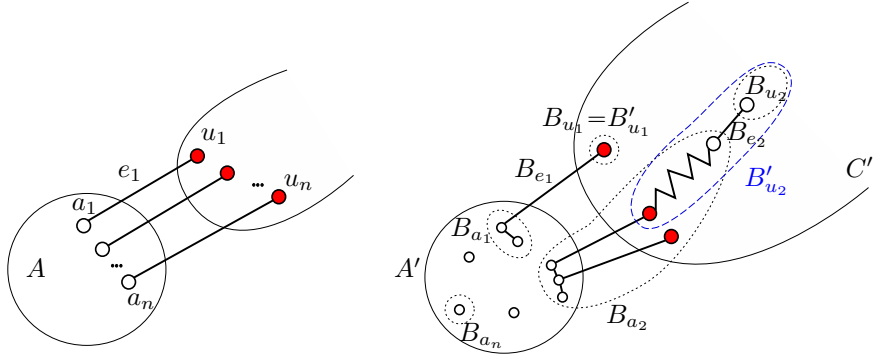
\begin{figure} 
\begin{center}
\begin{overpic}[width=1\linewidth]{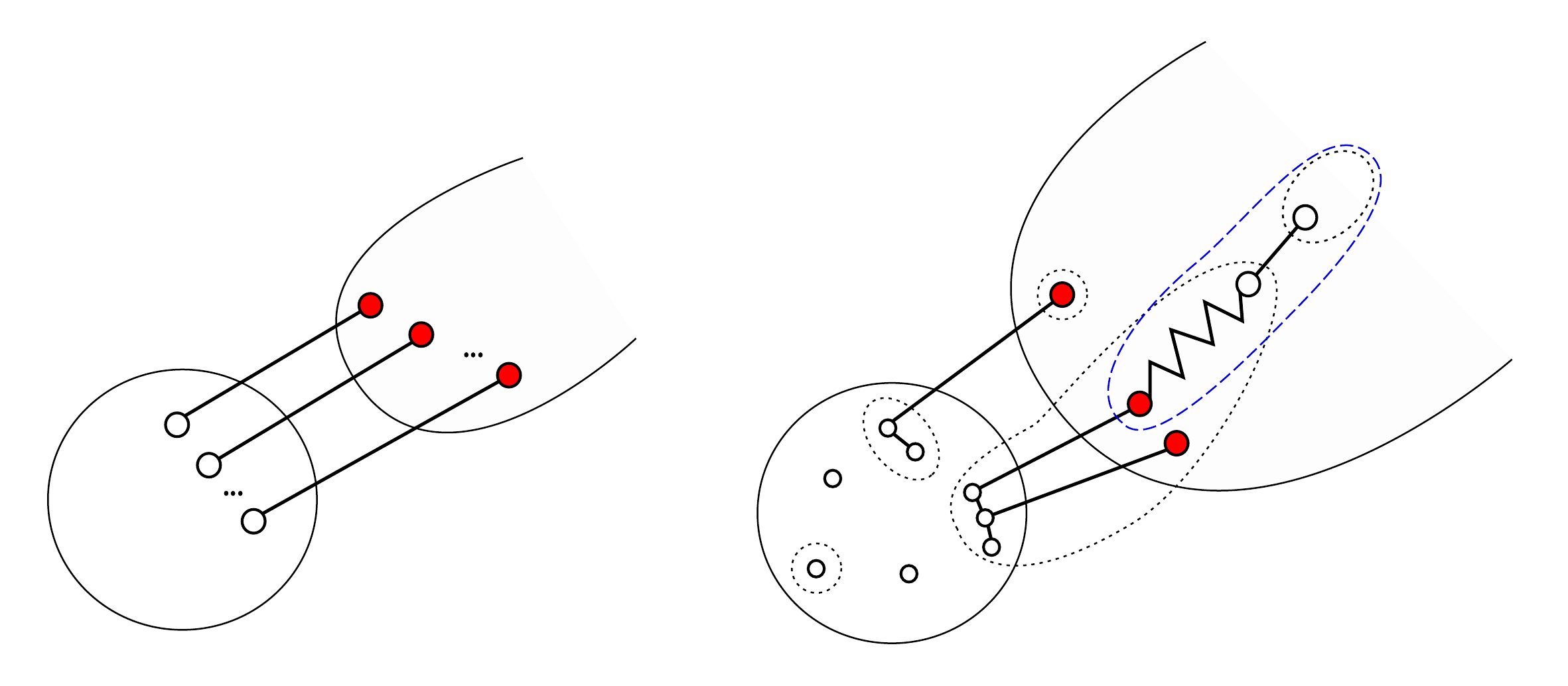} 
\put(5,11){$A$}
\put(45,11){$A'$}
\put(94,22){$C'$}
\put(24,26){$u_1$}
\put(15,22){$e_1$}
\put(33,21){$u_n$}
\put(10,18.4){$a_1$}
\put(16,8){$a_n$}
\put(68,6){$B_{a_2}$}
\put(51.1,15.5){$B_{a_1}$}
\put(52,4){$B_{a_n}$}
\put(83,31){$B_{u_2}$}
\put(83,21){\textcolor{blue}{$B'_{u_2}$}}
\put(61,27.5){$B_{u_1}\mathord{=}B'_{u_1}$}
\put(58,22){$B_{e_1}$}
\put(81,26){$B_{e_2}$}
\end{overpic}
\end{center}
\caption{Defining $B'_{u_i}$ in the proof of \eqref{CinCp}. The left picture depicts a portion of $G$, while the right picture depicts its minor model in $G'$. The dotted curves enclose the original branch sets, while the dashed curve (in blue) encloses $B'_{u_2}$.} \label{figBu}
\end{figure}

If $B_{e_i}$ has an end-vertex $a'\in A'$, then $B_{u_i} \cap B_{e_i}$ must be the unique neighbour of $a'$ in $C'$ because $\cb(C)$ avoids $A'$. That neighbour is a marked vertex of $C'$ by the construction of $G'$. In this case we let $B'_{u_i}:=B_{u_i}$. 


If not, then $B_{e_i}$ lies in $C'$. In this case, let $P_i$ be a path in $G'$ from the vertex $B_{e_i} \cap B_{a_i}$ to $A'$, which exists since $B_{a_i}$ is connected and intersects $A'$ by \Or{apices}. Note that the last edge of $P_i$ joins a vertex $u'_i\in C'$ to a vertex of $A'$, and therefore $u'_i$ is a marked vertex of $C'$. Let $B'_{u_i}:=B_{u_i} \cup B_{e_i} \cup (P_i \sm A')$, and note that $B'_{u_i}$ is connected and contains $u'_i$.


In both cases, $B'_{u_i}$ contains a marked vertex. Easily, $B_{u_i} \cap B_{u_j}=\emptyset$ whenever $i\neq j$ as $P_i \subset B_{a_i}$ and $B_{a_i} \cap B_{a_j}=\emptyset$. Therefore, replacing each $B_{u_i}$  by $B'_{u_i}$, and leaving $B'_x:= B_x$ \fe\ other vertex $x\in V(C)$, we obtain a minor model $\cb'$ of $C$ in $C'$. This proves \eqref{CinCp}. 
 
\medskip
For the backward implication, we assume $\cc \subseteq \cc'$, and construct an embedding of \g into $G'$ as follows. We begin by letting $B_a:= a'$ \fe\ $a\in A$, where $a\mapsto a'$ is an arbitrary but fixed bijection from $A$ to a subset of $A'$, which is possible since $|A|=n\leq m=|A'|$. For each `part' $H_i$ of $G$, find a part $H'$ of $G'$ \st\ $H_i \mm H'$ and a model of $H_i$ in $H'$ \st\ the marked vertices of $H_i$ are joined to $A$ the same way as the marked vertices of $H'$ are joined to $A'$ \wrt\ $a\mapsto a'$. Since there are infinitely many such $H'$ by item \ref{GC iv} of \Dr{def GC}, we can choose them disjointly over the $H_i$'s.
\end{proof}

We are now ready to prove the main result of this section.

\begin{proof}[Proof of \Tr{tfae}]
The equivalence of \ref{R iii} and \ref{R iiim} follows from \Cr{cor unmark}. 
The equivalence of \ref{R iv} and \ref{R ivp} is \Lr{lem unmark}. The implications  \ref{R iii} $\to$ \ref{R iiip} and \ref{R iip} $\to$ \ref{R ip} are trivial. 

For the other implications, we apply induction on $\alpha$. For $\alpha=0$, we define $\ran{<0}$ to be empty, and notice that all items are obviously true.

\medskip

\ref{R iip} $\to$ \ref{R iii}: The proof of this implication is very similar to that of \Tr{thm trees}. 

To prove that $|\ran{\alpha}|_<$ is countable, it suffices to prove that $|\ran{\alpha}^n|_<$ is countable \fe\ \nin, where $\ran{\alpha}^n:= \{G\in \ran{\alpha} \mid |A(G)|=n\}$. Let $G\in \ran{\alpha}^n$, and apply \Lr{subclasses} to deduce that the minor-twin class $[G]_<$ is determined by $n=|A(G)|$ and the class $\ccm(G)$. To be able to apply \Lr{subclasses}, we use the fact that $\ccm(G) \subseteq \ranmn{<\alpha}$ is \wqo\ by our assumption, and that $|\ranm{\beta}|_{\mm}$ is countable \fe\ $\beta<\alpha$ by our inductive hypothesis \ref{R iip} $\to$ \ref{R iiim}---whereby we use the fact that $\ranmn{\beta}$ is \wqo\ since its superset $\ranmn{<\alpha}$ is. 

As $\ccm(G)$ is \wqo, we can express it as $\forb{X} \cap \ranmn{<\alpha}$ for some finite $X \subset \ranmn{<\alpha}$. 
Since $|\ranmn{\beta}|_<$ is countable \fe\ $\beta<\alpha$ as noted above, $|\ranmn{<\alpha}|_<$ is countable too, being the sum of $|\ranmn{\beta}|_<$ over the  $\beta<\alpha$, which are countably many as $\alpha<\omega_1$. As $|\ranmn{<\alpha}|_<$ is countable, there are countably many ways to choose its finite subset $X$ from above, hence countably many ways to choose $\ccm(G)$, and therefore $[G]_<$. This proves that $|\ran{\alpha}|_<$ is countable.
\medskip

\ref{R ip} $\to$ \ref{R iip}: 
It suffices to show that $\ranmn{<\alpha}$ has no  \mm-descending chain by \Prr{wqo ch}, so suppose to the contrary $G_1 \gneq G_2 \gneq  \ldots$ is one. Then letting $\Rank(G_1)=: \beta < \alpha$, and noting that $\Rank(G_i) \leq \Rank(G_1)$ by \Or{minor rank} and $|A(G_i)| \leq |A(G_1)|$ by \Or{apices}, we deduce that this is a descending chain in $\ranmn{\beta}$. This contradicts  the inductive hypothesis \ref{R ip} $\to$ \ref{R iv} for $\beta$.

\medskip

\ref{R iip} $\to$ \ref{R ivp}: Suppose, for a contradiction, that $\ran{\alpha}$ has a descending chain $G_1 \gneq G_2 \gneq G_3 \gneq  \ldots$. By \Or{apices} we may assume that $G_i\in \ran{\alpha}^n$ for a fixed $n\in \N$. Then, by \Lr{subclasses}, 
$\ccm(G_1) \supsetneq \ccm(G_2) \supsetneq \ccm(G_3) \supsetneq  \ldots$ is a descending chain of sub-classes of $\ranmn{<\alpha}$ with respect to containment. For each $k\in \N$, choose a marked graph $H_k\in \ccm(G_k) \sm \ccm(G_{k+1})$, which is possible since  $\ccm(G_k) \supsetneq \ccm(G_{k+1}) $. 
Then $(H_k)$ is a bad sequence in $\ranmn{<\alpha}$\footnote{This idea also appears in \cite[LEMMA~6]{DieRel}.}. Indeed, if $H_k < H_{k+j}$ for some $k,j>0$, then $H_k\in  \ccm(G_{k+j}) \subseteq \ccm(G_{k+1})$ since $\ccm(G_{k+j})$ is marked-minor closed. This contradicts that $\ranmn{<\alpha}$ is \wqo.


\medskip


The following two implications are similar to the forward direction of \Cr{UF ctble}; the reader may want to recall that proof before reading them.
\medskip

\ref{R ivp} $\to$  \ref{R iip}: Suppose, to the contrary,  there is a bad sequence $(H_i)$ in $\ranmn{<\alpha}$. We may assume \obda\ that each $H_i$ is connected by replacing it by $S^\bullet(H_i)$ and applying \Lr{lem cones} (and increasing $n$ by 1). Let $\cc_i:= \ranmn{<\alpha} \cap \forb{H_1,\ldots, H_i}$ for each $i\in\N$. Note that $\cc_1 \supsetneq \cc_2 \supsetneq \ldots$ because $\cc_i$ contains $H_{i+1}$ but $\cc_{i+1}$ does not. Clearly, $\cc_i$ is closed under marked minors. Since the $H_i$ are connected, each $\cc_i$ is addable up to rank $\alpha$ as remarked after \Dr{def upto}. Let $G_i:=G_{\cc_i}$ be as in \Dr{def GC}. By our inductive hypothesis, \ref{R ivp} $\to$ \ref{R iiim} holds \fe\ $\beta<\alpha$, and therefore $\cc_i \subseteq \ranmn{<\alpha}$ consists of countably many minor-twin classes. Thus $G_i$ is countable. 

We claim that 
\labtequ{RCi}{$\Rank(\cc_i)= \alpha$ \fe\ $i\in \N$.} 
To see this, note first that if $\Rank(H_n)\leq \beta$ holds for some $\beta< \alpha$ and infinitely many \nin, then this subsequence of \seq{H}\ is a bad sequence in $\ranmn{\beta}$, contradicting our inductive hypothesis since $\ran{\beta}$ has no descending chain. Thus $\sup_n \Rank(H_n)= \alpha$. Since each $\cc_i$ contains every $H_j,j>i$, we deduce that $\Rank(\cc_i)\geq \alpha$. The converse inequality is obvious since $\cc_i \subset \ranmn{<\alpha}$.

\Lr{lem GC} now implies that $G_n \gneq G_{n+1}$ \fe\ $i$, i.e.\ $\seqi{G}$ is an infinite descending $<$-chain in $\ran{\alpha}$, contradicting our assumption that none exists.

\comment{
	.. Pick a representative from each minor-twin class of $\cc_i$, and let $G'_i$ be the graph consisting of the disjoint union of these representatives. We claim that $G'_i$ is countable. Indeed, by our inductive hypothesis, \ref{R ivp} $\to$ \ref{R iii} holds \fe\ $\beta<\alpha$, and therefore $\cc_i \subseteq \ran{<\alpha}$ consists of countably many minor-twin classes. 

Let $G_i:= \oo \cdot G'_i$ be the disjoint union of countably many copies of $G'_i$.
Note that $\Rank(G_i)\leq \alpha$. We claim that 
\labtequ{AGi}{$A(G_i)=\emptyset$.} 
To see this, note first that if $\Rank(H_n)\leq \beta$ holds for some $\beta< \alpha$ and every \nin, then \seq{H}\ is a bad sequence in $\ran{\beta}$, contradicting our inductive hypothesis. Thus $\sup_n \Rank(H_n)= \alpha$. Since each $\cc_i$, and hence $G_i$, contains (a minor-twin of) every $H_j,j>i$, we deduce that $\Rank(G_i)\geq \alpha$. But as each component of $G_i$ lies in $\ran{<\alpha}$, it follows that $A(G_i)=\emptyset$ and $\Rank(G_i)= \alpha$.

	We claim that $(G_n)_\nin$ is an infinite descending $<$-chain in $\ran{\alpha}$, contradicting our assumption that none exists. To see that $G_{n+1}< G_{n}$ holds, recall that $\cc_{n+1}\subset \cc_{n}$, and therefore each component of $G_{n+1}$ is a minor of a component $C$ of $G_{n}$. As there are infinitely many disjoint copies of $C$ in $G_{n}$, we can thus embed all components of $G_{n+1}$ into $G_{n}$. For the converse, note that $H_{n+1} \in \cc_n \sm \cc_{n+1}$ because no $H_j, j\leq n$ is a minor of $H_{n+1}$. Thus $H_{n+1}< G_{n}$. Since $H_{n+1}$ is connected, if $H_{n+1}< G_{n+1}$, then $H_{n+1}$ is a minor of a component of $G_{n+1}$, which is impossible as  $H_{n+1} \not\in \cc_{n+1}$. Thus $G_{n+1}\not> G_{n}$, completing the proof that $(G_n)_\nin$ is a descending $<$-chain.
} 

\medskip

\ref{R iiip} $\to$  \ref{R ip}: 
Suppose \seq{H}\ is an anti-chain in $\ranmn{<\alpha}$. As above, we may assume that each $H_n$ is connected by \Lr{lem cones}. Call a subset  $X$ of $\ch:= \{H_n \mid \nin\}$ \defi{co-infinite}, if $\ch \sm X$ is infinite. Easily, there are $2^{\aleph_0}$ such $X$, because any subset of the even $H_n$'s is co-infinite. We will follow the lines of the previous implication\mymargin{\tiny make sure this holds} to produce $2^{\aleph_0}$-many graphs $G_X$ 
none of which is a twin of another. Let $\cc_X:= \ranmn{<\alpha} \cap \forb{X}$. Since each $H_n$ is connected, $\cc_X$ is addable up to rank $\alpha$. 

For each co-infinite $X\subset \ch$, let $G_X:=G_{\cc_X}$ be as in \Dr{def GC}. Again, $G_X$ is countable by our inductive hypothesis \ref{R iiip} $\to$ \ref{R iiim} and the fact that $\alpha$ is countable. 
Similarly to \eqref{RCi}, we claim
\labtequ{AGX}{$\Rank(\cc_X)= \alpha$ \fe\ co-infinite $X \subseteq \ch$.} 
Indeed, $\Rank(\cc_X)\leq \alpha$ is obvious, and to confirm $\Rank(\cc_X)\geq \alpha$, we observe that each of the infinitely many graphs in $\ch \sm X$ is a minor of $G_X$ by construction. We claim that for every $\beta< \alpha$ there is a graph $G$ in $\ch \sm X$ with $\Rank(G)\geq \beta$. For if not, then $\ch \sm X$ is an anti-chain in $\ranmn{<\beta}$ contradicting our inductive hypothesis. (This argument is the reason why we are working with co-infinite sets $X$.)

Thus by \Lr{lem GC}, $G_X,G_Y$ are never minor-twins for co-infinite $X\neq Y \subset \ch$, because $\cc_X\neq \cc_Y$ as \ch\ is an anti-chain. 
Thus we have obtained $2^{\aleph_0}$ distinct minor-twin classes, contradicting our assumption \ref{R iiip}. 
\medskip

To summarize, we have obtained the implications 
\begin{align*}
\text{\ref{R iv} $\leftrightarrow$ \ref{R ivp} $\leftrightarrow$}  &\text{\ref{R iip} $\leftrightarrow$ \ref{R ip}, } \\  
&\text{\ref{R iip} $\to$ \ref{R iii}}   \text{ $\to$ \ref{R iiip} $\to$  \ref{R ip}, }\\ 
& \text{\ref{R iiim} $\leftrightarrow$ \ref{R iii},}
\end{align*}
combining which proves all 7 statements to be equivalent.
\end{proof} 

We will use similar ideas to obtain two corollaries that we will also need for our main theorem in the next section:

\begin{corollary} \label{cor LD}
Let $1\leq \alpha<\omega_1$, and suppose 
\labtequ{L}{$\ran{\alpha}$ has no antichain of cardinality \cont.}
Then $\ranmn{<\alpha}$ is \wqo\ \fe\ \nin.
\end{corollary}
\begin{proof}
We refine the proof of the implication \ref{R iiip} $\to$  \ref{R ip} above to show that if $\ranmn{<\alpha}$ has an anti-chain \seq{H} for some \nin, then \eqref{L} is contradicted. Indeed, given co-infinite sets $X,Y\subset \ch:= \{H_n \mid \nin\}$ 
note that $G_X < G_Y$ implies $\cc_X \subseteq \cc_Y$, which in turn implies $Y \subseteq X$. Thus if $\{X_i\}_{i\in \ci}$ is a family of co-infinite sets that are pairwise $\subseteq$-incomparable, then $\{G_{X_i}\}_{i\in \ci}$ is a $<$-antichain. Easily, there is such a family with $|\ci|= \cont$, and so  \eqref{L} implies that $\ranmn{<\alpha}$ has no infinite anti-chain. By the equivalence of items \ref{R ip}, \ref{R iip} of \Tr{tfae}, \eqref{L} implies the stronger statement that $\ranmn{<\alpha}$ is \wqo\ \fe\ \nin.
\end{proof}

\begin{remark} \label{rem L}
It is not clear whether  \eqref{L} can be added as a further equivalent condition in \Tr{tfae}; we have just seen that it implies \ref{R iip}, but the latter only implies the weaker statement that $\ran{<\alpha}$ has no infinite antichain.
\end{remark}

\comment{
	.. For each co-infinite $X\subset \ch$, pick a representative from each minor-twin class of $\cc_X$, and let $G_X$ be the graph consisting of the disjoint union of these representatives. Again, $G_X$ is countable by our inductive hypothesis. Similarly to \eqref{AGi}, we claim
\labtequ{AGX}{$\Rank(G_X)= \alpha$ and $A(G_X)=\emptyset$.} 
Indeed, $\Rank(G_X)\leq \alpha$ is obvious, and to confirm $\Rank(G_X)\geq \alpha$, we observe that each of the infinitely many graphs in $\ch \sm X$ is a minor of $G_X$ by construction. We claim that for every $\beta< \alpha$ there is a graphs $G$ in $\ch \sm X$ with $\Rank(G)\geq \beta$. For if not, then $\ch \sm X$ is an anti-chain in $\ran{<\beta}$ contradicting our inductive hypothesis. (This argument is the reason why we are working with co-infinite sets $X$.)
This implies $\Rank(G_X)= \alpha$, and from this we immediately deduce that $A(G_X)=\emptyset$ since each component of $G_X$ lies in $\ran{<\alpha}$.

	We claim that $G_X,G_Y$ are never minor-twins for  co-infinite $X\neq Y \subset \ch$. To see this, pick $H\in X \sydi Y$, say $H\in X \sm Y$. Then $H\in \cc_Y\sm \cc_X$ since no element of $Y$ is a minor of $H$. Thus $H< G_Y$. Since $H$ is connected, if $H< G_X$, then $H$ is a minor of a component of $G_X$, which is impossible as  $H\not\in \cc_X$. This proves that $G_X,G_Y$ are not minor-twins, and so we have obtained $2^{\aleph_0}$ distinct minor-twin classes, 	contradicting our assumption \ref{R iiip}.
}

Let $\RM$ denote the class of marked, countable, rayless  graphs.

\begin{corollary} \label{R wqo}
If $\ranmn{<\gamma}$ is \wqo\ \fe\ ordinal $1\leq \gamma<\omega_1$ and every \nin, then $\RM$ is \wqo. 
\end{corollary}
\begin{proof}
The proof is similar to the implication \ref{R ivp} $\to$  \ref{R iip} above, but we will have to increase the rank by one. 

Suppose, to the contrary,  there is a bad sequence \seq{H} in \RM. As $\omega_1$ is a regular ordinal, and this sequence is countable, we deduce that \ti\ $\alpha< \omega_1$ \st\ $H_n \in \ranm{<\alpha}$ \fe\ $n$. As usual, we may assume that each $H_i$ is connected by replacing each $H_i$ by $S^\bullet(H_i)$ and applying \Lr{lem cones}. 

Let $\cc_i:= \ranmN{<\alpha}{n_i} \cap \forb{H_1,\ldots, H_i}$ for each $i\in\N$, where $n_i$ is the maximum number of marked vertices in $\{H_1,\ldots, H_i\}$. We can no longer claim that $\cc_1 \supsetneq \cc_2 \supsetneq \ldots$, because graphs in $\cc_j$ can contain more marked vertices than those in $\cc_{j-1}$. To amend this, we introduce $\cc_i^k:= \cc_i \cap \ranmN{<\alpha}{k}$ \fe\ $k\in \N$. We now have $\cc_1^k \supseteq \cc_2^k \supseteq \ldots$, and $\cc_{i}^{i+\ell} \supsetneq \cc_{i+1}^{i+\ell} $ \fe\ fixed $i$ and large enough $\ell$,  because $\cc_i^{i+\ell}$ contains $H_{i+1}$ but $\cc_{i+1}^{i+\ell}$ does not. 

Let $G_{\cc_i^k}$ be as in \Dr{def GC}. By  \Tr{tfae} \ref{R iip} $\to$ \ref{R iiim}  $\cc_i^k \subseteq \ranm{<\alpha}$ consists of countably many minor-twin classes, and so $G_{\cc_i^k}$ is countable. Similarly to \eqref{RCi}, we will prove that $\Rank(\cc_i^k)=  \alpha$ \fe\ $i,k\in \N$. 
Indeed, if $\Rank(H_n)\leq \beta$ holds for some $\beta< \alpha$ and infinitely many \nin, then this  contradicts our inductive hypothesis. Thus $\sup_n \Rank(H_n)= \alpha$. Since each $\cc_i^k, i\in \N$ contains every $H_j,j>i$ as an unmarked graph, we deduce that $\Rank(\cc_i^k)\geq \alpha$. The converse inequality is obvious since $\cc_i^k \subset \ran{<\alpha}$.

Let $G_n:= \dot{\bigcup}_{k\in \N} G_{\cc_n^k}$ \fe\ \nin, and note that  $\Rank(G_n)\leq \alpha +1$. We claim that \seq{G} is a $<$-descending chain. Easily, we have $G_n > G_{n+1}$ by applying the backward direction of \Lr{lem GC} componentwise, since $\cc_n^k \supseteq \cc_{n+1}^k$. One the other hand, if $G_n < G_{n+1}$ holds for some $n$, then \fe\ $k$ \ti\ $k'$ \st\ $G_{\cc_n^k} < G_{\cc_{n+1}^{k'}}$. 
But there is $k$ \leth\ $H_{n+1} \in \cc_n^k$, while $H_{n+1} \not\in \cc_{n+1}^{k'}$ \fe\ $k'$. This contradicts the forward direction of \Lr{lem GC}. 
Thus \seq{G} is a $<$-descending chain in $\Rank_{\alpha+1}$, which contradicts \Tr{tfae} \ref{R iip} $\to$  \ref{R ivp}. 


\end{proof} 


\section{Concluding the proof of \Tr{main Intro}}  \label{sec main prf}

We can now complete the proof of our main \Tr{main Intro} in a more detailed version. Recall that \cgr\ denotes the class of countable rayless  graphs, and \RM\ the marked countable rayless  graphs. Let \RMn\ denote the class of countable rayless  graphs with at most $n$ marked vertices.

\begin{theorem} \label{main}
The following statements are equivalent:
\begin{enumerate}[itemsep=0.0cm,label=(\alph*)]
\item \label{M iir} $\cgr$ is \wqo;
\item \label{M ii} $\RM$ is \wqo;
 \item \label{M iip} $\RMn$ is \wqo\ \fe\ \nin;
  \item \label{M iipp} $\ranmn{<\alpha}$ is \wqo\ \fe\ \nin\ and $\alpha<\omega_1$;
 \item \label{M ivp} $\cgr$ has no descending chain;    
  \item \label{M ivpp} $\ran{\alpha}$ has no descending chain \fe\ $\alpha<\omega_1$;   
 \item \label{M iv} $\RM$ has no descending chain; 
\item \label{M i} $\cgr$ has no infinite antichain; 
\item \label{M ippp} $\cgr$ has no antichain of cardinality \cont;
\item \label{M iii} $|\ran{\alpha}|_<$ is countable \fe\ $\alpha<\omega_1$; 
 \item \label{M iiim} $|\ranm{\alpha}|_{\mm}$ is countable \fe\ $\alpha<\omega_1$; 
 \item \label{M iiip} $|\ran{\alpha}|_< < 2^{\aleph_0}$\fe\ $\alpha<\omega_1$.
\end{enumerate}
\end{theorem}
\begin{proof}
We trivially have \ref{M ii}$\to$ \ref{M iip}  $\to$ \ref{M iipp}, and the implication \ref{M iipp} $\to$ \ref{M ii} is \Cr{R wqo}.  Thus items \ref{M ii}, \ref{M iip},  \ref{M iipp} are equivalent.

We trivially have \ref{M ii}$\to$ \ref{M iir} $\to$ \ref{M ivp}. By the implication \ref{R ivp} $\to$ \ref{R iip} of \Tr{tfae}, \ref{M ivp} implies \ref{M iipp}. Thus \ref{M iir}, \ref{M ivp} are also equivalent to the above items.

We have \ref{M ivpp}$\to$ \ref{M ivp} since rank is monotone decreasing in any descending chain by \Or{minor rank}. Thus \ref{M ivpp} is also equivalent to the above items.

The equivalence between \ref{M ivp} and \ref{M iv} follows easily from \Lr{lem unmark}, by noting again that in any descending chain the rank is monotone decreasing.

We trivially have \ref{M iir} $\to$ \ref{M i} $\to$ \ref{M ippp}. 
\Cr{cor LD} provides  the implication  \ref{M ippp} $\to$  \ref{M iipp}, thus adding \ref{M i}, \ref{M ippp} to be above equivalences.

The equivalences between items \ref{M iii}, \ref{M iiim}, \ref{M iiip}  and \ref{M iipp} follow by applying \Tr{tfae} to every $\alpha<\omega_1$. 
\end{proof}

\section{\cf\ is WQO impies \cgr\ is WQO, and even BQO} \label{sec conv}

Let \cf\ denote the class of finite graphs (ordered by $<$). This section completes the proof of \Tr{fin bqo} (started with the implication \impl{RW}{FB} in \Sr{sec wbqo}) by providing the implication \impl{FB}{RB}: 

\begin{theorem} \label{thm BWQO}
If \cf\ is BQO then so is \cgr. 
\end{theorem}

Given two quasi-orders $(Q,\leq)$ and $(R,\preceq)$, a map $f: Q \to R$ is called an \defi{immersion}, 
if $f(p) \preceq f(q)$ implies $p \leq q$ \fe\ $p,q\in Q$. It is easy to see that immersions preserve WQOs:

\begin{observation} \label{pr im}
Suppose $f: R \to Q$ is an immersion between quasi-orders. If $Q$ is a WQO, then so is $R$. \qed 
\end{observation}

\mymargin{A straightforward consequence of \Tr{lem H wqo} is that immersions also preserve BQOs}

Given a quasi-order $(Q,\leq)$, we obtain a quasi-order $(\pst(Q), \leq_*)$ by restricting \Dr{def qo}, i.e.\ by letting $X \leq_* Y$ for every $X,Y \in \pst(Q)$ \st\ for every $X'\in X$ there exists $Y'\in Y$ with $X'\leq Y'$. Any map $m: (R,\preceq) \to (Q,\leq)$ between quasi-orders naturally extends to power sets: 
\begin{definition} \label{def star}
We define $m^*: (\pst(R), \preceq_*) \to (\pst(Q), \leq_*)$ by\\ $X \mapsto \{ m(X') \mid X' \in X \}$. 
\end{definition}
It is straightforward to check from the definitions that this operation preserves immersions too: 

\begin{observation} \label{obs pst}
If $m: (R,\preceq) \to (Q,\leq)$ is an immersion between quasi-orders, then so is $m^*: (\pst(R), \preceq_*) \to  (\pst(Q), \leq_*)$. \qed
\end{observation}

Our approach for the proof of \Tr{thm BWQO} can be summarized as follows. We will immerse \cgr\ into $H^*_{\omega_1}(\cf)$, which is WQO by \Tr{lem H wqo} and our assumption that \cf\ is BQO. This implies that \cgr\ is \wqo\ by \Or{pr im}.

This immersion starts by defining $f:\ran{\alpha} \to \pst(\ranm{<\alpha})$ by $G\mapsto \ccm(G)$; we will use \Lr{subclasses} to deduce that $f$ is an immersion. Note that $f$ maps rayless graphs to sets of marked rayless graphs of smaller ranks. By recursively composing $f$ with the set operation of \Or{obs pst}, and noting that compositions of immersions are immersions, we immerse \ran{\alpha}\ into $H^*_{\omega_1}(\cf)$. But this proof sketch is incomplete for two reasons. 

Firstly, \Lr{subclasses} requires $\Rank(H)= \Rank(G)$, and this equality needs to be relaxed in order for $f$ to immerse all of $\ran{\alpha}$; we do this with \Lr{subclasses R} below. An interesting aspect is that to be able to satisfy the conditions of \Lr{subclasses} we will need to apply \Tr{tfae}. 

Secondly, the image of $f$ is a set of {\emph marked} graphs, and so we cannot compose it with $f^*$ applied to smaller ranks. We will amend this with an unmarking trick in \Sr{sec u}. 
 
\medskip
Having constructed this immersion $m: \cgr \to \hst(\cf)$, it will then be easy to deduce that \cgr\ is \bqo\ by combining  Observation~\ref{pr im} with Observation~\ref{obs hst} below. To state it, let us extend \Dr{def star} to define a map $m^+: (\hst(Q), \leq_+) \to (\hst(R), \preceq_+)$, where $\leq_+, \preceq_+$ are as in \Dr{def qo}: we apply $X \mapsto \{ m(X') \mid X' \in X \}$ to each $X\in V^*_\alpha(Q)$ by transfinite induction on $\alpha$. Analogously to \Or{obs pst}, we now have the following, which is straightforward to prove by induction on the $R$-Rank: 

\begin{observation} \label{obs hst}
If $m: (R,\preceq) \to (Q,\leq)$ is an immersion between quasi-orders, then so is $m^+: (\hst(R), \preceq_+) \to (\hst(Q), \leq_+)$. \qed
\end{observation}

\subsection{Extending  \Lr{subclasses}} \label{sec ext sub}

We will need the following variant of the backward direction of \Lr{subclasses}, obtained be removing the condition $\Rank(H)= \Rank(G)$, and slightly strengthening the other conditions:

\begin{lemma} \label{subclasses R}
Let $G,H$ be countable graphs with $\Rank(H),  \Rank(G)\leq \alpha< \oo_1$.  Assume  $\ranm{<\alpha}$ is \wqo, and $|\ranm{\beta}|_{\mm}$ is countable \fe\ $\beta<\alpha$. If $\ccm(G)\subseteq \ccm(H)$, then $G<H$.
\end{lemma}
\begin{proof}
By \Prr{inf bet} and \Or{minor rank}, $\ccm(G)\subseteq \ccm(H)$ easily implies $\Rank(G)\leq  \Rank(H)$. If $\Rank(G)=\Rank(H)$, then \Lr{subclasses} applies. 

We may assume that $G,H$ are 2-connected by applying \Lr{lem cones unm}. 

If $\Rank(G)<\Rank(H)=:\beta$, we modify $G,H$ into supergraphs $G',H'$ with $\Rank(G')=\Rank(H')$ as follows. For each $a\in A(G) \cup A(H)$, attach the root $r_\alpha$ of a copy of the minimal tree $T_\alpha$ of rank $\alpha$ as provided by \Or{min tree}. Easily, $\Rank(G')=\Rank(H')$, and $A(G')=A(G)$ and $A(H')=A(H)$. Moreover, we have $\ccm(G')\subseteq \ccm(H')$, because for every $(K,M)\in \ccm(G')$, we can extend any model of $(K\cap G,M)$ in $(H,A(H))$ into a model of $(K,M)$ in $(H',A(H))$ by mapping the $T_\alpha$'s attached to  $M\subseteq A(G)$ to those attached to $A(H)$.

Thus \Lr{subclasses} applies to $G',H'$, and yields a model of $G'$ in $H'$. By \Lr{blocks}, this model yields a model of $G$ in $H$ since the former is 2-connected. 
\end{proof}

\comment{ 
	We first observe that the condition $|A(G)| = |A(H)|$ of \Lr{subclasses} can be dropped. 

\begin{lemma} \label{subclasses A}
\Lr{subclasses} remains true when $|A(G)| \leq |A(H)|$. 
\end{lemma}
\begin{proof}
Let $A:=A(G)$ and $A':=A(H)$. Let $\delta:= |A'| -|A|$, and note that \Lr{subclasses} handles the case $\delta=0$. Note that if $|A| > |A'|$ then we cannot have $G<H$, and so the statement is trivially true in that case. Thus it suffices to extend that lemma to the case $|A| < |A'|$, i.e.\ $\delta>0$. 

As before, the forward direction is obvious (and not needed). We will prove the backward direction by induction on $\delta$.

We may assume \obda\ that both $G,H$ are 2-connected by replacing them by $S(S(G))$ and  $S(S(H))$, and applying \Lr{lem cones unm}: indeed, $\ccm(G)\subseteq \ccm(H)$ implies $\ccm(S(S(G)))\subseteq \ccm(S(S(H)))$ by the forward direction of that lemma, and $S(S(G)) < S(S(H))$ implies $G<H$ by its backward direction applied twice. 

As in the proof of \Lr{subclasses}, we let $G_1 \subset G_{2} \subset \ldots$ be  a sequence of subgraphs of \g as provided by \Lr{lem Gn}, and observe that \ti\ a sequence of marked-minor embeddings $f_n: (G_n,A) \mm (H,A')$.

We call $h\in V(H)$ \defi{crucial} for $x\in V(G)$, if $h$ lies in the branch set $f_n(x)$ for all but finitely many $\nin$. We distinguish three cases: \smallskip

{\bf Case 1:} There is $h\in A'$ which is not crucial for any $x\in V(G)$. \\ \smallskip

In this case we claim that 
\labtequ{ccm Hh}{$\ccm(G) \subseteq \ccm(H-h)$.}
This is easily obtained by a compactness argument: let \seq{x}\ be an enumeration of the vertices of \G. Let $f'_1$ be a member of $\seq{f}$ \st\ $h\not\in f'_1(x_1)$, let $f'_2$ be a later member \st\ $h\not\in f'_2(x_1) \cup f'_2(x_2)$, and continue to define the infinite subsequence $f'$ of $f$. 
\end{proof}

\labtequ{f emb}{$f$ is an immersion.}
	\mymargin{(of $\ran{\alpha}^n$ for now)}
}

\subsection{The unmarking trick} \label{sec u}

We define a map $u: \ranm{\alpha} \to \pst(\ran{\alpha})$, i.e.\ mapping each marked graph to a set of unmarked graphs of the same rank. 

Given $(G,M) \in \RM$ with $\mu:=|M|\geq 2$, and \nin, we define the \defi{depth-$n$ self-amalgamation $u_n(G)$ of \G} as follows. We enumerate the elements of $M$ as $m_1,\ldots,m_\mu$. We let $T_n$ be the finite rooted tree of depth $n$ each non-leaf vertex of which (including the root $r$) has degree $\mu$. For each vertex $v$ of $T_n$, we let $G^v$ be a copy of $(G,M)$, disjoint from all other copies. We label each of the $\mu$ edges of $r$ by a distinct number in $[\mu]$. We then label the remaining edges of $T_n$ so that for each non-leaf vertex $x$, the labels of the edges of $x$ are in bijection with $[\mu]$; this is easy to achieve via a breadth-first process.

We let $\amal{G}{n}=\amal{(G,M)}{n}$ be the (unmarked) graph obtained from $\dot{\bigcup}_{x\in V(T_n)} G^x$ by identifying, for each edge $xy$ of $T_n$ with label $i\in [\mu]$, the copies of $m_i$ in $G^x$ and $G^y$ into a single vertex. By construction, if \g is 2-connected, then 
\labtequ{amal bl}{Each block of $\amal{G}{n}$ is isomorphic to \G.}
Finally,  let $S_G:= \Sb(\Sb((G,M)))$, and let $u((G,M)):= \{ \amal{S_G}{n} \mid \nin \}$, which we consider as an element of $\pst(\ran{\alpha})$ where $\alpha:= \Rank(G)= \Rank(S_G)$. The reason we applied the self-amalgamation to $S_G$ instead of $G$ is that this ensures our above requirement $\mu\geq 2$, so that $u((G,M))$ is well-defined \fe\ $G$. It also ensures that $S_G$ is 2-connected, and so \eqref{amal bl} holds.

\begin{lemma} \label{lem seq}
\Fe\ $0<\alpha< \oo_1$, $u$ is an immersion of 
\ranm{<\alpha} to $\pst(\ran{<\alpha})$.
\end{lemma}

\begin{proof}
Let $(G,M), (H,M')$ be two marked graphs as in the statement. We have to show that if $u((G,M)) \leq u((H,M'))$ then $(G,M) \mm (H,M')$. So suppose the former is the case, which means that 
\labtequ{GnHk}{\fe\ \nin\ \ti\ $k\in \N$ \st\ $\amal{S_G}{n} < \amal{S_H}{k}$.}
This implies in particular that $\Rank(G)\leq \Rank(H)$ by \Or{minor rank}, since  $\Rank(\amal{S_G}{n}) = \Rank(S_G)= \Rank(G)$ by construction.

For a graph $R\in \cgr$, we let \defi{$||R||:=|A(R)|$} if $\Rank(R)>0$, and let $||R||:=|V(R)|$ if $\Rank(R)=0$ (i.e.\ $R$ is finite). Note that 
\labtequ{mR}{If $\Rank(R)=\Rank(R')$ and $R<R'$, then $||R||\leq ||R'||$}
by \Or{apices}.
\medskip

Fix some $n> ||S_H||$, and let $f: \amal{S_G}{n} < \amal{S_H}{k}$ be a \minem\ as provided by \eqref{GnHk}. Let $G':= \amal{S_G}{n}$ and $H':= \amal{S_H}{k}$. Recall that $S_G^r$ denotes the copy of $S_G$ corresponding to the root of $T_n$ in the construction of $\amal{S_G}{n}$. By \Lr{blocks}, there is a block $B$ of $\amal{S_H}{k}$ \st\ each branch set of $S_G^r$ intersects $B$. By \eqref{amal bl}, $B$ is isomorphic to $S_H$. For $v\in V(S_G^r)$, let $K_v$ denote the union of components of $G' - v$ other than $S_G^r - v$, and let $\cls{K_v}:= S_{G'}[\{v\} \cup V(K_v)]$. Note that $||\cls{K_v}||\geq n$, because $||S_G||\geq 2$ by assumption, and therefore each non-root layer of $T_n$ contributes at least one to each component of $K_v$.

For a marked vertex $v$ of  $S_G^r$, let $f'(v):= \bigcup_{w\in V(\cls{K_v})} f(w)$. If $v$ is unmarked we just let $f'(v):= f(v)$. Note that $f'(v)$ spans a connected subgraph of $H'$ because \cls{K_v}\ is connected. Thus we can think of $f'$ as a \minem\ of  $S_G^r$ into $H'$. We claim that 
\labtequ{clm fm}{for each marked vertex $v$ of $S_G^r$, the set $f'(v)$ contains a marked vertex of $B$.}
This claim implies $S_G \mm S_H$, because by the second sentence of \Lr{blocks} we can restrict $f': S_G^r \to H'$ into a marked-\minem\ of $S_G^r$ into $B$. 

To prove \eqref{clm fm}, suppose to the contrary that $f'(v)$ fails to contain a marked vertex of $B$. Since the marked vertices of $B$ separate it from the rest of $H'$, we deduce that $f'(v)=f(\cls{K_v}) \subseteq B$, and therefore $\cls{K_v} < S_H$ since $B \isom S_H$. However, we have $||\cls{K_v}||\geq n > ||H||$, which contradicts \eqref{mR}.

\medskip
Thus we have proved $S_G \mm S_H$. Applying \Lr{lem cones} twice we deduce the desired $(G,M) \mm (H,M')$. 
\end{proof}

\subsection{Completing the proof of \Tr{thm BWQO}} \label{sec com BWQO}

We have now gathered all ingredients to complete the proof we sketched at the beginning of this section: 

\begin{proof}[Proof of \Tr{thm BWQO}]
\Fe\ ordinal $\alpha<\omega_1$ and every \nin\ we will define immersions $m_\alpha: \ran{\alpha} \to H^*_{\omega_1}(\cf)$ and $\mbu_\alpha: \ranm{\alpha} \to H^*_{\omega_1}(\cf)$ by transfinite induction on $\alpha$. It is $m_\alpha$ that we really care about, but we need $\mbu_\alpha$ to be able to perform our inductive step. More precisely, our inductive hypothesis is that such immersions exist, and that both families $\{m_\alpha\}_{\alpha<\omega_1}$ and $\{\mbu_\alpha\}_{\alpha<\omega_1}$ are \defi{compatible}, i.e.\ for every $G\in \ran{\delta}$, and ordinals $\beta >\gamma\geq \delta$, we have $m_\beta(G)=m_\gamma(G)$ and $\mbu_\beta((G,M))=\mbu_\gamma((G,M))$ \fe\ finite $M\subset V(G)$.

\medskip 
To start the induction, for $G\in \Rank_0$ (i.e.\ $G\in \cf$) we just let $m_0(G)=G$. (We will define $\mbu_0$ later.) 

For the induction step, for any ordinal $0<\alpha<\oo_1$, we assume that immersions $m_{\bet}: \ran{\beta} \to H^*_{\omega_1}(\cf)$ and $\mbu_{\bet}: \ranm{\beta} \to H^*_{\omega_1}(\cf)$ have been defined in all previous steps $\beta< \alpha$.  We then define  $m_\alpha: \Rank_\alpha \to H^*_{\omega_1}(\cf)$ by a composition of immersions 
\labtequ{imms}{$\Rank_\alpha \xrightarrow{\ccm} \pst(\ranm{<\alpha}) \xrightarrow{u^*} \pst(\pst(\ran{<\alpha})) \xrightarrow{m^{**}_{<\alpha}} H^*_{\omega_1}(\cf)$}
that we will now define. We let $u$ be the immersion of \Lr{lem seq}, and $u^*$ its extension as in \Or{obs pst}. Thus $u^*$ is an immersion by combining these two results. Similarly, $m^{**}_{<\alpha}$ is obtained by applying \Dr{def star} twice to $m_{<\alpha}$, defined as $m_{<\alpha}:=\bigcup_{\beta<\alpha} m_\beta$. Here  the $m_\beta$ have been defined in previous steps of our induction, and their limit $m_{<\alpha}$ is well-defined by our compatibility assumption. Note that $m_{<\alpha}$ is an immersion since each $m_\beta$ is, and being an immersion is a `local' property. (If $\alpha=\beta +1$ is a successor ordinal, then we can replace $m_{<\alpha}$ by $m_\beta$.)  
 Applying \Or{obs pst} twice to $m_{<\alpha}$ we deduce that $m^{**}_{<\alpha}$ is an immersion too. (The image of $m^{**}_{<\alpha}$ is not all of 
$H^*_{\omega_1}(\cf)$, but a subset of $V^*_\gamma(\cf)$ for some ordinal $\gamma< \oo_1$.)
Finally, we let
\labtequ{imms def}{$m_\alpha(G):=m^{**}_{<\alpha} (u^*(\ccm(G) )) $.}
%

To prove that $m_{\alpha}$ is an immersion, it suffices to check that each of the three maps it is composed of is an immersion, since compositions of immersions are immersions. We have already checked this for $m^{**}_{<\alpha}$ and  $u^*$, and so it remains to check that $\ccm$ is an immersion.

\Lr{subclasses R} says exactly that $\ccm$ is an immersion of $\ran{\alpha}$ into $\pst(\ranm{<\alpha})$, assuming its conditions are satisfied, namely a) $\ranm{<\alpha}$ is \wqo, and b) $|\ranm{\beta}|_{\mm}$ is countable \fe\ $\beta<\alpha$. We will prove that they are satisfied using our induction hypothesis. To confirm a), recall that we have a compatible family of immersions $\mbu_{\bet}: \ranm{\beta} \to H^*_{\omega_1}(\cf), \bet<\alpha$. As above, compatibility implies that the limit $\mbu_{<\alpha}:= \bigcup_{\bet<\alpha} \mbu_{\bet}$ is well-defined, and it is an immersion of $\ranm{<\alpha}$ in $H^*_{\omega_1}(\cf)$. Thus $\ranm{<\alpha}$ is \wqo\ by \Or{pr im}, i.e.\ a) holds. To confirm b), we plug a) into the implication \ref{R iip} $\to$ \ref{R iiim} of \Tr{tfae} and obtain the stronger $|\ranm{\alpha}|_{\mm}= \aleph_0$. This completes the proof that $\ccm$ is an immersion of $\ran{\alpha}$ into $\pst(\ranm{<\alpha})$, and therefore that $m_{\alpha}$ is an immersion.

It remains to ensure that the image of $m_{\alpha}$ is contained in $H^*_{\omega_1}(\cf)$, i.e.\ it consists of hereditarily countable sets. Both $m^{**}_{<\alpha}$ and  $u^*$ preserve countability by definition, but we need to be careful with $\ccm$. The easiest way to handle this issue is to choose a representative from each marked-minor-twin class of $\ranm{<\alpha}$, and replace $\ccm(G)$ by its subset $\cc'(G)$ consisting of the representative of each marked-minor-twin class of $\ccm(G)$. Note that \Lr{subclasses R} is unaffected by replacing $\ccm(G)$ by  $\cc'(G)$. Using (b) again, we deduce that $\cc'(G)$ is countable, and so $m_{\alpha}(G) \in H^*_{\omega_1}(\cf)$ as desired.
\medskip

Having defined $m_{\alpha}$, we define $\mbu_{\alpha}$ by 
\labtequ{mbu}{$(G,M)\mapsto m_{\alpha}^*(u((G,M)))$.}
This is an immersion to  $H^*_{\omega_1}(\cf)$ by \Lr{lem seq} and~\Or{obs pst} applied to $m_{\alpha}$. This formula can be applied to finite $(G,M)$, and we consider this as the definition of $\mbu_0$. 

To see that $m_{\alpha}$ and $\mbu_{\alpha}$ are both compatible, note that their definitions \eqref{imms def} and \eqref{mbu} are independent of the rank of \g (as long as the latter is at most $\alpha$), whereby for the former we use the inductive hypothesis of compatibility. This completes the proof of our inductive hypothesis. 

\medskip
Using compatibility as in the definition of $m_{<\alpha}$ above, we can now define an immersion $m: \cgr \to H^*_{\omega_1}(\cf)$ by  $m:=\bigcup_{\alpha< \oo_1} m_\alpha$, and apply \Lr{pr im} to deduce that \cgr\ is \wqo. (Similarly, we can deduce that \RM\ is \wqo.) 

\medskip
To prove that \cgr\ is \bqo, we plug $m$ into \Or{obs hst}, with $H^*_{\omega_1}(\cf)$ playing the role of $Q$, and obtain an immersion\\ $m^+: \hst(\cgr) \to \hst(\hst(\cf))$. But $\hst(\hst(\cf)) = \hst(\cf)$ by \Or{obs H}, and so $m^+$ immerses $\hst(\cgr)$ into the \wqo\ set $\hst(\cf)$. Thus  $\hst(\cgr)$ is \wqo\ by \Or{pr im}, and so \cgr\ is \bqo\ by \Tr{lem H wqo}.
\end{proof}

The last paragraph of this proof is applicable to arbitrary quasi-orders $\cgr,\cf$, and so we obtain
\begin{corollary} \label{cor gen}
Let $\cgr,\cf$ be quasi-orders. Suppose \cf\ is \bqo, and there is an immersion $m: \cgr \to H^*_{\omega_1}(\cf)$. Then \cgr\ is \bqo.
\end{corollary}

\comment{
\medskip
???From here it is easy to deduce that \cgr\ is in fact \bqo: ???

Let \cgrc\ denote the set of connected graphs in \cgr. 
We will define an immersion $i: \hst(\cgrc) \to \cgr$. Since the latter is \wqo, we deduce that $\hst(\cgrc)$ is \wqo\ by \Or{pr im}, and so \cgrc\ is \bqo\ by \Tr{lem H wqo}. It then follows easily from \Lr{lem Seq} that \cgr\ is \bqo\ too.

We define $i=i_\alpha$ on $V^*_\alpha(\cgrc) \cap \hst(\cgrc)$ by transfinite induction on $\alpha$ as follows. We define $i_0$ to be the identity on \cgrc. For any successor ordinal $0<\alpha<\oo_1$, we define $i_\alpha$ by $X \mapsto S(\dot{\bigcup}_{X'\in X} i_{< \alpha}(X'))$, where $S()$ stands for the suspension as in \Lr{lem cones unm}, and as usual $i_{< \alpha}:= \bigcup_{\bet<\alpha} i_\beta$; this is well-defined unless $X\in \cgrc$, in which case we just let $i_\alpha(X)=i_0(X)=X$. Our inductive hypothesis is that  $i_\alpha$ is an immersion of $V^*_\alpha(\cgrc) \cap \hst(\cgrc)$ to \cgrc. To prove this, we suppose that $i_\alpha(X) < i_\alpha(Y)$ for some $X,Y\in V^*_\alpha(\cgrc) \cap \hst(\cgrc)$, and have to deduce $X\leq Y$ in \hst(\cgrc)\ where $\leq$ is the relation of \Dr{def qo}. By \Lr{lem cones unm} we have $\dot{\bigcup}_{X'\in X} i_{< \alpha}(X') < \dot{\bigcup}_{Y'\in Y} i_{< \alpha}(Y')$, and since the $X'\in X$ are connected, we deduce that \fe\ $X'\in X$ \ti\ $Y'\in Y$ with $X'<Y'$, and therefore $X\leq Y$ by item \ref{qo iv} of \Dr{def qo}.???}

\section{Open problems} \label{sec outl}

Our main open problem, motivated by Thomas' conjecture and \Tr{fin bqo} is
\begin{problem} \label{prob rless}
Are the (countable, or of arbitrary cardinality) rayless graphs  \wqo? 
\end{problem}
Thomas \cite{ThoCou} provides an example of a bad sequence of graphs of the cardinality of the continuum, which however contain plenty of rays. 
\medskip

\comment{
It is not clear to me whether \Prb{prob rless} is decidable in ZFC. 
Thus perhaps the right question to ask is the following:

\begin{problem} \label{prob zfc}
Is the statement that the countable rayless graphs are  \wqo\ consistent with ZFC? 
\end{problem}

A positive answer to \Prb{prob zfc} would still imply a positive answer to \Prb{prob wqo}, because the latter is not independent of ZFC. 
\medskip
}

\comment{
	A well-known problem of R.~Bonnet \cite{PouSauWel} asks whether every \wqo\ poset is a countable union of \bqo\ posets. Our \Tr{main Intro} comes close to corroborating this for the poset of countable rayless graphs. This motivates:
\begin{problem} \label{prob Bon}
Is the class of countable, rayless, graphs a countable union of \bqo\ subclasses?
\end{problem}
}

I find the following special case of \Prb{prob rless} particularly interesting: 

\begin{problem} \label{prob plan}
Are the countable, planar, rayless graphs \wqo?
\end{problem}

I expect that a positive answer to this would imply a positive answer to the following problem, by following the lines of the proof of \Tr{fin bqo} (and \Cr{cor Q}):

\begin{problem} \label{prob fin pl}
Are the finite planar graphs \bqo?
\end{problem}

Recall that Thomas proved that $TW(k)$ is \bqo\ \fe\ $k\in \N$ \cite[(1.7)]{ThoWel}. Let $TW_{<\infty}= \bigcup_{k\in \N} TW(k)$ be the class of countable rayless graphs of finite (but not bounded) tree-width. To appreciate the difficulty of the last two problems, the reader may try the following: 

\begin{problem} \label{prob TW}
Is $TW_{<\infty}$ \wqo?
\end{problem}

See \cite[Conjecture 10.2.]{ExcSigma} for a problem of similar flavour.

Another interesting special case of Thomas' conjecture is whether the countable Cayley graphs are \wqo; see \cite[\S 5]{GeoHamFul} for some progress. 
\medskip

Our next problem is motivated by item \ref{I twin} of \Tr{main Intro}.

\begin{problem} \label{prob contnm}
Is there a family \cc\ of countable ---not necessarily rayless--- graphs with $|\cc|=\cont$ \st\ no two elements of \cc\ are minor-twins?
\end{problem}
A similar question can be asked for countable (rayless or not) marked graphs with no restriction on the number of marked vertices.

\medskip
Our last problems are motivated by \Tr{thm Borel Intro}. 

\begin{problem} \label{prob Bor}
Let $\cc\subseteq \cg$ be a family of $\N$-labelled rayless graphs which is closed under minor-twins. Is it true that \cc\ is \wqo\ \iff\ $\cc \cap \Rank_\alpha$ is Borel \fe\ $\alpha< \oo_1$? 
\end{problem}
(The forward direction is true.)

Similarly, one can ask
\begin{problem} \label{prob Bor fin}
Let $Q$ be a family of finite graphs. Is it true that $Q$ is \bqo\ \iff\ $\Rank^1_\alpha(Q)$ is Borel \fe\ $\alpha< \oo_1$? 
\end{problem}

\medskip
Recall that Robertson, Seymour, \& Thomas \cite{RoSeThoExc} wrote that there is not much chance of proving Thomas' conjecture, and even our restriction to the rayless case seems out of reach. \Tr{main} provides new tools for attacking it, and perhaps there is now a chance of disproving it: 

\begin{problem} \label{prob non Bor}
Is there a family of rayless \N-labelled graphs which is closed under minors, 
has rank less than $\omega_1$, and is not Borel? 
\end{problem}

A positive answer, combined with \Tr{thm min Bor} and the implication \ref{M iir} $\to$ \ref{M iii} of \Tr{main}, would disprove Thomas' conjecture.

\bibliographystyle{plain}
\bibliography{../collective}


\comment{
	Let \defi{$\ran{\alpha}^n$} be the set of those $G\in \ran{\alpha}$ with $|A(G)|=n$.

\begin{lemma} \label{lem ran n}
If $\ran{\alpha}^n$ is \wqo\ \fe\ ordinal $\alpha<\omega_1$ and every \nin, then \cgr\ is \wqo.
\end{lemma}
\begin{proof}
Suppose \seq{G}\ is a bad sequence in \cgr. We may assume that each $G_n$ is connected, because we can replace each $G_n$ with $S(G_n)$ by \Lr{lem cones unm}.

Let $\alpha:= \sup_{\nin} \Rank(G_n)$. \Fe\ $m\in \N$, let $H_m:= \dot{\bigcup}_{n\geq m} G_n$. We claim that \seq{H}\ is a bad sequence. Indeed, if $H_m < H_k$ for some $m<k$, then $G_m$, which is a component of $H_m$, is a minor of a component of $H_k$, i.e.\ some $G_j, j\geq k>m$. But $H_m <G_j$ contradicts that \seq{G}\ is bad. 

Easily, each $H_n$ belongs to $\ran{\alpha}^0$ by construction, proving our statement. (In fact, we have proved that we can replace $\ran{\alpha}^n$ by $\ran{\alpha}^0$.)
\end{proof}

}

\comment{
\mymargin{Not needed?}
\begin{lemma} \label{lem epi}
Let $Q_1,Q_2$ be quasi-orders, and let $g: Q_1 \to Q_2$ be \st\ $g(X)\leq g(Y)$ implies $X\leq Y$ \fe\ $X,Y\in Q$.  If $Q_2$ is BQO then so is $Q_1$.
\end{lemma}

{\tiny 
\begin{definition} \label{def bqo}
\mymargin{Not needed?} A \defi{$Q$-array} is an Ellentuck-continuous/Tychonov-continuous/Borel-measurable function $f: X^{(\omega)} \to Q$ for some $X \in \N^{(\omega)}$, and it is \defi{bad} if there is no $s\in X^{(\omega)}$ \st\ $f(s) \leq f(s \sm \{\min s\})$. The equivalence of these three definitions of $Q$-array is proved by Mathias  (see K\"uhn). Here $Q$ is endowed with the discrete topology.

A quasi-order $Q$ is \defi{\bqo} if there is no bad $Q$-array.
\end{definition}
}}

\comment{
\begin{lemma} \label{lem TW m}
For every $n,t\in \N$, the graph class  $TW(t) \cap \ranmn{}$ is \wqo.
\end{lemma}
\begin{proof}
Let $((G_i,M_i))_{i\in\N}$ be a sequence of marked graphs in $TW(t) \cap \ranmn{}$. We need to show that it is good. We can pick an infinite subsequence along which the number of marked vertices is constant, since this number is upper bounded by $n$. We may assume \obda\ that this constant is $n$, and the subsequence coincides with $((G_i,M_i))_{i\in\N}$.

For each $i$, we modify $(G_i,M_i)$ into an unmarked graph $G'_i$ of tree-width $t+1$ by attaching a copy of $K_{t+2}$ (which has tree-width $t+1$) to each $v\in M_i$ by identifying $v$ with an arbitrary vertex of $K_{t+2}$. We claim that \fe\ $i,j\in \N$, 
\labtequ{cla Gi2}{$G'_i< G'_j$ \iff\ $(G_i,M_i) \mm (G_i,M_i)$.}
This claim implies our statement, because $(G'_i)_{i\in\N}$ is good by Thomas' aforementioned theorem, implying that $((G_i,M_i))_{i\in\N}$ is good too.

The backward implication of \eqref{cla Gi} is trivial (and not needed for our proof). For the forward implication, note that each of $G'_i, G'_j$ has exactly $n$ copies of $K_{t+2}$, which is a 2-connected graph, and these copies are pairwise disjoint. Therefore, if \cb\ is a minor model of $G'_i$ in $G'_j$, then it maps each copy of $K_{t+2}$ in $G'_i$ to one in $G'_j$ by \Lr{blocks}, whereby we use the fact that $G_j$ has no $K_{t+2}$ minor as its tree-width is less than that of $K_{t+2}$. It follows easily from this that for each $v\in M_i$ the branch set $\cb(v)$ contains a vertex of $M_j$. Moreover, if $w\in V(G_i) \sm M_i$,  then   $\cb(w)$ cannot intersect $G'_j \sm G_j$, because all vertices in the latter subgraph are needed to accommodate the copies of $K_{t+2}$ in $G'_i$. Thus by restricting \cb\ to $G_i$ we obtain a marked-minor model of $(G_i,M_i)$ in $(G_i,M_i)$.
\end{proof}
}

\comment{
	\section{On the cardinality of $\ranmn{\alpha}$}

I do not know whether \ref{R iii} $\to$ \ref{R iiim} holds.  Given $G\in \ran{\alpha}$, there are countably many ways to mark finitely many vertices of \g to form a marked graph in $\ranm{\alpha}$. However, $[G]$ may be uncountable. However, we can prove the following weakening:

\begin{proposition}
If $\ran{\alpha}$ is countable then so is $\ranm{<\alpha}$. \label{marked ctble}
\end{proposition}
\begin{proof}
Suppose $\{G_i\}_{i\in \ci}$ is an uncountable family of graphs in $\ranm{<\alpha}$ no two of which are marked-minor-twins. We may assume \obda\ that $G_i\in \ranmn{<\alpha}$ for a fixed \nin. We obtain  unmarked graphs $G'_i\in \ran{\alpha}$ similarly to \Dr{def GC} as follows. We start with infinitely many disjoint copies of $G_i$, add a set $A_i$ of $n$ isolated vertices, and for each copy of $G_i$, and each of its $n$ marked vertices $v$, we add an edge from $v$ to a distinct element of $A_i$. 
As in the forward direction of the proof of \Lr{lem GC}, we deduce that no two of the $G'_i$'s are minor-twins, contradicting that $\ran{\alpha}$ is countable.
	\end{proof}
}

\comment{
	\begin{lemma} \label{unstable}
NOW COPIED INLINE; DELETE\\
Suppose $\ccm(G) \subseteq \ccm(H)$, $|A(G)|= |A(H)|$, and fix a bijection $z: A(G) \to A(H)$ `realised' by a sequence witnessing $\ccm(G) \subseteq \ccm(H)$. 
There is \seq{g}, $g_n: G_n < H$ with $g_n^A=z$, such that every part $P$ of \g that is $g^P$-unstable \wrt\ some such sequence \seq{g^P} satisfying $g_n^{P,A}=z$ \fe\ $n$ is also $g$-unstable. 
\end{lemma}
\begin{proof}
NOW COPIED INLINE; DELETE\\
Enumerate those $P$'s as \seq{P}, and form $g_n$ by picking infinitely many members from each $g^{P_i}$, assuming \obda\ that $g^{P_i}_n(P\sm A(G))$ lie in distinct parts of $H$ for different values of $n$.
\end{proof}
}

\comment{

NOT NEEDED; PARTLY COPY OF PREVIOUS ITEM: \ref{R ivp} $\to$  \ref{R iipp}: 
Suppose, to the contrary,  there is a bad sequence $(H_i)$ in $\ran{<\alpha}$. We may assume \obda\ that each $H_i$ is connected by replacing each $H_i$ by $S^\bullet(H_i)$ and applying \Lr{lem cones un}. Let $\cc_i:= \ran{<\alpha} \cap \forb{H_1,\ldots, H_i}$ for each $i\in\N$. Note that $\cc_1 \supsetneq \cc_2 \supsetneq \ldots$ because $\cc_i$ contains $H_{i+1}$ but $\cc_{i+1}$ does not. 

Note that $\cc_i$ is closed under minors. 
Pick a representative from each minor-twin class of $\cc_i$, and let $G'_i$ be the graph consisting of the disjoint union of these representatives. We claim that $G'_i$ is countable. Indeed, by our inductive hypothesis, \ref{R ivp} $\to$ \ref{R iii} holds \fe\ $\beta<\alpha$, and therefore $\cc_i \subseteq \ran{<\alpha}$ consists of countably many minor-twin classes. 

Let $G_i:= \oo \cdot G'_i$ be the disjoint union of countably many copies of $G'_i$.
Note that $\Rank(G_i)\leq \alpha$. We claim that 
\labtequ{AGi}{$A(G_i)=\emptyset$.} 
To see this, note first that if $\Rank(H_n)\leq \beta$ holds for some $\beta< \alpha$ and every \nin, then \seq{H}\ is a bad sequence in $\ran{\beta}$, contradicting our inductive hypothesis. Thus $\sup_n \Rank(H_n)= \alpha$. Since each $\cc_i$, and hence $G_i$, contains (a minor-twin of) every $H_j,j>i$, we deduce that $\Rank(G_i)\geq \alpha$. But as each component of $G_i$ lies in $\ran{<\alpha}$, it follows that $A(G_i)=\emptyset$ and $\Rank(G_i)= \alpha$.


We claim that $(G_n)_\nin$ is an infinite descending $<$-chain in $\ran{\alpha}$, contradicting our assumption that none exists. To see that $G_{n+1}< G_{n}$ holds, recall that $\cc_{n+1}\subset \cc_{n}$, and therefore each component of $G_{n+1}$ is a minor of a component $C$ of $G_{n}$. As there are infinitely many disjoint copies of $C$ in $G_{n}$, we can thus embed all components of $G_{n+1}$ into $G_{n}$. For the converse, note that $H_{n+1} \in \cc_n \sm \cc_{n+1}$ because no $H_j, j\leq n$ is a minor of $H_{n+1}$. Thus $H_{n+1}< G_{n}$. Since $H_{n+1}$ is connected, if $H_{n+1}< G_{n+1}$, then $H_{n+1}$ is a minor of a component of $G_{n+1}$, which is impossible as  $H_{n+1} \not\in \cc_{n+1}$. Thus $G_{n+1}\not> G_{n}$, completing the proof that $(G_n)_\nin$ is a descending $<$-chain.

\medskip
NOT NEEDED; PARTLY COPY OF PREVIOUS ITEM: \ref{R iiip} $\to$  \ref{R iipp}: If not, then $\ran{<\alpha}$ contains a descending chain or an anti-chain. Suppose  first \seq{H}\ is a descending chain in $\ran{<\alpha}$. Then $\Rank(H_n)$ is decreasing, hence upper bounded by some $\beta< \alpha$. But this contradicts our inductive hypothesis because $|\ran{\beta}|_< \leq |\ran{\alpha}|_< < 2^{\aleph_0}$. 

Suppose next \seq{H}\ is an anti-chain in $\ran{<\alpha}$. Again, by \Lr{lem cones un}, we may assume that each $H_n$ is connected. Call a subset  $X$ of $\ch:= \{H_n \mid \nin\}$ \defi{co-infinite}, if $\ch \sm X$ is infinite. Easily, there are $2^{\aleph_0}$ such $X$, because any subset of the even $H_n$'s forms such an $X$. Let $\cc_X:= \ran{<\alpha} \cap \forb{X}$. We will follow the lines of the previous implication\mymargin{make sure this holds} to produce $2^{\aleph_0}$-many graphs $G_X$ 
none of which is a twin of another. 

%
%
For each co-infinite $X\subset \ch$, pick a representative from each minor-twin class of $\cc_X$, and let $G_X$ be the graph consisting of the disjoint union of these representatives. Again, $G_X$ is countable by our inductive hypothesis. Similarly to \eqref{AGi}, we claim
\labtequ{AGX}{$\Rank(G_X)= \alpha$ and $A(G_X)=\emptyset$.} 
Indeed, $\Rank(G_X)\leq \alpha$ is obvious, and to confirm $\Rank(G_X)\geq \alpha$, we observe that each of the infinitely many graphs in $\ch \sm X$ is a minor of $G_X$ by construction. We claim that for every $\beta< \alpha$ there is a graphs $G$ in $\ch \sm X$ with $\Rank(G)\geq \beta$. For if not, then $\ch \sm X$ is an anti-chain in $\ran{<\beta}$ contradicting our inductive hypothesis. (This argument is the reason why we are working with co-infinite sets $X$.)
This implies $\Rank(G_X)= \alpha$, and from this we immediately deduce that $A(G_X)=\emptyset$ since each component of $G_X$ lies in $\ran{<\alpha}$.

We claim that $G_X,G_Y$ are never minor-twins for  co-infinite $X\neq Y \subset \ch$. To see this, pick $H\in X \sydi Y$, say $H\in X \sm Y$. Then $H\in \cc_Y\sm \cc_X$ since no element of $Y$ is a minor of $H$. Thus $H< G_Y$. Since $H$ is connected, if $H< G_X$, then $H$ is a minor of a component of $G_X$, which is impossible as  $H\not\in \cc_X$. This proves that $G_X,G_Y$ are not minor-twins, and so we have obtained $2^{\aleph_0}$ distinct minor-twin classes, contradicting our assumption \ref{R iiip}.
}

\comment{
The following refines \Lr{subclasses}; perhaps delete the former.

\mymargin{somewhere: The following allows us to bypass the part of the proof of \Lr{UF} that uses compactness to embed $G_S$, and therefore the finiteness of the parts}

\begin{lemma} \label{subclasses2}
Assume $\ranmn{<\alpha}$ is \wqo\ \fe\ \nin. Let $G,H\in \ran{\alpha}$ and suppose $\Rank(G)\leq \Rank(H)$, and $|A(G)|= |A(H)|$. Then the following are equivalent: 
\begin{enumerate}
\item \label{s i} $(G,A(G)) <(H,A(H))$; and
\item \label{s ii} $\ccm(G)\subseteq \ccm(H)$. 
\item \label{s iii}  ?? $\ccmf(G)\subseteq \ccmf(H)$, where $\ccmf(G)$ is defined by considering finite unions of parts.
\end{enumerate}
\end{lemma}
\begin{proof}
The implication \ref{s i} $\to$  \ref{s ii} is 
trivial. 
We will prove the implication \ref{s ii} $\to$  \ref{s i} by induction on $\Rank(G)$. Let \seq{P}\ be an enumeration of the parts of \G, and let $G_n:= \bigcup_{i\leq n} P_i$. Since $|A(G)|=|A(H)|$, every minor embedding $h: G_n < H$ induces a bijection of $A(G)$ onto $A(H)$, which bijection we denote by $h^A$.
Instead of the implication \ref{s ii} $\to$  \ref{s i} we will inductively prove the following strengthening:
\labtequ{fA}{\Fe\ sequence \seq{f}\ of minor embeddings $f_n: G_n < H$ witnessing \ref{s ii}, \ti\ $g:G<H$ \st\ $g^A=f_n^A$ for infinitely many \nin.}
Recall that $\Rank(G)=0$ means that \g is finite, and that we have defined $A(G):= V(G)$ in that case. Thus the base case $\Rank(G)=0$ of \eqref{fA} is obvious\mymargin{For finite \g define $A(G):= V(G)$ in prels}, so let us assume from now on $\Rank(G)\geq 1$.

Every co-part of \g is mapped to a co-part of $H$ by any marked minor embedding, and so we can extend the definition of stable.. 

We may assume \obda\ that $f_n^A$ is a constant bijection $z$, because we can achieve this by passing to a subsequence. Applying \Lr{unstable} \mymargin{dissolve it and embed here} we obtain a sequence $g'_n: G_n < H, \nin$ of minor embeddings \st\ $g_n^{'A}=z$ and every part $P$ of \g that is $p$-unstable \wrt\ some such sequence \seq{p}\ with $p_n^{A}=z$ is also $g'$-unstable. 
 
Let $\cs$ be the set of $g'$-stable parts of $G$, and $\cu$ the set of all other parts of \G. Let $\cs'$ be the set of $H$-stable parts of $H$, and $\cu'$ the set of all other parts of $H$. By \Lr{lem redundant}, 
 \labtequ{csfin}{$\cs'$ is finite.}
(This is why we need the assumption that  $\ranmn{<\alpha}$ is \wqo, which assumption we need to drop in some applications of \Lr{subclasses} with $A(G)=\emptyset$.)
Suppose there is some $P\in \cs$ \st\ $g'_n(P^c)$ is contained in an element $U_n$ of $U'$ for infinitely many values of $n$. Since $P$ is $g'$-stable, $\bigcup_n g'_n(P^c)$ meets only finitely many parts of $H$, and so we may assume that these $U_n$ coincide with a fixed $U\in U'$. Let $\seq{h}, h_n: H < H$ be a sequence of \minem s witnessing that $U$ is $H$-unstable, i.e.\ embedding $U$ into infinitely many distinct parts of $H$, and such that $h_n^{A(H)}$ is the identity on $A(H)$. Then the sequence of compositions $h_n \circ  g'_n$ embed $P^c$ into infinitely many distinct parts of $H$, and so $P$ is $(h \circ  g')$-unstable. Here, we used the fact that $(h_n \circ  g'_n)^A=z$ by construction.  But this contradicts our choice of $g'$ since $P$ is $g'$-stable.

This contradiction proves that $g'_n(P^c)\subset \cs'$ for almost all $n$ \fe\ $P\in \cs$. We modify $(g'_n)$ into the desired sequence $(g_n)$ as follows. For every $n$, and every $P\in \cs$ \st\ $g'_n(P^c)\not\subset \bigcup \cs'$, we omit $P^c$ from the domain of definition of $g'_n$ to obtain $g''_n$. Note that each $P$ is omitted for at most finitely many $n$ by the previous statement. Then, we let $(g_n)$ be a subsequence of $(g''_n)$, chosen so that $g_n$ does not omit any of the first $n$ members of our fixed enumeration \seq{P}\ of the parts of \G, which is possible since each $P_i$ is omitted finitely often. By construction, we have $g_n(P^c)\subset \bigcup \cs'$ \fe\  $P\in \cs$ and every $\nin$ \st\ $g_n(P^c)$ is defined. In particular, every $g'$-stable part is also $g$-stable. By the choice of $g'$, the converse also holds, i.e.\ every $g$-stable part is $g'$-stable. By construction, $(g_n)$ retains the property of $(g'_n)$ that each member induces the same bijection $z$ of  $A(G)$ onto $A(H)$.

\medskip

Let $G_S:= \bigcup \cs$. We claim that
\labtequ{RGS p}{$\Rank(G_S)< \Rank(G)$,}
and to prove this we will show that a graph with the same rank as $G_S$ is a minor of a graph with the same rank as $\bigcup \cs'$. 

For this, let $H'$ be the graph obtained from $\bigcup \cs'$ by adding all possible edges with one end-vertex in $A(H)$ and one end-vertex in $\bigcup \cs'\sm A(H)$. By \eqref{csfin}, $\Rank(H')< \Rank(H)$. 

Let $A':= \bigcup_{Q\in \cs'} A(Q^c)$, and note that $A'$ is finite since $\cs'$ is. Thus \fe\ $n$, $g_n(P)$ intersects $A'$ for a bounded number of parts $P$ of $G$ ..Moreover, \seq{g}\ proves that 
\medskip

Note that \ref{RGS}\ implies $\Rank(G_S)< \Rank(H)$, and so combining \ref{RGS}\ with our inductive hypothesis \eqref{fA}\ applied to the sequence $\seq{g}$ restricted to $G_S$, we deduce that there is $g_S: G_S < H$ \st\ $g_S^A=g_n^A=z$ for infinitely many $n$, and in fact for all $n$ since $g_n^A=z$ \fe\ $n$.

From now on we can proceed as in the proof of \Lr{UF} (and \Lr{Rank 1} if it is kept) to extend $g_S$ to $g: G<H$ with $g^A=z$ using the HH principle. This completes the proof of \eqref{fA}.
\end{proof}
}

\comment{
The following allows us to bypass the part of the proof of \Lr{UF} that uses compactness to embed $G_S$, and therefore the finiteness of the parts:

\begin{lemma} \label{lem GS}
Suppose $\ccm(G) \subseteq \ccm(H)$ and $\Rank(H)\leq \Rank(G)\leq \alpha$ and $|A(G)|=|A(H)|$. Assume $\ranm{<\alpha}$ is \wqo. Then we can choose \seq{g},  $g_n: G_n < H$, \st\ $G_S^g< H$, and each $g_n$ induces the same bijection of  $A(G)$ onto $A(H)$. 
\end{lemma}
\begin{proof}
Since $|A(G)|=|A(H)|$, every minor embedding $m_n: G_n < H$ induces a bijection of $A(G)$ onto $A(H)$. 
Moreover, every co-part of \g is mapped to a co-part of $H$ by any marked minor embedding, and so we can extend the definition of stable.. 

We begin with a sequence
 $g'_n: G_n < H, \nin$ of minor embeddings as in \Lr{unstable}, i.e.\ \st\ every part $P$ of \g that is $g^P$-unstable \wrt\ some such sequence \seq{g^P} is also $g'$-unstable. By passing to a subsequence, we may assume that each $g_n$ induces the same bijection of  $A(G)$ onto $A(H)$. 
 
Let $\cs$ be the set of $g'$-stable parts of $G$, and $\cu$ the set of all other parts of \G. Let $\cs'$ be the set of $H$-stable parts of $H$, and $\cu'$ the set of all other parts of $H$. By \Lr{lem redundant}, 
 \labtequ{csfin}{$\cs'$ is finite.}
Suppose there is some $P\in \cs$ \st\ $g'_n(P^c)$ is contained in an element $U_n$ of $U'$ for infinitely many values of $n$. Since $P$ is $g'$-stable, $\bigcup_n g'_n(P^c)$ meets only finitely many parts of $H$, and so we may assume that these $U_n$ coincide with a fixed $U\in U'$. Let $\seq{h}, h_n: H < H$ be a sequence of \minem s witnessing that $U$ is $H$-unstable, i.e.\ embedding $U$ into infinitely many distinct parts of $H$. Then the sequence of compositions $h_n \circ  g'_n$ embed $P^c$ into infinitely many distinct parts of $H$, and so $P$ is $h' \circ  g$-unstable. But this contradicts our choice of $g'$ since $P$ is $g$-stable.

This contradiction proves that $g'_n(P^c)\subset \cs'$ for almost all $n$ \fe\ $P\in \cs$. We modify $(g'_n)$ into the desired sequence $(g_n)$ as follows. For every $n$, and every $P\in \cs$ \st\ $g'_n(P^c)\not\subset \bigcup \cs'$, we omit $P^c$ from the domain of definition of $g'_n$ to obtain $g''_n$. Note that each $P$ is omitted for at most finitely many $n$ by the previous statement. Then, we let $(g_n)$ be a subsequence of $(g''_n)$, chosen so that $g_n$ does not omit any of the first $n$ members of our fixed enumeration \seq{P}\ of the parts of \G, which is possible since each $P_i$ is omitted finitely often. By construction, we have $g_n(P^c)\subset \bigcup \cs'$ \fe\  $P\in \cs$ and every $\nin$ \st\ $g_n(P^c)$ is defined. In particular, every $g'$-stable part is also $g$-stable. By the choice of $g'$, the converse also holds, i.e.\ every $g$-stable part is $g'$-stable. By construction, $(g_n)$ retains the property of $(g'_n)$ that each member induces the same bijection of  $A(G)$ onto $A(H)$.

Let $H'$ be the graph obtained from $\bigcup \cs'$ by adding all possible edges with one end-vertex in $A(H)$ and one end-vertex in $\bigcup \cs'\sm A(H)$. By \eqref{csfin}, $\Rank(H')< \Rank(H)$. 

Let $G_S:= \bigcup \cs$. We claim that
\labtequ{RGS}{$\Rank(G_S)< \Rank(G)$,}
and to prove this we will show that a graph with the same rank as $G_S$ is a minor of a graph with the same rank as $\bigcup \cs'$. 
For this, let $A':= \bigcup_{Q\in \cs'} A(Q^c)$, and note that $A'$ is finite since $\cs'$ is. Thus \fe\ $n$, $g_n(P)$ intersects $A'$ for a bounded number of parts $P$ of $G$ ..Moreover, \seq{g}\ proves that 

By \eqref{RGS}, the marked graph $(G_S,A(G))$ lies in $\ccm(G)$, and therefore it has a twin in $\ccm(H)$, which proves that $(G_S,A(G))< (H,A(H))$. 


\end{proof}
}

\comment{
\begin{proof}[Proof of \Lr{subclasses}]
{\small We will follow the approach of \Lr{Rank 1}, and it is assumed that the reader is familiar with its proof. The main technical difficulty in comparison to that lemma will be handling the stable parts. 
\medskip

Again, the forward implication is trivial. 

For the backward implication, let \seq{P}\ be an enumeration of the parts of \G.
Let \seq{G}\ be a sequence of subgraphs of \g as provided by \Lr{lem Gn}. We may assume that $P_n \subseteq G_n$, because we can add $\bigcup_{i\leq n} P_i$ to $G_n$ without increasing its rank. 

Let $A:=A(G)$ and $A':=A(H)$. Since $|A|=|A'|$, every marked-minor embedding $f: (G_n,A) \mm (H,A')$ induces a bijection $f^A$ of $A$ onto $A'$ as in \Dr{def hA}.
%
%
Since $\ccm(G)\subseteq \ccm(H)$, \ti\ a sequence of marked-minor embeddings $f_n: (G_n,A) \mm (H,A')$. We may assume \obda\ that $f_n^A$ is a constant bijection $z$, because we can achieve this by passing to a subsequence. We call any such sequence \seq{f}\ a \defi{($G_n,z$)-sequence}.

Easily, every co-part of \g is mapped to a co-part of $H$ by any marked minor embedding, and so we can adapt the definition of an $h$-stable part from \Lr{UF} to any sequence \seq{g}\ of embeddings $g_n: (G_n,A) \mm (H,A')$: we call a part $P$ of \g \defi{$g$-stable}, if $\bigcup_n g_n(P)$ is contained in the union of a finite set of parts of $H$, and call $P$ \defi{$g$-unstable} otherwise.

We claim that there is ($G_n,z$)-sequence \seq{g'}\ maximizing the set of unstable parts: 
 \labtequ{unstab}{There is a ($G_n,z$)-sequence \seq{g'}, $g'_n: G_n \to H$, such that every part $P$ of \g that is $g^P$-unstable \wrt\ some ($G_n,z$)-sequence \seq{g^P} is also $g'$-unstable.
}
To see this, enumerate those parts $P$ as \seqi{P'}, and form $g'_n$ by picking infinitely many members from each $g^{P'_i}$, assuming \obda\ that $g^{P'_i}_n(P'_i\sm A)$ lie in distinct parts of $H$ for different values of $n$. This proves \eqref{unstab}.

 \medskip
Let $\cs$ be the set of $g'$-stable parts of $G$, and $\cu$ the set of all other parts of \G. Let $\cs'$ be the set of $H$-stable parts of $H$, and $\cu'$ the set of all other parts of $H$. Let $H_S:= \bigcup \cs'$. By \Lr{lem redundant}, 
 \labtequ{csfin}{$\cs'$ is finite, and therefore $\Rank(H_S)<\Rank(H)$.}
 %

Next, we claim that
\labtequ{gn cs}{$g'_n(P^c)\subset H_S$ holds for almost all $n$ \fe\ $P\in \cs$,}
where $P^c:= P \sm A$ is the co-part of $P$.
To prove this, suppose to the contrary there is some $P\in \cs$ \st\ $g'_n(P^c)$ is contained in an element $U_n$ of $\cu'$ for infinitely many values of $n$. Since $P$ is $g'$-stable, $\bigcup_n g'_n(P^c)$ meets only finitely many parts of $H$, and so we may assume that these $U_n$ coincide with a fixed $U\in U'$. Let $\seq{h}, h_n: H < H$ be a sequence of \minem s witnessing that $U$ is $H$-unstable, i.e.\ embedding $U$ into infinitely many distinct parts of $H$, and such that $h_n^{A}$ is the identity on $A'$. Then the sequence of compositions $h_n \circ  g'_n$ embed $P^c$ into infinitely many distinct parts of $H$, and so $P$ is $(h \circ  g')$-unstable. Here, we used the fact that $(h_n \circ  g'_n)^A=z$ by construction.  But this contradicts our choice of $g'$ because of \eqref{unstab}, since $P$ is $g'$-stable and $(h \circ  g')$ is a ($G_n,z$)-sequence. This contradiction proves \eqref{gn cs}.

\medskip
Let $G_S:= \bigcup \cs$. We claim that
\labtequ{RGS}{$\Rank(G_S)< \Rank(H) (= \Rank(G))$.}
For this, let $S_1, S_2, \ldots$ be an enumeration of \cs, and let $G'_n:= S_1 \cup \ldots \cup S_n$. We consider $G'_n$ as a marked graph, with $A$ being the set of marked vertices. Using $(g'_n)$ it is not hard to obtain a sequence of marked-minor embeddings $g_n: G'_n \mm H_S$. Indeed, 
for every $n$, and every $P\in \cs$ \st\ $g'_n(P^c)\not\subset H_S$, we omit $P^c$ from the domain of definition of $g'_n$ to obtain a minor embedding $g''_n$ from a subgraph of $G'_n$ to $H_S$. Note that each $P$ is omitted for at most finitely many $n$ by \eqref{gn cs}. Then, we let $(g_n)$ be a subsequence of $(g''_n)$, chosen so that $g_n$ does not omit any of $S_1, \ldots S_n$, which is possible since each $S_i$ is omitted finitely often. 
 By construction, $(g_n)$ retains the property of $(g'_n)$ that each member induces the same bijection $z$ of  $A$ onto $A'$, and $g_n$ embeds $G'_n$ into $H_S$. 

We will now deduce \eqref{RGS} using \eqref{csfin} and the sequence $(g_n)$. Let $\beta:= \Rank(H_S)$. Let $A'':= A(H) \cup \bigcup_{Q\in \cs'} A(Q)$  (\fig{figQs}), and note that $A''$ is finite by \eqref{csfin}, and that 
\labtequ{lt beta}{each component of $H_S - A''$ has rank smaller than $\beta$}
by the definition of $A(Q)$. 

\begin{figure} 
\begin{center}
\begin{overpic}[width=.6\linewidth]{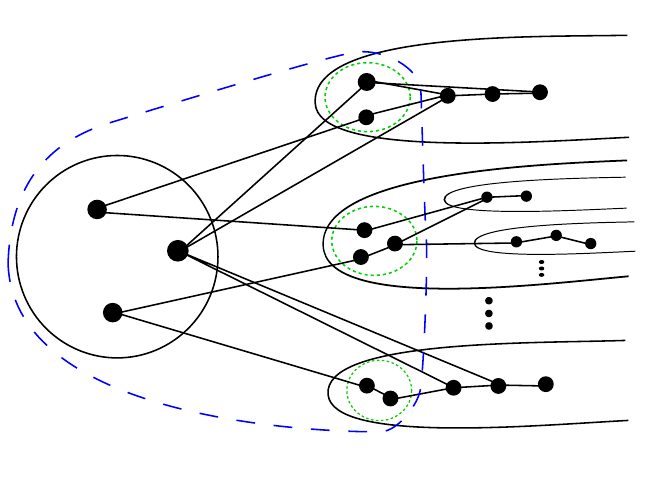} 
\put(88,59){$Q_1$}
\put(88,13){$Q_k$}
\put(34.5,57){\textcolor{green}{$A(Q_1)$}}
\put(36.4,13){\textcolor{green}{$A(Q_k)$}}
\put(12,33){$A(H)$}
\put(25,60){\textcolor{blue}{$A''$}}
\put(97,36){$C$}
\end{overpic}
\end{center}
\caption{The vertex set $A'':= A(H) \cup \bigcup_{Q\in \cs'} A(Q)$ in the proof of \eqref{lt beta}, enclosed by the dashed curve (blue).} \label{figQs}
\end{figure}

Suppose $\Rank(G_S)$ equals $\Rank(G)$, which is larger than $\beta$ by \eqref{csfin}. Then there are infinitely many $P\in \cs$ with $\Rank(P^c)\geq \beta$ by the definition of rank. Let $n$ be \leth\ $G'_n$ contains more than $|A''|$ such $P$'s. Since the $P^c$'s are pairwise disjoint, by the pigeonhole principle, $g_n$ maps at least one of them to $H_S - A''$, and in fact to a component $C$ of $H_S - A''$ since each $P^c$ is connected. This contradicts \Or{minor rank}, since $\Rank(P^c)\geq \beta > \Rank(C)$ by \eqref{lt beta}.


\medskip


This contradiction proves \eqref{RGS}. Thus $G_S \in \ccg$, and so  $G_S< G_n$ for some $n$, and in particular there is a minor embedding $g_S: G_S < H$, obtained by restricting some $g_n$. (We do not claim that $g_S(G_S)$ is contained in $H_S$.)

From now on we can proceed as in the proof of \Lr{Rank 1} to extend $g_S$ to $g: G<H$ with $g^A=z$ by applying the Hilbert Hotel Principle to the $g'$-unstable parts of \G. The only difference is that instead of appealing to the \GMT, we now use our assumption that $\ccm(H)$ is \wqo---combined with \Or{good seqs}--- in order to find a chain \seq{Y}\ within any sequence of parts of $H$. }
\end{proof}
}

\end{document}